\documentclass[10pt]{ijnam}
\usepackage[parfill]{parskip}    
\usepackage{graphicx}
\usepackage{amssymb}
\usepackage{epstopdf}
\usepackage{psfrag}
\usepackage{subfigure}
\newcommand{\figref}[1]{{Figure~\ref{#1}}}
\usepackage{amsmath}
\usepackage{amsfonts}
\usepackage{geometry}   
\usepackage{bbm}
\usepackage{multicol}
\usepackage{amsthm}
\def\currentvolume{1}
\def\currentissue{1}
\def\currentyear{2016}
\def\month{March}
\pagespan{1}{18}
\copyrightinfo{2016}{} 

\newtheorem{theorem}{Theorem}[section]
\newtheorem{Assumption}{Assumption}[section]
\newtheorem{lemma}{Lemma}[section]
\newtheorem{remark}{Remark}[section]

\newcommand{\thmref}[1]{{Theorem~\ref{#1}}}
\newcommand{\lemref}[1]{{Lemma~\ref{#1}}}
\newcommand{\secref}[1]{{Section~\ref{#1}}}
\newcommand{\assref}[1]{{Assumption~\ref{#1}}}

\newcommand{\rmref}[1]{{Remark~\ref{#1}}}

\begin{document}

\title[Strong convergence of the semi-tamed  and  tamed Euler schemes]
{Strong convergence of the  semi-tamed  and  tamed Euler schemes  for stochastic differential
 equations with jumps under non-global Lipschitz condition}

\author[ A. Tambue]{Antoine Tambue}
\address{
Department of Computing Mathematics and Physics,  Western Norway University of Applied Sciences, Inndalsveien 28, 5063 Bergen.
}

\address{
Center for Research in Computational and Applied Mechanics (CERECAM), and Department of Mathematics
and Applied Mathematics, University of Cape Town, 7701 Rondebosch, South Africa.
}
\address{
  The African Institute for Mathematical Sciences(AIMS) of South Africa,
6-8 Melrose Road, Muizenberg 7945, South Africa.
}
\email{antoine.tambue@hvl.no, antonio@aims.ac.za}

\author[ J. D. Mukam ]{Jean Daniel Mukam}
\address{
  African Institute for Mathematical Sciences(AIMS) of Senegal, Km 2, Route de Joal, B.P. 1418, Mbour, Senegal.
}
\address{
Technische Universit\"{a}t Chemnitz, 09126 Chemnitz, Germany.
}
\email{jean.d.mukam@aims-senegal.org}






\abstract{
We consider the explicit numerical approximations of stochastic differential equations (SDEs) driven by Brownian process and Poisson jump.
 It is well known that under non-global Lipschitz condition, Euler Explicit method fails to converge strongly to the exact solution of such SDEs without jumps, 
 while implicit Euler method converges but requires much computational efforts. 
 We investigate the strong convergence, the linear and nonlinear exponential stabilities
 of   tamed Euler  and semi-tamed methods for stochastic differential equation driven by Brownian process and Poisson jumps, both in compensated and non compensated forms. 
 We prove that under non-global Lipschitz condition  and superlinearly growing  drift term,
 these schemes converge strongly with the standard one-half order. 
 Numerical simulations to substain the theoretical results are provided. }


\maketitle
\section{Introduction}
\label{intro}
In this work, we consider jump-diffusion It\^{o}'s stochastic differential equations (SDEs) of the form  in  the  interval $[0, T]$
\begin{eqnarray}
\label{model}
 dX(t)=  f(X(t^{-}))dt +g(X(t^{-}))dW(t)+h(X(t^{-}))dN(t), \quad  X(0)=X_0.
\end{eqnarray}
Here $W(t)$ is a $m$-dimensional Brownian motion, $f :\mathbb{R}^d\longrightarrow\mathbb{R}^d$, $d \in \mathbb{N}$ 
satisfies the  one-sided Lipschitz condition and the polynomial growth condition,
the  functions $g : \mathbb{R}^d \longrightarrow\mathbb{R}^{d\times m}$ and $h :\mathbb{R}^d \longrightarrow\mathbb{R}^d$ satisfy
the globally Lipschitz, and $N(t)$ is a one dimensional Poisson process with parameter $\lambda$. Extension to vector-valued jumps with independent entries is straightforward.
The one-sided Lipschitz function $f$ can be decomposed as $f=u+v$, where the function $u : \mathbb{R}^d\longrightarrow\mathbb{R}^d$ 
is the global Lipschitz continuous part and $v : \mathbb{R}^d\longrightarrow\mathbb{R}^d$ is the non-global Lipschitz continuous part, see e.g. \cite{semitamed}. 
Using this decomposition, we can rewrite the jump-diffusion SDEs \eqref{model} in the following equivalent form
\begin{eqnarray}
\label{model1}
X(t)=  \left(u(X(t^{-})+ v(X(t^{-}))\right)dt +g(X(t^{-}))dW(t)+h(X(t^{-}))dN(t).
\end{eqnarray}
This decomposition will be used only for semi-tamed schemes.
Equations of  type \eqref{model} arise in a range of scientific, engineering and financial applications \cite{appf2,appf,Fima}.
Most of such equations do not have explicit solutions and therefore one requires numerical schemes for  their
approximations. Their numerical analysis has  been studied in \cite{Desmond2,Wang,desmond,platen} with implicit and explicit schemes where strong and weak convergence have been investigated. 
The implementation of implicit schemes requires significantly more computational effort than 
the explicit Euler-type approximations as Newton method is usually required to solve nonlinear systems at each time iteration  in implicit schemes.
The standard explicit method for approximating SDEs of  type \eqref{model} is  the Euler-Maruyama method \cite{platen}. 
Recently it has been proved (see \cite{Armulfdiv,lamport94}) that the Euler-Maruyama method often fails to converge strongly 
to the exact solution of nonlinear SDEs of the form \eqref{model} without jump  term  when at least one of the  functions $f$ and $g$
 grows superlinearly. To  overcome  this drawback of the Euler-Maruyama method, numerical approximation, 
with computational cost  close to that of the Euler-Maruyama method and which converges strongly even in the case  the function $f$  is superlinearly growing was first introduced in
\cite{Armulf} and strong convergence was investigated. Further investigations have been  performed  in  the litterature (see for example \cite{tamed2,tamed3,semitamed}  and references therein),
where in \cite{tamed2} the time step $\Delta t$ in \cite{Armulf} is replaced by its power  $\Delta t^{\alpha}, \,\, \alpha \in (0,1/2]$ in the denominator of the taming  drift term. 
Recently  the work in \cite{tamed2} has been extended for SDEs driven by compensated Levy noise in \cite{Kumar,tamedmilsteinjump}.
The condition $ \alpha \in (0,1/2]$ is key  in  the convergence proofs in \cite{Kumar,tamed2,tamedmilsteinjump}, so the proofs cannot be  extended for $ \alpha \in [1/2,1]$.
Strong and weak convergences are not the only features of numerical techniques. Stability is also a good feature as 
the information about time step size for which  does a particular numerical method replicate the stability properties of the exact solution is valuable. 
 The linear stability is an extension of  the deterministic A-stability
 while exponential stability can guarantee that errors introduced in one time step will decay exponentially in future time steps,  exponential
 stability  also implies asymptotic stability \cite{Huang}.  By the Chebyshev inequality and the Borel--Cantelli lemma,
 it is well known that exponential mean-square stability implies almost sure stability \cite{Huang}. 
 The stability of  classical implicit and explicit methods for  \eqref{model} are well understood \cite{Desmond2,Huang,Wang}. 
 Although the strong convergence  of tamed schemes with and without jump have been studied, a rigorous stability properties have not yet been investigated to the best of our knowledge.
 

The aim of this paper is to study the strong convergence  of tamed schemes  driven by Brownian process and Poisson jump for  $\alpha\in[1/2,1]$,  and to  
 provide a rigorous study of the linear and  exponential stabilities of semi-tamed and tamed  schemes  for  $\alpha\in[0,1]$.
Following  closely the breakthrough idea  in \cite{Armulf}, 
 we  provide  the strong convergence of the tamed schemes and the corresponding semi tamed  schemes   both in compensated and non compensated forms  for  $\alpha\in[1/2,1]$.
 The extensions are not straightforward as several technical lemmas are needed.  Numerical experiments  
 show that the semi-tamed works better than the tamed and compensated tamed schemes. Numerical results also show that the  tamed and the compensated tamed Euler scheme 
 have good stability behavior when $\alpha$ approaches $1$. 
 Therefore, our tamed schemes  with  $\alpha\in[1/2,1]$  have  better   stability property than the tamed schemes presented in \cite{Kumar} for $\alpha\in(0,1/2]$.


The  paper is organized as follows. \secref{setting} presents 
the classical result of existence and uniqueness of the solution $X$ of \eqref{model}. The  compensated and non compensated tamed schemes and semi-tamed scheme are presented in \secref{schemes}
along with their strong convergences.  The linear stability of the schemes is provided in \secref{linearsta} while the  nonlinear exponential stability is provided in \secref{nlinearsta}.
We end  in \secref{simulations} by providing some numerical simulations.

\section{Notations, assumptions and well posedness}
\label{setting}
Throughout this work,  $(\Omega, \mathcal{F}, \mathbb{P})$ denotes a complete probability space with a filtration $(\mathcal{F}_t)_{t\geq 0}$. 
 For all $x, y\in\mathbb{R}^d$, 
we denote by $ \langle x, y \rangle= x_1y_1+x_2y_2+\cdots+x_dy_d$, $\|x\|= \langle x, x \rangle^{1/2}$, $\|A\|= \underset {x\in\mathbb{R}^d, \Vert x \Vert \leq 1}{\sup } \Vert Ax\Vert$ for all $A\in\mathbb{R}^{m\times d}$. $a\vee b$ represents $\max\{a,b\}$.
We use also the following convention : $\sum_{i=u}^n=0$ for $u>n$.

We first ensure  that SDEs \eqref{model} is well-posed.  The following assumption is needed.
\begin{Assumption}
\label{ass1}
We assume  that:\\
$(A.1)$ For all $p>0$, there exists  $M_p>0$ such that $\mathbb{E}\|X_0\|^p\leq M_p$,  and  $f,g,h\in C^1 (\mathbb{R}^d)$.\\
$(A.2)$ The functions $g$, $h$ and $u$ satisfy the  following global Lipschitz condition
\begin{eqnarray*}
\Vert g(x)-g(y) \Vert \vee \Vert h(x)-h(y)\Vert \vee \Vert u(x)-u(y) \Vert  \leq C \Vert x-y \Vert \hspace{0.5cm} \forall\hspace{0.5cm} x,y\in \mathbb{R}^d.\\
\end{eqnarray*}
$(A.3)$ The  function $f$ satisfies  the following one-sided Lipschitz condition
\begin{eqnarray*}
\langle x-y, f(x)-f(y) \rangle\leq C\|x-y\|^2\hspace{0.5cm} \forall\hspace{0.5cm} x,y\in \mathbb{R}^d.
\end{eqnarray*}
$(A.4)$ The function $f$ satisfies the following superlinear growth condition 
\begin{eqnarray*}
\Vert f(x)-f(y) \Vert \leq C(1+ \Vert x \Vert^c+ \Vert y \Vert^c) \Vert x-y \Vert \hspace{0.5cm} \forall\hspace{0.5cm} x,y\in \mathbb{R}^d,
\end{eqnarray*}
where $C$ and $c$ are positive constants. 
\end{Assumption}
\begin{remark}
\label{Lips}
 Note that from \assref{ass1},  $u$ satisfies the global Lipchitz condition,  and $f$  satisfies the one-sided Lipschitz condition  and the superlinear growth condition, which  implies that 
 the function $v$ satisfies the one-sided Lipschitz condition $(A.3)$ and the  superlinear growth condition $(A.4)$ in  \assref{ass1}.
 \end{remark}
\begin{theorem}
 Under  the conditions $(A.1)$, $(A.2)$ and $(A.3)$  of \assref{ass1}, the SDE \eqref{model} has a unique solution with all bounded moments.
\end{theorem}
\begin{proof} 
 See \cite{Gyongy} for the existence and the uniqueness and \cite[Lemma 1]{Desmond2} for the boundedness of the moments of the solution.
\end{proof}
\section{Numerical Schemes and main results}
\label{schemes}
We consider  the  SDEs \eqref{model} in  the current non compensated  form. Applying the  tamed Euler scheme (as in \cite{Armulf}) in the drift term of \eqref{model} yields the following schemes
that we will call non compensated tamed scheme (NCTS)
\begin{equation}
 X_{n+1}^{M}=X_{n}^{M}+\dfrac{\Delta t f(X_{n}^{M})}{1+ \Delta t^{\alpha}\Vert f(X_{n}^{N}) \Vert }+g(X_{n}^{M}) \Delta W_n^M +h(X_{n}^{M})\Delta N_n^M,
 \label{ncts}
\end{equation}
where $\Delta t=T/M$ is the  time step-size, $M\in\mathbb{N}$ is the number of time subdivisions, $\alpha\in[1/2,1]$, $\Delta W_n^M = W(t_{n+1})- W(t_{n})$ and $\Delta N_n^M = N(t_{n+1})- N(t_{n})$.
Applying the semi-tamed Euler scheme (as in \cite{semitamed}) in the non globally Lipschitz part $v$ of  the drift term of \eqref{model1} yields the following scheme 
that we will call  semi-tamed scheme (STS)
\begin{eqnarray}
 Z_{n+1}^{M}=Z_{n}^{M}+u(Z^M_n)\Delta t+\dfrac{\Delta t v(Z_{n}^{M})}{1+ \Delta t^{\alpha} \Vert v(Z_{n}^{M}) \Vert }+g(Z_{n}^{M}) \Delta W_n +h(Z_{n}^{M})\Delta N_n^M.
 \label{sts}
\end{eqnarray}
Recall that the compensated poisson process $\overline{N}(t) := N (t) - \lambda t $ is a martingale and satisfies the  the following properties
\begin{eqnarray}
 \mathbb{E}\left(\overline{N}(t+s)-\overline{N}(t)\right)=0\,\qquad \qquad  \mathbb{E}\vert \overline{N}(t+s)-\overline{N}(t)\vert^2=\lambda s,\qquad s, t \geqslant 0.
\end{eqnarray}
We  can  easily check  that the quadratic variation of $\overline{N}(t)$  is   $[\overline{N},\overline{N}]_t=N(t)$.

We can therefore rewrite the jump-diffusion SDEs \eqref{model} in the following equivalent form
\begin{eqnarray}
\label{modeln}
 dX(t)=  f_\lambda(X(t^{-})dt +g(X(t^{-})dW(t)+h(X(t^{-})d\overline{N}(t),
\end{eqnarray}
where
 $f_\lambda(x)=f(x)+\lambda h(x)$.
Note that  as $f$, the  function $f_\lambda$ satisfies the one-sided Lipschitz condition $(A.3)$ and  the superlinear growth $(A.4)$.
 Applying the  tamed Euler scheme  in  the drift term of \eqref{modeln} as  in \cite{Armulf} yields the following  updated scheme for jump SDEs \eqref{model} that we will call
 compensated tamed scheme (CTS)
\begin{eqnarray}
\label{tamedjump}
 Y_{n+1}^{M}=Y_{n}^{M}+\dfrac{\Delta t f_\lambda(Y_{n}^{M})}{1+ \Delta t^{\alpha} \Vert f_{\lambda}(Y_{n}^{M}) \Vert }+g(Y_{n}^{M}) \Delta W_n^M +h(Y_{n}^{M})\Delta \overline{N}_n^M ,
\end{eqnarray}
where $\Delta \overline{N}_n^M = \overline{N}(t_{n+1})- \overline{N}(t_{n})$.  

Note that if the equivalent model \eqref{model1} is putting in  the compensated form, and the semi-tamed method  is applied on the non globally Lipschitz
part $v$ of  the drift term $f$, we will obtain the same scheme as in \eqref{sts}.

We define the continuous time interpolations of the discrete numerical approximations  
of \eqref{ncts}, \eqref{sts} and \eqref{tamedjump} respectively by

\begin{eqnarray}
\label{dncts}
\overline{X}^M_t =X^M_n+ \dfrac{(t-n\Delta t)f(X^M_n)}{1+\Delta t^{\alpha}\|f(X^M_n)\|}+g(X^M_n)(W_t-W_{n\Delta t})+ h(X^M_n)(N_t-N_{n\Delta t}),
\end{eqnarray}
\begin{eqnarray}
\label{dsts}
\overline{Z}^M_t =Z^M_n+(t-n\Delta t) \left(u(Z^M_n)+\dfrac{v(Z^M_n)}{1+\Delta t^{\alpha}\|v(Z^M_n)\|}\right)+g(Z^M_n)(W_t-W_{n\Delta t})\nonumber\\
+ h(Z^M_n)(N_t-N_{n\Delta t}),
\end{eqnarray}
and 
\begin{eqnarray}
\label{dtamedjump}
\overline{Y}^M_t =Y^M_n+\dfrac{(t-n\Delta t)f_{\lambda}(Y^M_n)}{1+\Delta t^{\alpha}\|f_{\lambda}(Y^M_n)\|}+g(Y^M_n)(W_t-W_{n\Delta t})+ h(Y^M_n)(\overline{N}_t-\overline{N}_{n\Delta t}),
\end{eqnarray}
for all $t\in[n\Delta t, (n+1)\Delta t), \hspace{0.3cm} n\in\{0,\cdots, M-1\}$.

The main result of this section is given  in the following theorem.
\begin{theorem} 
\label{ch4theorem1}
[\text{Main result} ]\\
Let $X_t$ be the exact solution of \eqref{model} and  $\chi^M_t$ the discrete continuous  form of the numerical approximations given by \eqref{dncts},\eqref{dsts} and \eqref{dtamedjump}
($\chi^M_t= \overline{X}^M_t$ for scheme NCTS, $\chi^M_t= \overline{Z}^M_t$ for  scheme STS and $\chi^M_t= \overline{Y}^M_t$  for scheme CTS).
Under \assref{ass1}, for all $p\in[1,+\infty)$  there exists a  constant $C_p> 0$ independent of $\Delta t$ such that
 \begin{eqnarray}
 \left(\mathbb{E}\left[\sup_{t\in[0,T]}\|X_t-\chi^M_t\|^p\right]\right)^{1/p}\leq C_p\Delta t^{1/2}, \qquad \,\, \Delta t=T/M.
 \end{eqnarray}
\end{theorem}
\subsection{ Proof of \thmref{ch4theorem1} for $\chi^M_t= \overline{Y}^M_t$}
\label{tamedjs}
Before giving the proof  of \thmref{ch4theorem1}, some preparatory results are needed. Here we consider the  compensated tamed scheme (CTS) given by \eqref{tamedjump}.
 \subsubsection{Preparatory results}
 Throughout  this work the following notations will be used with slight modification  in the next section
 {\small {
\begin{eqnarray} 
\label{notation}
  \left\lbrace \begin{array}{l} 
 \alpha^M_k := \mathbbm{1}_{\{\|Y^M_k\|\geq 1\}}\left\langle\dfrac{Y^M_k}{\|Y^M_k\|}, \dfrac{g(Y^M_k)}{\|Y^M_k\|}\Delta W^M_k\right\rangle,\quad k=0,\cdots, M\\
 \beta^M_k := \mathbbm{1}_{\{ \|Y^M_k||\geq 1\}}\left\langle\dfrac{Y^M_k}{ \|Y^M_k\|}, \dfrac{h(Y^M_k)}{||Y^M_k\|}\Delta\overline{N}^M_k\right\rangle,\quad k=0,\cdots, M\\
 \beta :=\left(1+K+2C+KTC+TC+T\|f_{\lambda}(0)\|+\|g(0)\|+\|h(0)\|\right)^4,\\
 D^M_n := \left(\beta+\|X_0\| \right)\exp\left(\dfrac{3\beta}{2}+ \underset{m\in\{0,\cdots,n\}}{\sup } \underset {k=m}{\sum ^{n-1}} \left[\dfrac{3\beta}{2}\|\Delta W^M_k\|^2+\dfrac{3\beta}{2}|\Delta\overline{N}^M_k|+\alpha^M_k+\beta^M_k\right]\right),\\
 \Omega^M_n  :=\Big \{\omega\in \Omega : \underset{k\in\{0,1,\cdots, n-1\}} {\sup} D^M_k(\omega)\leq M^{1/2c}, \underset{k\in\{0,1,\cdots, n-1\}} {\sup} \Big(\|\Delta W^M_k(\omega)\|\vee|\Delta \overline{N}^M_k(\omega)|\Big)\leq 1 \Big\}.
 \end{array} \right.
 \end{eqnarray}
 }}
  The  aims  of this section is to update  all lemmas used in \cite{Armulf} and provide  new lemmas  for Poisson jump. 
 \begin{lemma}
 \label{ch4lemma1}
 For all  $a,b\geq 0$, the following inequality holds 
 \begin{eqnarray*}
 1+a+b^2\leq e^{a+\sqrt{2}b}.
 \end{eqnarray*}
 \end{lemma}
 
\begin{proof}
  For $a\geq 0$ fixed, let us define the function $l(b)=e^{a+\sqrt{2}b}-1-a-b^2$. It can be easily checked that $l'(b)=\sqrt{2}e^{a+\sqrt{2}b}-2b$ and
  $l''(b)=2(e^{a+\sqrt{2}b}-1)$. Since $a$ and $b$ are positive, it follows that $l''(b)\geq 0$ for all $b\geq 0$. So $l'$ is a non-decreasing function. 
  Therefore, $l'(b)\geq l'(0)=\sqrt{2}e^a>0$ for all $b\geq 0$. This implies that $l$ is a non-decreasing function. Hence $l(b)\geq l(0)=e^a-1-a$ 
  for all $b\geq 0$. Since $1+a\leq e^a$ for all positive number $a$, it follows that $l(b)\geq 0$ for all positive number $b$, so 
  $1+a+b^2\leq e^{a+\sqrt{2}b}$, $\forall\;b\geq0$. Therefore for all $a\geq 0$ fixed, $1+a+b^2\leq e^{a+\sqrt{2}b}$, $\forall\;b\geq0$.
 \end{proof}
 
 Following closely  \cite[Lemma 3.1]{Armulf}, we have  the following main lemma.
 \begin{lemma}\label{ch4lemma2}
 The following inequality holds for all $M\in \mathbb{N}$ 
 and all $n\in\{0,1,\cdots, M\}$
 \begin{eqnarray}
 \mathbbm{1}_{\Omega^M_n}\|Y^M_n\|\leq D^M_n,
 \label{ch4Denobound}
 \end{eqnarray}
 where $D_n^M$ and $\Omega^M_n$ are given in \eqref{notation}.
 \end{lemma}
 \begin{proof}
 As $\dfrac{\Delta t}{1+\Delta t^{\alpha}\|f_{\lambda}(x)\|}\leq T$, 
   using  \assref{ass1} on the functions $g$, $h$ and $f_{\lambda}$, following closely \cite{Armulf},
  the following estimation  holds  on $\Omega^M_{n+1}\cap\{\omega\in \Omega : \|Y^M_n(\omega)\|\leq 1\}$,  for all $n\in\{0,1,\cdots, M-1\}$ 
 \begin{align*}
 \|Y^M_{n+1}\|&\leq \|Y^M_n\|+\dfrac{\Delta t \|f_{\lambda}(Y^M_n) \|}{1+\Delta t^{\alpha} \|f_{\lambda}(Y^M_n)\|}+\|g(Y^M_n)\| \|\Delta W^M_n\|+\|h(Y^M_n)\||\Delta\overline{N} ^M_n| \nonumber \\
 &\leq \|Y^M_n\|+T \|f_{\lambda}(Y^M_n)-f_{\lambda}(0)\|+T \|f_{\lambda}(0)\|+ \|g(Y^M_n)-g(0)\|+\|g(0)\|\nonumber\\
 &+\|h(Y^M_n)-h(0)\|+\|h(0)\|\nonumber\\
 &\leq \|Y^M_n\|+TC(K+\|Y^M_n\|^c)\|Y^M_n\|+T\|f_{\lambda}(0)\|+C\|Y^M_n\|+C\|Y^M_n\|+\|g(0)\|+ \|h(0)\|.\nonumber
 \end{align*}
 On  $\Omega^M_{n+1}\cap\{\omega\in \Omega : \|Y^M_n(\omega)\|\leq 1\}$, we therefore have 
 \begin{align}
 \|Y^M_{n+1}\|
 &\leq  1+KTC +TC+2C+T \|f_{\lambda}(0)\|+\|g(0)\|+\|h(0)\|\leq \beta.
 \label{ch4normY}
 \end{align}
 Using Cauchy-Schwartz inequality  and Holder inequality,  after some simplications, we have
 \begin{eqnarray}
 \label{ys}
\|Y^M_{n+1}\|^2 &\leq & \|Y^M_n\|^2+3\Delta t^2\|f_{\lambda}(Y^M_n)\|^2+3\|g(Y^M_n)\|^2 \|\Delta W^M_n\|^2+3 \|h(Y^M_n)\|^2|\Delta\overline{N}^M_n|^2\nonumber\\
 &+&2\Delta t\left|\langle Y^M_n,f_{\lambda}(Y^M_n)\rangle\right|+ 2\langle Y^M_n,g(Y^M_n)\Delta W^M_n\rangle\nonumber\\
 &+&2\langle Y^M_n,h(Y^M_n)\Delta\overline{N}^M_n\rangle
 \label{ch4ine16}
 \end{eqnarray}
 on $\Omega$, for all $M\in\mathbb{N}$ and all $n\in\{0,1,\cdots, M-1\}$.
 
  Using \assref{ass1}, we can easily  prove (see \cite{Armulf}) that 
  for all $x\in\mathbb{R}^d$ such that $1\leq \Vert x\Vert\leq M^{\frac{1}{2c}}$ we have 
 \begin{eqnarray}
 \label{ch4normdeg2}
  \left\lbrace \begin{array}{l}
 \|g(x)\|^2
  \leq \beta\|x\|^2 \\
  \|h(x)\|^2\leq \beta\|x\|^2\\
  \langle x,f_{\lambda}(x)\rangle \leq \sqrt{\beta}\|x \|^2\\
  \|f_{\lambda}(x)\|^2 \leq  M\sqrt{\beta}\|x \|^2.
  \end{array} \right.
 \end{eqnarray}
 Using \eqref{ch4normdeg2} in \eqref{ys} yields
 \begin{eqnarray}
 \|Y^M_{n+1}\|^2 &\leq& \|Y^M_n \|^2+\dfrac{3T^2\sqrt{\beta}}{M}\|Y^M_{n}\|^2+3\beta\|Y^M_n\|^2\|\Delta W^M_n\|^2+3\beta\|Y^M_n\|^2|\Delta \overline{N}^M_n|^2\nonumber\\
 &+&\dfrac{2T\sqrt{\beta}}{M}\|Y^M_n\|^2+2\langle Y^M_n, g(Y^M_n)\Delta W^M_n\rangle+2\langle Y^M_n, h(Y^M_n)\Delta\overline{N}^M_n\rangle\nonumber\\
 &\leq & \|Y^M_n\|^2+\dfrac{(3T^2+2T)\sqrt{\beta}}{M}\|Y^M_n\|^2+3\beta\|Y^M_n\|^2 \|\Delta W^M_n\|^2+3\beta\|Y^M_n\|^2|\Delta\overline{N}^M_n|^2\nonumber\\
 &+&2\langle Y^M_n, g(Y^M_n)\Delta W^M_n\rangle
 +2\langle Y^M_n, h(Y^M_n)\Delta\overline{N}^M_n\rangle.
 \end{eqnarray}
 Using  the fact  that $3T^2+2T\leq 3\sqrt{\beta}$, it follows that
 \begin{eqnarray}
 \|Y^M_{n+1}\|^2&\leq& \|Y^M_n\|^2+\dfrac{3\beta}{M}\|Y^M_n\|^2+3\beta\|Y^M_n\|^2 \|\Delta W^M_n\|^2+3\beta\|Y^M_n\|^2|\Delta\overline{N}^M_n|^2\nonumber\\
 &+&2\langle Y^M_n,g(Y^M_n)\Delta W^M_n\rangle+2\langle Y^M_n,h(Y^M_n)\Delta\overline{N}^M_n\rangle\nonumber\\
 &=&\|Y^M_n \|^2 \left(1+\dfrac{3\beta}{M}+3\beta\|\Delta W^M_n \|^2+3\beta|\Delta\overline{N}^M_n|^2+2\left<\dfrac{Y^M_n}{\|Y^M_n\|}, \dfrac{g(Y^M_n)}{\|Y^M_n\|}\Delta W^M_n\right>   \right.\nonumber\\
 &+&\left. 2\left\langle\dfrac{Y^M_n}{\|Y^M_n\|}, \dfrac{h(Y^M_n)}{\|Y^M_n\|}\Delta\overline{N}^M_n\right\rangle\right)\nonumber\\
 &=&\|Y^M_n\|^2\left(1+\dfrac{3\beta}{M}+3\beta \|\Delta W^M_n\|^2+3\beta|\Delta\overline{N}^M_n|^2+2\alpha^M_n+2\beta^M_n\right).
 \label{ch4expY1}
 \end{eqnarray}
 Using Lemma \ref{ch4lemma1} for  
  $a=\dfrac{3\beta}{M}+3\beta\|\Delta W^M_n\|^2+2\alpha^M_n+2\beta^M_n$ and $b=\sqrt{3\beta}|\Delta\overline{N}^M_n| $ it follows from \eqref{ch4expY1} that  :
 \begin{eqnarray}
 \|Y^M_{n+1}\|^2\leq \|Y^M_n\|^2\exp\left(\dfrac{3\beta}{M}+3\beta\|\Delta W^M_n\|^2+3\beta|\Delta\overline{N}^M_n|+2\alpha^M_n+2\beta^M_n\right)
 \label{ch4expY2}
 \end{eqnarray}
 on $\{w\in\Omega : 1\leq \|Y^M_n(\omega)\|\leq M^{1/2c}\}$, for all $M\in\mathbb{N}$ and all $n\in\{0,1,\cdots,M-1\}$.
 
 Our proof is concluded by induction exactly as in \cite[Lemma 3.1]{Armulf}.  Details can be found  in \cite{Mukam}.
 
 \end{proof}

 The  proofs of  the following can be found in \cite{Armulf,Mukam}.
 \begin{lemma}\label{ch4lemma4}
 The following inequality holds 
 \begin{eqnarray*}
 \sup_{M\in\mathbb{N}, M\geq 4\beta pT}\mathbb{E}\left[\exp\left(\beta p\sum_{k=0}^{M-1}\|\Delta W^M_k\|^2\right)\right]<\infty, \qquad \qquad \forall p \in [1,\infty).
 \end{eqnarray*}
 \end{lemma}
 \begin{lemma}\label{ch4lemma7}
 Let $\alpha^M_n :\Omega\longrightarrow\mathbb{R},\, M\in\mathbb{N},\,n\in\{0,1,\cdots,M\}$  be  the  process defined in  \eqref{notation}. 
 The following inequality holds 
 \begin{eqnarray*}
 \sup_{z\in\{-1,1\}}\sup_{M\in\mathbb{N}}\left\|\sup_{n\in\{0,1,\cdots,M\}}\exp\left(z\sum_{k=0}^{n-1}\alpha^M_k\right)\right\|_{L^p(\Omega, \mathbb{R})}<\infty,
 \end{eqnarray*}
 for all $p\in[2,+\infty)$.
 \end{lemma}
 \begin{lemma}
 \label{ch4lemma8}
 Let  $\hspace{0.3cm} c\in\mathbb{R}$, the following equality holds
 \begin{eqnarray*}
 \mathbb{E}[\exp(c\Delta \overline{N}^M_n)]=\exp\left[\dfrac{(e^c+c-1)\lambda T}{M}\right]
 \end{eqnarray*}
 for all $M\in\mathbb{N}$ and all $n\in\{0,\cdots,M\}$.
 \end{lemma}
 \begin{proof}
 From the moment generating function of a Poisson process $Y$ with parameter $\lambda$, we have 
 \begin{eqnarray*}
 \mathbb{E}[\exp(cY)]=\exp(\lambda(e^c-1)).
 \end{eqnarray*}
  Since $\Delta N_n$ follows  a poisson law with parameter $\lambda\Delta t$, it follows that 
 \begin{eqnarray*}
 \mathbb{E}[\exp(c\Delta\overline{N}^M_n)]&=&\mathbb{E}[\exp(c\Delta N^M_n+c\lambda\Delta t)]\\
 &=&\mathbb{E}\left[\exp\left(\dfrac{ c\lambda T}{M}\right)\exp(c\Delta N^M_n)\right]\\
 &=&\exp\left(\dfrac{c\lambda T}{M} \right)\exp\left[ \dfrac{\lambda T}{M}(e^c-1)\right]\\
 &=&\exp\left[\dfrac{(e^c+c-1)\lambda T}{M}\right].
 \end{eqnarray*}
 \end{proof}
 
 \begin{lemma}\label{ch4lemma9}
 The following inequality holds
 \begin{eqnarray*}
 \mathbb{E}\left[\exp\left(pz\mathbbm{1}_{\{\|x\|\geq 1\}}\left<\dfrac{x}{\|x\|},\dfrac{h(x)}{\|x\|}\Delta\overline{N}^M_n\right>\right)\right]\leq\exp\left[\dfrac{\lambda \left(e^{p(C+\|h(0)\|)}+p(C+\|h(0)\|\right)}{M}\right],
 \end{eqnarray*}
 for all $M\in\mathbb{N}$, $z\in\{-1,1\}$, all $p\in[1, +\infty)$ and all $n\in\{0,\cdots,M\}$.
 \end{lemma}
 
 \begin{proof} For $x\in\mathbb{R}^d$ such that $\|x\|\neq 0$, we have 
 \begin{eqnarray*}
  \mathbb{E}\left[\exp\left(pz\left<\dfrac{x}{\|x\|},\dfrac{h(x)}{\|x\|}\Delta\overline{N}^M_n\right>\right)\right]  &\leq& \mathbb{E}\left[\exp\left(pz\dfrac{\|x\|\|h(x)\|}{\|x\|^2}\Delta\overline{N}^M_n\right)\right]\\
  &=& \mathbb{E}\left[\exp\left(pz\dfrac{\|h(x)\|}{\|x\|}\Delta \overline{N}^M_n\right)\right].
  \end{eqnarray*}
 For all $x\in\mathbb{R}^d$ such that $\|x\|\geq 1$, since $h$ satisfied the global Lipschitz condition,  we have 
  \begin{eqnarray}
  \dfrac{\|h(x)\|}{\|x\|}\leq \dfrac{\|h(x)-h(0)\|+\|h(0)\|}{\|x\|}\leq C+\|h(0)\|. \label{ch4normeh}
  \end{eqnarray}
  So from inequality \eqref{ch4normeh} and using Lemma \ref{ch4lemma8} it follows that 
  \begin{eqnarray*}
   \mathbb{E}\left[\exp\left(pz\mathbbm{1}_{\{\|x\|\geq 1\}}\left<\dfrac{x}{||x||},\dfrac{h(x)}{\|x\|}\Delta\overline{N}^M_n\right>\right)\right] &\leq &\mathbb{E}[\exp(pz(C+\|h(0)\|)\Delta\overline{N}^M_n)]\\
   &\leq &\exp\left[\dfrac{\left(e^{p(C+\|h(0)\|)}+p(C+\|h(0)\|)-1\right)\lambda T}{M}\right]\\
   &\leq &\exp\left[\dfrac{\left(e^{p(C+\|h(0)\|)}+p(C+\|h(0)\|\right)\lambda T}{M}\right].
  \end{eqnarray*}
 \end{proof}
 
 \begin{lemma}\label{ch4lemma10}
 Let $\beta^M_n :\Omega \longrightarrow \mathbb{R}$ be the process  defined as in \eqref{notation} for all $M\in\mathbb{N}$ and all $n\in\{0,\cdots,M\}$. The following inequality holds
 \begin{eqnarray*}
 \sup_{z\in\{-1,1\}}\sup_{M\in\mathbb{N}}\left\|\sup_{n\in\{0,\cdots, M\}}\exp\left(z\sum_{k=0}^{n-1}\beta^M_k\right)\right\|_{L^p(\Omega, \mathbb{R})}<+\infty,\quad \quad p\in (1,\infty).
 \end{eqnarray*} 
 \end{lemma}
 
 \begin{proof}
 For the same reasons as for $\alpha^M_k$, $\beta^M_k$ is an $(\mathcal{F}_{nT/M})$- martingale, see e.g. the proof of \cite[Lemma 3.4]{Armulf}. So $\exp\left(pz\sum_{k=0}^{n-1}\beta^M_k\right)$ is a positive $(\mathcal{F}_{nT/M})$- submartingale for all $M\in\mathbb{N}$  and all $n\in\{0,\cdots, M\}$. Using Doop's maximal inequality we have :
 \begin{eqnarray}
 \left\|\sup_{n\in\{0,\cdots, M\}}\exp\left(z\sum_{k=0}^{n-1}\beta^M_k\right)\right\|_{L^p(\Omega, \mathbb{R})} &\leq &\left(\dfrac{p}{p-1}\right)\left\|\exp\left(z\sum_{k=0}^{M-1}\beta^M_k\right)\right\|_{L^p(\Omega,\mathbb{R})},
 \label{ch4beta1}
 \end{eqnarray}
 \begin{eqnarray*}
 \left\|\exp\left(z\sum_{k=0}^{M-1}\beta^M_k\right)\right\|^p_{L^p(\Omega,\mathbb{R})}&=&\mathbb{E}\left[\exp\left(pz\sum_{k=0}^{M-1}\beta^M_k\right)\right]=\mathbb{E}
 \left[\exp\left(pz\left(\sum_{k=0}^{M-2}\beta^M_k\right)+pz\beta_{M-1}^M\right)\right]\\
 &=&\mathbb{E}\left[\exp\left(pz\sum_{k=0}^{M-2}\beta^M_k\right)
 \mathbb{E}\left[\exp\left(pz\beta^M_{M-1}\right)/\mathcal{F}_{(M-1)T/M}\right]
 \right].
 \end{eqnarray*}
 Using Lemma \ref{ch4lemma9} it follows that
 \begin{eqnarray*}
 \left\|\exp\left(z\sum_{k=0}^{M-1}\beta^M_k\right)\right\|^p_{L^p(\Omega,\mathbb{R})} \leq \mathbb{E}\left[\exp\left(pz\sum_{k=0}^{M-2}\beta^M_k\right)\right]\exp\left[\dfrac{\left(e^{p(C+||h(0)||)}+p(C+||h(0)||)\right)\lambda T}{M}\right].
  \end{eqnarray*}
 Iterating this last inequality $M$ times leads to
 \begin{eqnarray}
 \mathbb{E}\left[\exp\left(pz\sum_{k=0}^{M-1}\beta^M_k\right)\right] \leq \exp\left[\lambda T\left(e^{p(C+\|h(0)\|)}+Tp(C+\|h(0)\|\right)\right],
 \label{ch4beta2}
 \end{eqnarray}
 for all $M\in\mathbb{N}$, all $p\in (1,\infty)$ and all $z\in\{-1,1\}$.
 
 Combining   inequalities \eqref{ch4beta1} and \eqref{ch4beta2} completes the proof of Lemma \ref{ch4lemma10}.
 \end{proof}
 
 \begin{lemma}\label{ch4lemma11}
The following inequality holds 
\begin{eqnarray*}
\sup_{M\in\mathbb{N}}\mathbb{E}\left[\exp\left(p\beta\sum_{k=0}^{M-1}|\Delta\overline{N}^M_k|\right)\right]<+\infty,
\end{eqnarray*}
for all $p\in[1, +\infty)$.
\end{lemma}  
   \begin{proof}
  Using the independence and the stationarity of $\Delta\overline{N}_k^M$, along with \lemref{ch4lemma8}, it follows that
  \begin{eqnarray*}
  \sup_{M\in\mathbb{N}}\mathbb{E}\left[\exp\left(p\beta\sum_{k=0}^{M-1}|\Delta\overline{N}^M_k|\right)\right]&=&\sup_{M\in\mathbb{N}}\left(\prod_{k=0}^{M-1}\mathbb{E}[\exp(p\beta|\Delta\overline{N}^M_k|)]\right)\\
  &=&\sup_{M\in\mathbb{N}}\left[\left(\mathbb{E}[\exp(p\beta|\Delta\overline{N}^M_k|)]\right)^M\right]\\
  &=&\sup_{M\in\mathbb{N}}\left[\left(\exp\left[\dfrac{(e^{p\beta}+p\beta-1)\lambda T}{M}\right]\right)^M\right]\\
  &=&\exp[ \lambda T\left(e^{p\beta}+p\beta-1\right)]<+\infty,
  \end{eqnarray*}
  for all $p\in[1, +\infty)$.
  \end{proof}
  
  Inspired by   \cite[Lemma 3.5]{Armulf}, we have the following estimation.
  \begin{lemma}
  \label{ch4lemma12}
  [\textbf{Uniformly bounded moments of the process $D_n^M$}]\\
  Let  $D_n^M : \Omega \longrightarrow [0,\infty),\, M\in\mathbb{N},\,\, n\in\{0,1,\cdots, M\}$ be the process  defined in \eqref{notation}, then we have 
  \begin{eqnarray*}
  \sup_{M\in\mathbb{N}, M\geq8\lambda pT}\left\|\sup_{n\in\{0,1,\cdots, M\}}D_n^M\right\|_{L^p(\Omega, \mathbb{R})}<\infty,
  \end{eqnarray*}
  for all $p\in[1,\infty)$. 
  \end{lemma}
  \begin{proof}
  Recall  that 
  \begin{eqnarray*}
  D_n^M =(\beta+\|X_0\|)\exp\left(\dfrac{3\beta}{2}+\sup_{m\in\{0,\cdots,n\}}\sum_{k=m}^{n-1}\dfrac{3\beta}{2}\|\Delta W^M_k\|^2+\dfrac{3\beta}{2}|\Delta\overline{N}^M_k|+\alpha^M_k+\beta^M_k\right).
  \end{eqnarray*}
  Using Holder inequality, it follows that
  \begin{eqnarray*}
  \sup_{M\in\mathbb{N}, M\geq 8\lambda pT}\left\|\sup_{n\in\{0,\cdots, M\}}D_n^M\right\|_{L^p(\Omega, \mathbb{R}^d)}&\leq &e^{3\beta/2}\left(\beta+\|X_0\|_{L^{4p}(\Omega, \mathbb{R})}\right)\\
  &\times& \underset{M\in\mathbb{N}, M\geq 8\lambda pT}{\sup}\left\|\exp\left(\dfrac{3\beta}{2}\sum_{k=0}^{M-1}\|\Delta W^M_k\|^2\right)\right\|_{L^{2p}(\Omega, \mathbb{R})}\\
  &\times &\sup_{M\in\mathbb{N}}\left\|\exp\left(\dfrac{3\beta}{2}\sum_{k=0}^{M-1}|\Delta\overline{N}_k^M|\right)\right\|_{L^{8p}(\Omega, \mathbb{R})}\\
  &\times &\left(\sup_{M\in\mathbb{N}}\left\|\sup_{n\in\{0,\cdots, M\}}\exp\left(\sup_{m\in\{0,\cdots,n\}}\sum_{k=m}^{n-1}\alpha_k^M\right)\right\|_{L^{16p}(\Omega, \mathbb{R})}\right)\\
  &\times &\left(\sup_{M\in\mathbb{N}}\left\|\sup_{n\in\{0,\cdots, M\}}\exp\left(\sup_{m\in\{0,\cdots,n\}}\sum_{k=m}^{n-1}\beta_k^M\right)\right\|_{L^{16p}(\Omega, \mathbb{R})}\right)\\
  &=& A_1\times A_2\times A_3\times A_4\times A_5.
  \end{eqnarray*}
  By assumption $A_1$ is bounded. Lemma \ref{ch4lemma4} and \ref{ch4lemma11} show that $A_2$ and $A_3$ are bounded.  Using again Holder inequality and Lemma \ref{ch4lemma7} it follows that 
  \begin{eqnarray*}
  A_4=\left\|\sup_{n\in\{0,\cdots, M\}}\exp\left(\sup_{m\in\{0,\cdots,n\}}\sum_{k=m}^{n-1}\alpha_k^M\right)\right\|_{L^{16p}(\Omega, \mathbb{R})}
  \end{eqnarray*}
  \begin{eqnarray*}
   \leq\left\|\sup_{n\in\{0,\cdots, M\}}\exp\left(\sum_{k=0}^{n-1}\alpha^M_k\right)\right\|_{L^{32p}(\Omega,\mathbb{R})}
  \times \left\|\sup_{m\in\{0,\cdots, M\}}\exp\left(-\sum_{k=0}^{m-1}\alpha^M_k\right)\right\|_{L^{32p}(\Omega,\mathbb{R})} <+\infty,
  \end{eqnarray*}
  for all $M\in\mathbb{N}$ and all $p\in[1,\infty)$.
  
  Along the same lines as above,  we prove that $A_5$ is bounded.
  
  Since each of the terms $A_1, A_2, A_3, A_4$ and $A_5$ is bounded, this completes the proof of  Lemma \ref{ch4lemma12}.
  \end{proof}
  
 The following lemma is an extension of \cite[Lemma 3.6]{Armulf}. Here, we include  the jump part.
  \begin{lemma}\label{ch4lemma13}
  Let  $\Omega_M^M\in\mathcal{F},\, M\in \mathbb{N}$ be the process defined in \eqref{notation}. The following holds
  \begin{eqnarray*}
  \sup_{M\in\mathbb{N}}\left(M^p\mathbb{P}[(\Omega_M^M)^c]\right)<+\infty,
  \end{eqnarray*}
  for all $p\in[1,\infty)$.
  \end{lemma}
   \begin{proof}
   Using the subadditivity of the probability measure and the Markov's inequality, it follows that
   \begin{eqnarray*}
   \mathbb{P}[(\Omega_M^M)^c] &\leq & \mathbb{P}\left[\sup_{n\in\{0,\cdots, M-1\}}D_n^M>M^{1/2c}\right]+M\mathbb{P}\left[\|W_{T/M}\|>1\right]+M\mathbb{P}\left[|\overline{N}_{T/M}|>1\right]\nonumber\\
    &\leq & \mathbb{P}\left[\sup_{n\in\{0,\cdots, M-1\}}|D_n^M|^q>M^{q/2c}\right]+M\mathbb{P}\left[\|W_{T}\|>\sqrt{M}\right]+M\mathbb{P}\left[|\overline{N}_{T}|>M\right]\nonumber\\
     &\leq & \mathbb{P}\left[\sup_{n\in\{0,\cdots, M-1\}}|D_n^M|^q>M^{q/2c}\right]+M\mathbb{P}\left[\|W_{T}\|^q>M^{q/2}\right]+M\mathbb{P}\left[|\overline{N}_{T}|^q>M^q\right]\nonumber\\
   &\leq &\mathbb{E}\left[\sup_{n\in\{0,\cdots, M-1\}}|D_n^M|^q\right]M^{-q/2c}+\mathbb{E}[\|W_T\|^q]M^{1-q/2}+ \mathbb{E}[|\overline{N}_{T}|^q] M^{1-q} \nonumber,\\
   \end{eqnarray*}
   for all $q>1$.
   
   Multiplying both sides of the above inequality  by $M^p$ leads to
   \begin{eqnarray*}
   M^p\mathbb{P}[(\Omega_M^M)^c]\leq \mathbb{E}\left[\sup_{n\in\{0,\cdots, M-1\}}|D_n^M|^q\right]M^{p-q/2c}+\mathbb{E}[\|W_T\|^q]M^{p+1-q/2}+\mathbb{E}[|\overline{N}_{T}|^q] M^{p+1-q}
   \end{eqnarray*}
   for all $q>1$.
   
    For $q>\max\{2pc, 2p+2\}$, we have  $M^{p+1-q/2}<1$, $M^{p-q/2c}<1$ and $ M^{p+1-q}$. It follows for this choice of $q$ that 
   \begin{eqnarray*}
   M^p\mathbb{P}[(\Omega_M^M)^c]\leq \mathbb{E}\left[\sup_{n\in\{0,\cdots, M-1\}}|D_n^M|^p\right]+\mathbb{E}[\|W_T\|^q]+\mathbb{E}[|\overline{N}_{T}|^q].
   \end{eqnarray*}
   Using Lemma \ref{ch4lemma12} and the fact that  $W_T$ and  $\overline{N}_{T}$ are independent of $M$, it follows that
   \begin{eqnarray*}
   \sup_{M\in\mathbb{N}}\left(M^p\mathbb{P}[(\Omega_M^M)^c]\right)<+\infty.
   \end{eqnarray*}
   \end{proof}
   
   The following lemma can be found in  \cite[Theorem 48 pp 193]{Protter} or in \cite[Theorem 1.1, pp 1]{Fima}.
\begin{lemma}\label{ch4lemma18b}[Burkholder-Davis-Gundy inequality (BDG)]

 Let $M$ be a martingale with c\`{a}dl\`{a}g paths and let $p\geq 1$ be fixed. Let
   $M_t^*=\sup\limits_{s\leq t}\|M_s\|$. Then there exist constants $c_p$ and $C_p$ such that 
 \begin{eqnarray*}
c_p \left[\mathbb{E}\left([M,M]_t\right)^{p/2}\right]^{1/p}\leq \left[\mathbb{E}(M_t^*)^p\right]^{1/p}\leq C_p\left[\mathbb{E}\left([ M, M]_t\right)^{p/2}\right]^{1/p},
 \end{eqnarray*}
 for all $0\leq t\leq \infty$, where $[M,M]_t$ stand for the quadratic variation of the process $M$. 
\end{lemma}  
%

The proof of the following lemma can be found  in \cite[Lemma 3.7]{Armulf} or \cite{Mukam}.
  \begin{lemma} \label{ch4lemma15}
  
  Let $k\in\mathbb{N}$ and let $Z : [0,T]\times \Omega \longrightarrow \mathbb{R}^{k\times m}$ be a predictable stochastic process satisfying     
  $\mathbb{P}\left[\int_0^T\|Z_s\|^2ds<+\infty\right]=1$. Then we have the following inequality 
  \begin{eqnarray*}
  \left\|\sup_{s\in[0,t]}\left\|\int_0^sZ_udW_u\right\|\right\|_{L^p(\Omega, \mathbb{R})}\leq C_p\left(\int_0^t\sum_{i=1}^m \|Z_s\vec{e}_i\|^2_{L^p(\Omega, \mathbb{R}^k)}ds\right)^{1/2}
  \end{eqnarray*}
  for all $t\in[0,T]$ and all $p\in[1,\infty)$, where $(\vec{e}_1,\cdots,\vec{e}_m)$ is the canonical basis of $\mathbb{R}^m$.
  \end{lemma}

 The following lemma  can be found in \cite[Lemma 3.8, pp 16]{Armulf} or \cite{Mukam}. 
  \begin{lemma}\label{ch4lemma16} 
  
  Let $k\in\mathbb{N}$ and let $ Z^M_l : \Omega \longrightarrow \mathbb{R}^{k\times m}$, $l\in\{0,1,\cdots, M-1\}$, $M\in\mathbb{N}$ be a familly  of mappings such that $Z^M_l$ is $\mathcal{F}_{lT/M}/\mathcal{B}(\mathbb{R}^{k\times m})$-measurable for all $l\in\{0,1,\cdots,M-1\}$ and $M\in\mathbb{N}$. Then the following inequality holds :
  \begin{eqnarray*}
  \left\|\sup_{j\in\{0,1,\cdots,n\}}\left\|\sum_{l=0}^{j-1}Z_l^M\Delta W^M_l\right\|\right\|_{L^p(\Omega, \mathbb{R})} \leq C_p\left(\sum_{l=0}^{n-1}\sum_{i=1}^m||Z^M_l.\vec{e}_i||^2_{L^p(\Omega, \mathbb{R}^k)}\dfrac{T}{M}\right)^{1/2},
  \end{eqnarray*}
  for all $p\geq 1$.
  \end{lemma}
%
%
%

 \begin{lemma}\label{ch4lemma18}
 Let $k\in\mathbb{N}$ and $Z :[0, T]\times\Omega \longrightarrow \mathbb{R}^k$ be a predictable stochastic process satisfying $\mathbb{P}\left[\int_0^T\|Z_s\|^2ds<+\infty\right]=1$. Then the following inequality holds :
 \begin{eqnarray*}
 \left\|\sup_{s\in[0,T]}\left\|\int_0^sZ_ud\overline{N}_u\right\|\right\|_{L^p(\Omega, \mathbb{R})}\leq C_p\left(\int_0^T \|Z_s\|^2_{L^p(\Omega, \mathbb{R}^k)}ds\right)^{1/2},
 \end{eqnarray*}
 for all $t\in[0, T]$ and all $p\in[1,+\infty)$.
 \end{lemma}
 
 \begin{proof}
 Since $\overline{N}$ is a martingale with c\`{a}dl\`{a}g paths satisfying $d[ \overline{N},\overline{N}]_s=dN_s$, it follows from  the property of the quadratic variation (see \cite[(8.21), pp 219]{Fima}) that 
 {\small
 \begin{eqnarray}
 \label{revue}
 \mathbb{E}\left[\int_0^tZ_sd\overline{N}_s, \int_0^tZ_sd\overline{N}_s\right]=\mathbb{E}\left[\int_0^t\Vert Z_s\Vert^2dN_s\right]=\mathbb{E}\left[\int_0^t\Vert Z_s\Vert^2d\overline{N}_s\right]+\lambda \mathbb{E}\left[\int_0^t\Vert Z_s\Vert^2ds\right].
 \end{eqnarray}
 }
 The first term of \eqref{revue} vanishes as the compensated Poisson process is a martingale. Therefore, we have
 \begin{eqnarray}
 \label{ch4Bur1}
 \mathbb{E}\left[\int_0^tZ_sd\overline{N}_s, \int_0^tZ_sd\overline{N}_s\right]=\lambda \mathbb{E}\left[\int_0^t\Vert Z_s\Vert^2ds\right].
 \end{eqnarray}
 The proof follows from BDG inequality and Minkowski's inequality. In fact
 Applying Lemma \ref{ch4lemma18b} with $M_t=\sup\limits_{0\leq t\leq T}\int_0^tZ_sd\overline{N}_s$ and using \eqref{ch4Bur1} leads to
 { \small{
 \begin{eqnarray}
\label{rev1}
 \left[\mathbb{E}\left[\sup_{0\leq t\leq T}\left\|\int_0^tZ_ud\overline{N}_u\right\|^p\right]\right]^{1/p}\leq C_p\left[\mathbb{E}\left(\int_0^T \|Z_s \|^2ds\right)^{p/2}\right]^{1/p},
 \end{eqnarray}
 }}
where $C_p$ is a positive constant depending on $p$ and $\lambda$.

 Using the definition of $\|X\|_{L^p(\Omega, \mathbb{R}^d)}$ for any random variable $X$, it follows from \eqref{rev1} that 
 \begin{eqnarray}
 \left\|\sup_{s\in[0,T]}\left\|\int_0^sZ_ud\overline{N}_u\right\|\right\|_{L^p(\Omega, \mathbb{R})}\leq C_p\left\|\int_0^T\|Z_s\|^2ds\right\|^{1/2}_{L^{p/2}(\Omega, \mathbb{R})}.
 \end{eqnarray}
 Using Minkowski's inequality in its integral form  yields
 \begin{eqnarray*}
\left\|\sup_{s\in[0,T]}\left\|\int_0^sZ_ud\overline{N}_u\right\|\right\|_{L^p(\Omega, \mathbb{R})}&\leq& C_p\left(\int_0^T\left\|\|Z_s\|^2\right\|_{L^{p/2}(\Omega, \mathbb{R})}ds\right)^{1/2}\\
  &=&C_p\left(\int_0^T \|Z_s\|^2_{L^p(\Omega, \mathbb{R}^k)}ds\right)^{1/2}.
 \end{eqnarray*}
 This completes the proof of the lemma.
 \end{proof}

 \begin{lemma}\label{ch4lemma19}
 Let $k\in\mathbb{N}$,  $M\in\mathbb{N}$ and $Z^M_l : \Omega \longrightarrow \mathbb{R}^k, l\in\{0,1,\cdots, M-1\}$ be a family of mappings such that $Z^M_l$ is $\mathcal{F}_{lT/M}/\mathcal{B}(\mathbb{R}^k)$-measurable for all $l\in\{0,1,\cdots, M-1\}$, then  $\forall\; n\in\{0,1\cdots, M\}$ the following inequality holds 
 \begin{eqnarray*}
 \left\|\sup_{j\in\{0,1,\cdots, n\}}\left\|\sum_{l=0}^{j-1}Z^M_l\Delta\overline{N}^M_l\right\|\right\|_{L^p(\Omega, \mathbb{R})}\leq C_p\left(\sum_{j=0}^{n-1}||Z^M_j||^2_{L^p(\Omega, \mathbb{R}^k)}\dfrac{T}{M}\right)^{1/2},
 \end{eqnarray*}
 for all $p\geq 1$, where $C_p$ is a positive constant independent of $M$.
 \end{lemma} 
 \begin{proof}
 Let us define $\overline{Z}^M : [0, T]\times\Omega\longrightarrow \mathbb{R}^k$ such that $\overline{Z}^M_s := Z^M_l$ for all $s\in\left[\dfrac{lT}{M}, \dfrac{(l+1)T}{M}\right)$, $l\in\{0, 1,\cdots, M-1\}$.
 
 Using the definition of stochastic integral and Lemma \ref{ch4lemma18}, it follows that 
 \begin{eqnarray*}
 \left\|\sup_{j\in\{0,1,\cdots, n\}}\left\|\sum_{l=0}^{j-1}Z^M_l\Delta\overline{N}^M_l\right\|\right\|_{L^p(\Omega, \mathbb{R})}
  &= &\left\|\sup_{j\in\{0,1,\cdots, n\}}\left\|\int_0^{jT/M}\overline{Z}^M_ud\overline{N}^M_u\right\|\right\|_{L^p(\Omega, \mathbb{R})}\\
 &\leq &\left\|\sup_{s\in[0, nT/M]}\left\|\int_0^s\overline{Z}^M_ud\overline{N}_u\right\|\right\|_{L^p(\Omega, \mathbb{R}^k)}\\
 &\leq & C_p\left(\int_0^{nT/M}\|\overline{Z}^M_u\|^2_{L^p(\Omega, \mathbb{R}^k)}ds\right)^{1/2}\\
 &=&C_p\left(\sum_{j=0}^{n-1}||Z_j^M \|^2_{L^p(\Omega,\mathbb{R}^k)}\dfrac{T}{M}\right)^{1/2}.
 \end{eqnarray*}
 This completes the proof of the lemma.
 \end{proof}

   \begin{lemma}
   \label{keylemma}
   Let $Y_n^M : \Omega\longrightarrow \mathbb{R}^d$ be defined by \eqref{tamedjump} for $n\in\{0,\cdots, M\}$ and all $M\in\mathbb{N}$. The following inequality holds
   \begin{eqnarray*}
   \sup_{M\in\mathbb{N}}\sup_{n\in\{0,\cdots, M\}}\mathbb{E}\left[\|Y_n^M\|^p\right]<+\infty
   \end{eqnarray*}
   for all $p\in[1,\infty)$.
   \end{lemma}
   \begin{proof}
   Let us first represent the numerical approximation $Y^M_n$ in the following appropriate form 
   \begin{eqnarray*}
   Y^M_n&=&Y_{n-1}^M+\dfrac{\Delta tf_{\lambda}(Y^M_{n-1})}{1+\Delta t^{\alpha}\|f_{\lambda}(Y^M_{n-1})\|}+g(Y_{n-1})\Delta W^M_{n-1}+h(Y^M_{n-1})\Delta\overline{N}^M_{n-1}\\
   &=&X_0+\sum_{k=0}^{n-1}\dfrac{\Delta t f_{\lambda}(Y_k^M)}{1+\Delta t^{\alpha} \|f_{\lambda}(Y^M_k)\|}+\sum_{k=0}^{n-1}g(Y^M_k)\Delta W^M_k+\sum_{k=0}^{n-1}h(Y^M_k)\Delta\overline{N}^M_k\\
   &=& X_0+ \sum_{k=0}^{n-1}g(0)\Delta W^M_k+\sum_{k=0}^{n-1}h(0)\Delta\overline{N}^M_k+\sum_{k=0}^{n-1}\dfrac{\Delta tf_{\lambda}(Y^M_{n-1})}{1+\Delta t^{\alpha}\|f_{\lambda}(Y^M_{n-1})\|}\\
   &+&\sum_{k=0}^{n-1}(g(Y^M_k)-g(0))\Delta W^M_k+\sum_{k=0}^{n-1}(h(Y^M_k)-h(0))\Delta\overline{N}^M_k,
   \end{eqnarray*}
for all $M\in\mathbb{N}   $ and all $n\in\{0,\cdots,M\}$.

As $\alpha \in [1/2,1]$, using the inequality 
\begin{eqnarray*}
\left\|\dfrac{\Delta tf_{\lambda}(Y^M_k)}{1+\Delta t^{\alpha}\|f_{\lambda}(Y^M_k)\|}\right\|_{L^p(\Omega, \mathbb{R}^d)} <\Delta t^{1-\alpha}<T^{1-\alpha},
\end{eqnarray*} 
it follows that 
\begin{eqnarray*}
\|Y^M_n\|_{L^p(\Omega, \mathbb{R}^d)}  &\leq &\|X_0\|_{L^p(\Omega,\mathbb{R}^d)}+\left\|\sum_{k=0}^{n-1}g(0)\Delta W^M_k\right\|_{L^p(\Omega, \mathbb{R}^d)}+\left\|\sum_{k=0}^{n-1}h(0)\Delta\overline{N}^M_k\right\|_{L^p(\Omega,\mathbb{R}^d)}+MT^{1-\alpha}\\
&+&\left\|\sum_{k=0}^{n-1}(g(Y^M_k)-g(0))\Delta W^M_k\right\|_{L^p(\Omega,\mathbb{R}^d)}+\left\|\sum_{k=0}^{n-1}(h(Y^M_k)-h(0))\Delta\overline{N}^M_k\right\|_{L^p(\Omega,\mathbb{R}^d)}.
\end{eqnarray*}
Using Lemma \ref{ch4lemma16} and Lemma \ref{ch4lemma19}, it follows that 
\begin{eqnarray}
\|Y^M_n\|_{L^p(\Omega,\mathbb{R}^d)} &\leq &\|X_0\|_{L^p(\Omega,\mathbb{R}^d)}+C_p\left(\sum_{k=0}^{n-1}\sum_{i=1}^{m}\|g_i(0)\|^2\dfrac{T}{M}\right)^{1/2}+C_p\left(\sum_{k=0}^{n-1}\|h(0)\|^2\dfrac{T}{M}\right)^{1/2}\nonumber\\
&+&MT^{1-\alpha}+ C_p\left(\sum_{k=0}^{n-1}\sum_{i=1}^m\|(g_i(Y_k^M)-g_i(0))\Delta W^M_k\|^2_{L^p(\Omega, \mathbb{R}^d)}\dfrac{T}{M}\right)^{1/2}\nonumber\\ &+&C_p\left(\sum_{k=0}^{n-1} \|(h(Y_k^M)-h(0))\Delta\overline{N}^M_k \|^2_{L^p(\Omega, \mathbb{R}^d)}\dfrac{T}{M}\right)^{1/2}\nonumber\\
&\leq&\|X_0\|_{L^p(\Omega, \mathbb{R}^d)}+C_p\left(\dfrac{nT}{M}\sum_{i=1}^m \|g_i(0)\|^2\right)^{1/2}+C_p\left(\dfrac{nT}{M}||h(0)||^2\right)^{1/2}\nonumber\\
 &+&MT^{1-\alpha}+C_p\left(\sum_{k=0}^{n-1}\sum_{i=1}^m \|g_i(Y^M_k)-g_i(0)\|^2_{L^p(\Omega,\mathbb{R}^d)}\dfrac{T}{M}\right)^{1/2}\nonumber\\
&+&C_p\left(\sum_{k=0}^{n-1}\|h(Y^M_k)-h(0)\|^2_{L^p(\Omega,\mathbb{R}^d)}\dfrac{T}{M}\right)^{1/2}.
\label{ch4MB1}
\end{eqnarray}
From $\|g_i(0)\|^2\leq \|g(0)\|^2$  and the global Lipschitz condition satisfied by   $g$ and $h$, we obtain 
\begin{eqnarray*}
 \|g_i(Y^M_k)-g_i(0)\|_{L^p(\Omega, \mathbb{R}^d)}\leq  C\|Y^M_k\|_{L^p(\Omega,\mathbb{R}^d)}\\
\|h(Y^M_k)-h(0)\|_{L^p(\Omega, \mathbb{R}^d)}\leq  C\|Y^M_k\|_{L^p(\Omega,\mathbb{R}^d)}.
\end{eqnarray*}
So using \eqref{ch4MB1}, we obtain 
\begin{eqnarray*}
\|Y^M_n\|_{L^p(\Omega, \mathbb{R}^d)} &\leq & \|X_0\|_{L^p(\Omega,\mathbb{R}^d)}+C_p\sqrt{Tm}\|g(0)\|+C_p\sqrt{T}\|h(0)\|+MT^{1-\alpha}\\
&+&C_p\left(\dfrac{Tm}{M}\sum_{k=0}^{n-1}\|Y^M_k\|^2_{L^p(\Omega,\mathbb{R}^d)}\right)^{1/2}
+C_p\left(\dfrac{T}{M}\sum_{k=0}^{n-1}\|Y^M_k\|^2_{L^p(\Omega,\mathbb{R}^d)}\right)^{1/2}.
\end{eqnarray*}
Using the inequality $(a+b+c)^2\leq 3a^2+3b^2+3c^2$, it follows that :
\begin{eqnarray*}
\|Y^M_n\|^2_{L^p(\Omega,\mathbb{R}^d)}&\leq& 3\left(\|X_0\|_{L^p(\Omega,\mathbb{R}^d)}+C_p\sqrt{Tm}\|g(0)\|+C_p\sqrt{T}\|h(0)\|+MT^{1-\alpha}\right)^2\nonumber\\
&+&\dfrac{3T(C_p\sqrt{m}+C_p)^2}{M}\sum_{k=0}^{n-1} \|Y^M_k\|^2_{L^p(\Omega, \mathbb{R}^d)},
\label{ch4MB2}
\end{eqnarray*}
for all $p\in[1,\infty)$. Using the fact that 
$
\dfrac{3T(C_p\sqrt{m}+C_p)^2}{M}<3 T(C_p\sqrt{m}+C_p)^2
$  we obtain the following estimation
\begin{eqnarray}
\|Y^M_n\|^2_{L^p(\Omega,\mathbb{R}^d)}&\leq& 3 \left(\|X_0\|_{L^p(\Omega,\mathbb{R}^d)}+C_p\sqrt{Tm}\|g(0)\|+C_p\sqrt{T}\|h(0)\|+MT^{1-\alpha}\right)^2\nonumber\\
&+&3 T(C_p\sqrt{m}+C_p)^2\sum_{k=0}^{n-1}\|Y^M_k\|^2_{L^p(\Omega, \mathbb{R}^d)},
\label{ch4MB}
\end{eqnarray}

Applying  Gronwall lemma to \eqref{ch4MB} leads to
{\small {
\begin{eqnarray}
\|Y^M_n\|^2_{L^p(\Omega, \mathbb{R}^d)}\leq e^{3 T(C_p\sqrt{m}+C_p)^2}\left(\|X_0\|_{L^p(\Omega,\mathbb{R}^d)}+C_p\sqrt{Tm}\|g(0)\|+C_p\sqrt{T}\|h(0)\|+MT^{1-\alpha}\right)^2.
\label{ch4MB3}
\end{eqnarray}
}}
Taking the square root and the supremum in the both sides of \eqref{ch4MB3} leads to
{\small {
\begin{eqnarray}
\lefteqn {\sup\limits_{n\in\{0,\cdots, M\}}\|Y^M_n\|_{L^p(\Omega, \mathbb{R}^d)}}&& \nonumber\\
&\leq&  2e^{3 T(C_p\sqrt{m}+C_p)^2}\left(\|X_0\|_{L^p(\Omega,\mathbb{R}^d)}+C_p\sqrt{Tm}\|g(0)\|+C_p\sqrt{T}\|h(0)\|+MT^{1-\alpha}\right)
\label{ch4MB4}
\end{eqnarray}
}}
Unfortunately, \eqref{ch4MB4} is not enough to conclude the proof of the lemma due to the term  $M$ in the right hand side. Using the fact that $(\Omega_n^M)_n$ is a decreasing sequence and by using Holder's inequality, we obtain :
\begin{eqnarray}
\sup_{M\in\mathbb{N}}\sup_{n\in\{0,\cdots,M\}}\left\|\mathbbm{1}_{(\Omega_n^M)^c}Y^M_n\right\|_{L^p(\Omega,\mathbb{R}^d)}& \leq &\sup_{M\in\mathbb{N}}\sup_{n\in\{0,\cdots,M\}}\left\|\mathbbm{1}_{(\Omega^M_M)^c}\right\|_{L^{2p}(\Omega,\mathbb{R})}\left\|Y^M_n\right\|_{L^{2p}(\Omega,\mathbb{R}^d)}\nonumber\\
&\leq &\left(\sup_{M\in\mathbb{N}}\sup_{n\in\{0,\cdots,M\}}\left(M\left\|\mathbbm{1}_{(\Omega^M_M)^c}\right\|_{L^{2p}(\Omega,\mathbb{R})}\right)\right)\nonumber\\
&\times &\left(\sup_{M\in\mathbb{N}}\sup_{n\in\{0,\cdots,M\}}\left(M^{-1} \|Y^M_n\|_{L^{2p}(\Omega,\mathbb{R}^d)}\right)\right).
\label{ch4MB5}
\end{eqnarray}
Using  inequality \eqref{ch4MB4} yields

$
\left(\sup\limits_{M\in\mathbb{N}}\sup\limits_{n\in\{0,\cdots,M\}}\left(M^{-1}\|Y^M_n\|_{L^{2p}(\Omega,\mathbb{R}^d)}\right)\right)$
\begin{eqnarray}
\leq 2 e^{3(C_p\sqrt{m}+C_p)^2}\left(\dfrac{\|X_0\|_{L^{2p}(\Omega, \mathbb{R}^d)}}{M}+\dfrac{C_p\sqrt{Tm}\|g(0)\|+C_p\sqrt{T}\|h(0)\|}{M}+T^{1-\alpha}\right)\nonumber\\
\leq 2 e^{3(C_p\sqrt{m}+C_p)^2}\left(\|X_0\|_{L^{2p}(\Omega,\mathbb{R}^d)}+C_p\sqrt{Tm}\|g(0)\|+C_p\sqrt{T}\|h(0)\|+T^{1-\alpha}\right)<+\infty,
\label{ch4MB6}
\end{eqnarray}
for all $p\geq 1$.
From the relation 
\begin{eqnarray*}
\left\|\mathbbm{1}_{(\Omega^M_M)^c}\right\|_{L^{2p}(\Omega, \mathbb{R})}=
\mathbb{E}\left[\mathbbm{1}_{(\Omega^M_M)^c}\right]^{1/2p}=
\mathbb{P}\left[(\Omega^M_M)^c\right]^{1/2p}
\end{eqnarray*}
it follows using Lemma  \ref{ch4lemma13} that
\begin{eqnarray}
\sup_{M\in\mathbb{N}}\sup_{n\in\{0,\cdots, M\}}\left(M\left\|\mathbbm{1}_{(\Omega^M_M)^c}\right\|_{L^{2p}(\Omega,\mathbb{R}}\right)=\sup_{M\in\mathbb{N}}\sup_{n\in\{0,\cdots, M\}}\left(M^{2p}\mathbb{P}\left[(\Omega^M_M)^c\right]\right)^{1/2p}<+\infty,
\label{ch4MB7}
\end{eqnarray}
for all $p\geq 1$. 

So  plugging  \eqref{ch4MB6} and \eqref{ch4MB7} in \eqref{ch4MB5}  leads to 
\begin{eqnarray}
\sup_{M\in\mathbb{N}}\sup_{n\in\{0,\cdots, M\}}\left\|\mathbbm{1}_{(\Omega^M_n)^c}Y^M_n\right\|_{L^p(\Omega,\mathbb{R})}<+\infty.
\label{ch4MB8}
\end{eqnarray}
Futhermore, we have 
\begin{eqnarray}
\sup_{M\in\mathbb{N}}\sup_{n\in\{0,\cdots, M\}}\left\|Y^M_n\right\|_{L^p(\Omega,\mathbb{R}^d)} &\leq &\sup_{M\in\mathbb{N}}\sup_{n\in\{0,\cdots, M\}}\left\|\mathbbm{1}_{(\Omega^M_n)}Y^M_n\right\|_{L^p(\Omega,\mathbb{R}^d)}\nonumber\\
&+&\sup_{M\in\mathbb{N}}\sup_{n\in\{0,\cdots, M\}}\left\|\mathbbm{1}_{(\Omega^M_n)^c}Y^M_n\right\|_{L^p(\Omega,\mathbb{R}^d)}.
\label{ch4MB9}
\end{eqnarray}
From  \eqref{ch4MB8}, the second term  of inequality \eqref{ch4MB9} is bounded, while using Lemma \ref{ch4lemma2} and Lemma \ref{ch4lemma12} we  have
\begin{eqnarray}
\sup_{M\in\mathbb{N}}\sup_{n\in\{0,\cdots, M\}}\left\|\mathbbm{1}_{(\Omega^M_n)}Y^M_n\right\|_{L^p(\Omega,\mathbb{R}^d)}\leq \sup_{M\in\mathbb{N}}\sup_{n\in\{0,\cdots, M\}}\left\|D_n^M\right\|_{L^p(\Omega,\mathbb{R})}<+\infty.
\label{ch4MB10}
\end{eqnarray}
Finally plugging \eqref{ch4MB8} and \eqref{ch4MB10} in \eqref{ch4MB9} leads to
\begin{eqnarray*}
\sup_{M\in\mathbb{N}}\sup_{n\in\{0,\cdots, M\}}\left\|Y^M_n\right\|_{L^p(\Omega,\mathbb{R}^d)}<+\infty.
\end{eqnarray*}
\end{proof}

\begin{lemma} \label{ch4lemma21}
Let $Y_n^M$ be defined by \eqref{tamedjump} for all $M\in\mathbb{N}$ and all $n\in\{0,1,\cdots, M\}$, then we have 
\begin{eqnarray*}
\sup_{M\in\mathbb{N}}\sup_{n\in\{0,1,\cdots, M\}}\left(\mathbb{E}\left[\|f_{\lambda}(Y_n^M)\|^p\right]\vee \mathbb{E}\left[\left\|g(Y_n^M)\right\|^p\right]\vee \mathbb{E}\left[\left\|h(Y_n^M)\right\|^p\right]\right)<+\infty,
\end{eqnarray*}
for all $p\in[1,\infty)$.
\end{lemma}   

\begin{proof}
As $f_{\lambda}$ satisfies the polynomial growth condition,  for all $x\in\mathbb{R}^d$ we have 
\begin{eqnarray*}
\|f_{\lambda}(x)\|\leq C(K+\|x\|^c)\|x\|+\|f_{\lambda}(0)\|=C K\|x\|+C\|x\|^{c+1}+\|f_{\lambda}(0)\|.
\end{eqnarray*}
$\bullet$ If $\|x\|\leq 1$, then $C K\|x\|\leq CK$, hence 
\begin{eqnarray}
\|f_{\lambda}(x)\|&\leq& CK+C\|x\|^{c+1}+\|f_{\lambda}(0)\|\nonumber\\
&\leq & KC+KC\|x\|^{c+1}+C+C\|x\|^{c+1}+\|f_{\lambda}(0)\|+\|f_{\lambda}(0)\|x\|^{c+1}\nonumber\\
&=&(KC+C+\|f_{\lambda}(0)\|)(1+\|x\|^{c+1}).
\label{ch4eq1}
\end{eqnarray}
$\bullet$ If $\|x\|\geq 1$, then $\|x\|\leq \|x\|^{c+1}$,  hence
\begin{eqnarray}
\|f_{\lambda}(x)\|&\leq & KC\|x\|^{c+1}+C\|x\|^{c+1}+\|f_{\lambda}(0)\|\nonumber\\
&\leq & KC+KC\|x\|^{c+1}+C+C\|x\|^{c+1}+\|f_{\lambda}(0)\|+\|f_{\lambda}(0)\|\|x\|^{c+1}\nonumber\\
&=&(KC+C+\|f_{\lambda}(0)\|)(1+\|x\|^{c+1}).
\label{ch4eq2}
\end{eqnarray}
So it follows from \eqref{ch4eq1} and \eqref{ch4eq2} that 
\begin{eqnarray}
\|f_{\lambda}(x)\| \leq (KC+C+\|f_{\lambda}(0)\|)(1+\|x\|^{c+1}), \hspace{0.5cm} \text{for all} \hspace{0.3cm} x\in\mathbb{R}^d.
\label{ch4eq3}
\end{eqnarray}
Using  inequality \eqref{ch4eq3} and \lemref{keylemma}, it follows that
\begin{eqnarray*}
\sup_{M\in\mathbb{N}}\sup_{n\in\{0,\cdots,M\}}\left\|f_{\lambda}(Y_n^M)\right\|_{L^p(\Omega, \mathbb{R}^d)}&\leq& (KC+C+||f_{\lambda}(0)||)\\
&\times&\left(1+\sup_{M\in\mathbb{N}}\sup_{n\in\{0,\cdots, M\}}\left\|Y^M_n\right\|^{c+1}_{L^{p(c+1)}(\Omega, \mathbb{R}^d)}\right)\\
&<&+\infty,
\end{eqnarray*}
for all $p\in[1,\infty)$. In other hand, using the global Lipschitz condition satisfied by   $g$ and $h$, it follows that 
\begin{eqnarray}
\|g(x)\| \leq C\|x\|+ \|g(0)\| \hspace{0.2cm} \text{and}\hspace{0.2cm} \|h(x)\|\leq C\|x\|+\|h(0\|.
\label{ch4stron2}
\end{eqnarray}

Using once again   \lemref{keylemma}, it follows from \eqref{ch4stron2} that 
\begin{eqnarray*}
 \sup_{M\in\mathbb{N}, n\in\{0, \cdots, M\}}\left\|g(Y^M_n)\right\|_{L^p(\Omega, \mathbb{R}^d)}\leq\|g(0)\|+C\sup_{M\in\mathbb{N}}\sup_{n\in\{0,\cdots, M\}}\left\|Y^M_n\right\|_{L^p(\Omega, \mathbb{R}^d)}<+\infty,
\end{eqnarray*}
for all $p\in[1,\infty)$. Using the same argument as for $g$ the following holds 
\begin{eqnarray*}
\sup_{M\in\mathbb{N}}\sup_{n\in\{0,\cdots, M\}}\left\|h(Y^M_n)\right\|_{L^p(\Omega,\mathbb{R}^d)}<+\infty, 
\end{eqnarray*}
for all $p\in[1, +\infty)$.
This complete the proof of  Lemma \ref{ch4lemma21}.
\end{proof}
 
 In  the sequel,  for all $s\in[0,T]$ we denote by $\lfloor s\rfloor$ the greatest grid point less than $s$.
 \begin{lemma}\label{ch4lemma22}
 Let $\overline{Y}^M_t$ be the time continuous approximation given by \eqref{dtamedjump},  there exists a  constant  $C_p$ such that the following inequalities hold
 \begin{eqnarray}
 \label{eq1lemma18}
 &&\sup_{t\in[0,T]}\left\|\overline{Y}^M_t-\overline{Y}^M_{\lfloor t\rfloor}\right\|_{L^p(\Omega, \mathbb{R}^d)}\leq C_p\Delta t^{1/2},\\
 &&\sup_{M\in\mathbb{N}}\sup_{t\in[0,T]}\left\|\overline{Y}^M_t\right\|_{L^p(\Omega, \mathbb{R}^d)}<\infty,\\
 && \sup_{t\in[0,T]}\left\|f_{\lambda}(\overline{Y}^M_t)-f_{\lambda}(\overline{Y}^M_{\lfloor t\rfloor})\right\|_{L^p(\Omega, \mathbb{R}^d)}\leq C_p\Delta t^{1/2}.
 \end{eqnarray}
 for all $p \in [1, \infty)$.
 \end{lemma}
 
 \begin{proof}
  Using Lemma \ref{ch4lemma18}, Lemma \ref{ch4lemma15} and the time continuous approximation \eqref{dtamedjump}, it follows that
  \begin{eqnarray}   
 \lefteqn{\sup\limits_{t\in[0,T]}\left\|\overline{Y}^M_t-\overline{Y}^M_{\lfloor t\rfloor}\right\|_{L^p(\Omega, \mathbb{R}^d)}} \nonumber\\
  &\leq &\dfrac{T}{M}\left(\sup_{t\in[0,T]}\left\|\dfrac{f_{\lambda}(\overline{Y}^M_{\lfloor t\rfloor})}{1+\Delta t^{\alpha}\|f_{\lambda}(\overline{Y}^M_{\lfloor t\rfloor})\|}\right\|_{L^p(\Omega, \mathbb{R}^d)}\right)+\sup_{t\in[0, T]}\left\|\int^t_{\lfloor t\rfloor} g(\overline{Y}^M_{\lfloor t\rfloor})dW_s\right\|_{L^p(\Omega, \mathbb{R}^d)}\nonumber\\
 &+& \sup_{t\in[0,T]}\left\|\int^t_{\lfloor t\rfloor}h(\overline{Y}^M_{\lfloor t\rfloor})d\overline{N}_s\right\|_{L^p(\Omega, \mathbb{R}^d)}\nonumber\\
 &\leq &\dfrac{T}{\sqrt{M}}\left(\sup_{n\in\{0,\cdots, M\}}\|f_{\lambda}(Y^M_n)\|_{L^p(\Omega, \mathbb{R}^d)}\right)+\sup_{t\in[0,T]}\left(\dfrac{T}{M}\sum_{i=1}^m\int^t_{\lfloor t\rfloor}\|g_i(\overline{Y}^M_s)\|^2_{L^p(\Omega, \mathbb{R}^d)}ds\right)^{1/2}\nonumber\\
 &+& \sup_{t\in[0,T]}\left(\dfrac{TC_p}{M}\int^t_{\lfloor t\rfloor}\|h(\overline{Y}^M_s)\|^2_{L^p(\Omega, \mathbb{R}^d)}ds\right)^{1/2}\nonumber\\ 
 &\leq &\dfrac{T}{\sqrt{M}}\left(\sup_{n\in\{0,\cdots, M\}}\|f_{\lambda}(Y^M_n)\|_{L^p(\Omega, \mathbb{R}^d)}\right)+\dfrac{\sqrt{Tm}}{\sqrt{M}}\left(\sup_{i\in\{1,\cdots, m\}}\sup_{n\in\{0,\cdots, M\}}\|g_i(Y^M_n)\|_{L^p(\Omega, \mathbb{R}^d)}\right)\nonumber\\
 &+&\dfrac{C_p\sqrt{T}}{\sqrt{M}}\left(\sup_{n\in\{0, \cdots,M\}}\|h(Y^M_n)\|_{L^p(\Omega, \mathbb{R}^d)}\right),
 \label{ch4Thcontinous}
 \end{eqnarray}
 for all $M\in\mathbb{N}$.
 
  Using  inequality \eqref{ch4Thcontinous}  and Lemma \ref{ch4lemma21}, it follows  that 
\begin{eqnarray}
\left[\sup_{t\in[0,T]}\left\|\overline{Y}^M_t-\overline{Y}^M_{\lfloor t\rfloor}\right\|_{L^p(\Omega, \mathbb{R}^d)}\right]<C_p\Delta t^{1/2},
\label{Ch4bon1}
\end{eqnarray}
for all $p\in[1,\infty)$. 
 
 Using  the inequalities \eqref{Ch4bon1}, $\|a\|\leq \|a-b\|+\|b\|$ for all $a,b\in\mathbb{R}^d$ and \lemref{keylemma} it follows that 
\begin{eqnarray*}
\sup_{t\in[0,T]}\|\overline{Y}^M_t\|_{L^p(\Omega, \mathbb{R}^d)}&\leq&\left[\sup_{t\in[0,T]}\left\|\overline{Y}^M_t-\overline{Y}^M_{\lfloor t\rfloor}\right\|_{L^p(\Omega, \mathbb{R}^d)}\right]+\sup_{t\in[0,T]}\left\|\overline{Y}^M_{\lfloor t\rfloor}\right\|_{L^p(\Omega, \mathbb{R}^d)}\\
&\leq&\dfrac{C_p}{M^{1/2}}+\sup_{t\in[0,T]}\left\|\overline{Y}^M_{\lfloor t\rfloor}\right\|_{L^p(\Omega, \mathbb{R}^d)}\\
&<&C_pT^{1/2}+\sup_{t\in[0,T]}\left\|\overline{Y}^M_{\lfloor t\rfloor}\right\|_{L^p(\Omega, \mathbb{R}^d)}\\
&<&\infty,
\end{eqnarray*}
for all $p\in[1,+\infty)$ and all $M\in\mathbb{N}$.
Further, using the polynomial growth condition 
\begin{eqnarray*}
\|f_{\lambda}(x)-f_{\lambda}(y)\|\leq C(K+\|x\|^c+\|y\|^c)\|x-y\|,
\end{eqnarray*}
for all $x, y\in\mathbb{R}^d$, it follows using Holder  inequality that 
\begin{eqnarray}
\sup_{t\in[0,T]}\|f_{\lambda}(\overline{Y}^M_t)-f_{\lambda}(\overline{Y}^M_{\lfloor t\rfloor})\|_{L^p(\Omega, \mathbb{R}^d)} &\leq &C\left(K+2\sup_{t\in[0,T]}\|\overline{Y}^M_t\|^c_{L^{2pc}(\Omega, \mathbb{R}^d)}\right)\nonumber\\
&\times & \left(\sup_{t\in[0,T]}\|\overline{Y}^M_t-\overline{Y}^M_{\lfloor t\rfloor}\|_{L^{2p}(\Omega, \mathbb{R}^d)}\right).
\label{ch4Thfcontinou}
\end{eqnarray}
Using \eqref{ch4Thfcontinou} and \eqref{eq1lemma18},   the following inequality holds  
\begin{eqnarray}
\left[\sup_{t\in[0,T]}\|f_{\lambda}(\overline{Y}^M_t)-f_{\lambda}(\overline{Y}^M_{\lfloor t\rfloor})\|_{L^p(\Omega, \mathbb{R}^d)}\right]<C_p\Delta t^{1/2},
\label{ch4Thffinal}
\end{eqnarray} 
for all $p\in[1,\infty)$.
\end{proof}
 
 Now we are ready to give the proof of \thmref{ch4theorem1}.
 
 \subsubsection{Main part of the proof of Theorem \ref{ch4theorem1} for $\chi^M_t= \overline{Y}^M_t$}
 Recall that for $s\in[0,T]$, $\lfloor s\rfloor$ denote the greatest grid point less than $s$.
 The time continuous solution \eqref{dtamedjump} can be written in its integral form as bellow 
 \begin{eqnarray}
 \overline{Y}^M_s=X_0+\int_0^s\dfrac{f_{\lambda}(\overline{Y}^M_{\lfloor u\rfloor})}{1+\Delta t^{\alpha}\|f_{\lambda}(\overline{Y}^M_{\lfloor u\rfloor})\|}du+ \int_0^s g(\overline{Y}^M_{\lfloor u\rfloor})dW_u+\int_0^s h(\overline{Y}^M_{\lfloor u\rfloor})d\overline{N}_u,
 \label{ch4continoussol2}
 \end{eqnarray}
 for all $s\in[0, T]$ almost surely and all $M\in\mathbb{N}$.
 
 Let us estimate first the quantity  $\|X_s-\overline{Y}^M_s\|^2$, where $X_s$ is the exact solution of \eqref{model}.
 \begin{eqnarray*}
 X_s-\overline{Y}_s&=&\int_0^s\left(f_{\lambda}(X_u)-\dfrac{f_{\lambda}(\overline{Y}^M_{\lfloor u\rfloor})}{1+\Delta t^{\alpha}\|f_{\lambda}(\overline{Y}^M_{\lfloor u\rfloor})\|}\right)du+\int_0^s\left(g(X_u)-g(\overline{Y}^M_{\lfloor u\rfloor})\right)dW_u\\
 &+&\int_0^s\left(h(X_u)-h(\overline{Y}^M_{\lfloor u\rfloor})\right)d\overline{N}_u.
 \end{eqnarray*}
 Using the relation $d\overline{N}_u=dN_u-\lambda du$, it follows that
 \begin{eqnarray*}
 X_s-\overline{Y}_s&=&\int_0^s\left[\left(f_{\lambda}(X_u)-\dfrac{f_{\lambda}(\overline{Y}^M_{\lfloor u\rfloor})}{1+\Delta t^{\alpha} \|f_{\lambda}(\overline{Y}^M_{\lfloor u\rfloor})\|}\right)-\lambda \left(h(X_u)-h(\overline{Y}^M_{\lfloor u\rfloor})\right)\right]du\\
 &+&\int_0^s\left(g(X_u)-g(\overline{Y}^M_{\lfloor u\rfloor})\right)dW_u+\int_0^s\left(h(X_u)-h(\overline{Y}^M_{\lfloor u\rfloor})\right)dN_u.
 \end{eqnarray*}
 The function $ k :\mathbb{R}^d\longrightarrow \mathbb{R}$, $x \longmapsto \|x\|^2$ is twice differentiable. Applying It\^{o}'s formula for jumps process (\cite[pp. 6-9]{Oks}) to the process $X_s-\overline{Y}^M_s$ leads to 
 \begin{eqnarray*}
 \left\|X_s-\overline{Y}^M_s\right\|^2&=&2\int_0^s\left<X_u-\overline{Y}^M_u, f_{\lambda}(X_u)-\dfrac{f_{\lambda}(\overline{Y}^M_{\lfloor u\rfloor})}{1+\Delta t^{\alpha}\|f_{\lambda}(\overline{Y}^M_{\lfloor u\rfloor})\|}\right>du\\
 &-&2\lambda\int_0^s\left<X_u-\overline{Y}^M_u, h(X_u)-h(\overline{Y}^M_{\lfloor u\rfloor})\right>du +\sum_{i=1}^m\int_0^s\|g_i(X_u)-g_i(\overline{Y}^M_{\lfloor u\rfloor})\|^2du\\
 &+&2\sum_{i=1}^m\int_0^s\left<X_u-\overline{Y}^M_u, g_i(X_u)-g_i(\overline{Y}^M_{\lfloor u\rfloor})\right>dW^i_u\\
 &+& \int_0^s\left[\|X_u-\overline{Y}^M_u+h(X_u)-h(\overline{Y}^M_{\lfloor u\rfloor})\|^2-\|X_u-\overline{Y}^M_u\|^2\right]dN_u.
 \end{eqnarray*}
 Using again the relation $dN_u=d\overline{N}_u+\lambda du$ leads to 
 \begin{eqnarray}
 \left\|X_s-\overline{Y}^M_s\right\|^2&=&2\int_0^s\left<X_u-\overline{Y}^M_u, f_{\lambda}(X_u)-\dfrac{f_{\lambda}(\overline{Y}^M_{\lfloor u\rfloor})}{1+\Delta t^{\alpha}\|f_{\lambda}(\overline{Y}^M_{\lfloor u\rfloor})\|}\right>du\nonumber\\
 &-&2\lambda\int_0^s\left<X_u-\overline{Y}^M_u, h(X_u)-h(\overline{Y}^M_{\lfloor u\rfloor})\right>du+\sum_{i=1}^m\int_0^s\|g_i(X_u)-g_i(\overline{Y}^M_{\lfloor u\rfloor})\|^2du\nonumber\\
 &+&2\sum_{i=1}^m\int_0^s\left<X_u-\overline{Y}^M_u, g_i(X_u)-g_i(\overline{Y}^M_{\lfloor u\rfloor})\right>dW^i_u\nonumber\\
 &+& \int_0^s\left[\|X_u-\overline{Y}^M_u+h(X_u)-h(\overline{Y}^M_{\lfloor u\rfloor})\|^2-\|X_u-\overline{Y}^M_u\|^2\right]d\overline{N}_u\nonumber\\
 &+&\lambda\int_0^s\left[\|X_u-\overline{Y}^M_u+h(X_u)-h(\overline{Y}^M_{\lfloor u \rfloor})\|^2-\|X_u-\overline{Y}^M_u\|^2\right]du\nonumber\\
 &=&A_1+A_2+A_3+A_4+A_5+A_6.
 \label{ch4Th1}
 \end{eqnarray}
 In the next step, we give some useful estimations of $A_1, A_2, A_3$ and $A_6$. 
 \begin{eqnarray*}
 A_1& : =&2\int_0^s\left<X_u-\overline{Y}^M_u,f_{\lambda}(X_u)-\dfrac{f_{\lambda}(\overline{Y}_{\lfloor u\rfloor})}{1+\Delta t^{\alpha}\|f_{\lambda}(\overline{Y}^M_{\lfloor u\rfloor})\|}\right>du\\
 &=&2\int_0^s\left\langle X_u-\overline{Y}^M_u,f_{\lambda}(X_u)-f_{\lambda}(\overline{Y}^M_u)\right\rangle du\\
 &+&2\int_0^s\left<X_u-\overline{Y}^M_u,f_{\lambda}(\overline{Y}^M_u)-\dfrac{f_{\lambda}(\overline{Y}^M_{\lfloor u\rfloor})}{1+\Delta t^{\alpha}\|f_{\lambda}(\overline{Y}^M_{\lfloor u\rfloor})\|}\right>du,\\
 &=& A_{11}+A_{12}.
 \end{eqnarray*}
 Using the one-sided Lipschitz condition satisfied by $f_{\lambda}$ leads to
 \begin{eqnarray}
 A_{11} &: =&2\int_0^s\left\langle X_u-\overline{Y}^M_u,f_{\lambda}(X_u)-f_{\lambda}(\overline{Y}^M_u)\right\rangle du\nonumber\\
 &\leq& 2C\int_0^s\|X_u-\overline{Y}^M_u\|^2du.
 \label{ch4ThA11}
 \end{eqnarray}
 Moreover, using the inequality $2\langle a, b\rangle\leq 2\|a\|\|b\|\leq \Vert a \Vert^2+\Vert b \Vert^2$ for all  $a,b \in \mathbb{R}^d$ leads to
 \begin{eqnarray}
 A_{12}&=& 2\int_0^s\left<X_u-\overline{Y}^M_u,f_{\lambda}(\overline{Y}^M_u)-\dfrac{f_{\lambda}(\overline{Y}^M_{\lfloor u\rfloor})}{1+\Delta t^{\alpha}\|f_{\lambda}(\overline{Y}^M_{\lfloor u\rfloor})\|}\right>du\nonumber\\
 &=&2\int_0^s\left<X_u-\overline{Y}^M_u, f_{\lambda}(\overline{Y}^M_u)-f_{\lambda}(\overline{Y}^M_{\lfloor u\rfloor})\right>ds\nonumber\\
 &+&2\Delta t^{\alpha}\int_0^s\left<X_u-\overline{Y}^M_u, \dfrac{f_{\lambda}(\overline{Y}^M_{\lfloor u\rfloor})\|f_{\lambda}(\overline{Y}^M_{\lfloor u\rfloor})\|}{1+\Delta t^{\alpha}\|f_{\lambda}(\overline{Y}^M_{\lfloor u\rfloor})\|}\right>du\nonumber\\
 &\leq &\int_0^s\|X_u-\overline{Y}^M_u\|^2du+\int_0^s\|f_{\lambda}(\overline{Y}^M_u)-f_{\lambda}(\overline{Y}^M_{\lfloor u\rfloor})\|^2du\nonumber\\
 &+&\int_0^s\|X_u-\overline{Y}^M_u\|^2du+\dfrac{T^{2\alpha}}{M^{2\alpha}}\int_0^s\|f_{\lambda}(\overline{Y}^M_{\lfloor u\rfloor})\|^4du\nonumber\\
 &\leq &2\int_0^s\|X_u-\overline{Y}^M_u\|^2du+\int_0^s\|f_{\lambda}(\overline{Y}^M_u)-f_{\lambda}(\overline{Y}_{\lfloor u\rfloor})\|^2du\nonumber\\
 &+&\dfrac{T^{2\alpha}}{M^{2\alpha}}\int_0^s\|f_{\lambda}(\overline{Y}^M_{\lfloor u\rfloor})\|^4du.
 \label{ch4ThA12}
 \end{eqnarray}
 Combining \eqref{ch4ThA11} and \eqref{ch4ThA12} give the following estimation of $A_1$
 \begin{eqnarray}
 A_1 &\leq & (2C+2)\int_0^s\|X_u-\overline{Y}^M_u\|^2du+\int_0^s\|f_{\lambda}(\overline{Y}^M_u)-f_{\lambda}(\overline{Y}_{\lfloor u\rfloor})\|^2du\nonumber\\
 &+&\dfrac{T^{2\alpha}}{M^{2\alpha}}\int_0^s\|f_{\lambda}(\overline{Y}^M_{\lfloor u\rfloor})\|^4du.
 \label{ch4ThA1}
 \end{eqnarray}
 Using again the inequality $2\langle a, b\rangle\leq 2\|a\| \|b\|\leq \Vert a \Vert^2+ \Vert b\Vert^2$ for all $a, b \in \mathbb{R}^d$  and the global Lipschitz condition satisfied by $h$ leads to
 \begin{eqnarray}
 A_2 & : =&-2\lambda\int_0^s\left<X_u-\overline{Y}^M_u, h(X_u)-h(\overline{Y}^M_{\lfloor u\rfloor})\right>du\nonumber\\
 &=&-2\lambda\int_0^s\left<X_u-\overline{Y}^M_u, h(X_u)-h(\overline{Y}^M_u)\right>du-2\lambda\int_0^s\left<X_u-\overline{Y}^M_u, h(\overline{Y}^M_u)-h(\overline{Y}^M_{\lfloor u\rfloor})\right>du\nonumber\\
 &\leq &(2\lambda+\lambda C^2)\int_0^s\|X_u-\overline{Y}^M_u\|^2du+\lambda C^2\int_0^s\|\overline{Y}^M_u-\overline{Y}^M_{\lfloor u\rfloor}\|^2du.
 \label{ch4ThA2}
 \end{eqnarray}
 Using the inequalities $\|g_i(x)-g_i(y)\|\leq \|g(x)-g(y)\|$ and $\Vert a+b \Vert ^2\leq 2\Vert a \Vert^2+2\Vert b\Vert^2$ for all  $a,b \in \mathbb{R}^d$ and the global Lipschitz condition  we have 
 \begin{eqnarray}
 A_3 &: =&\sum_{i=1}^m\int_0^s\|g_i(X_u)-g_i(\overline{Y}^M_{\lfloor u\rfloor})\|^2du\nonumber\\
 &\leq &m\int_0^s\|g(X_u)-g(\overline{Y}^M_{\lfloor u\rfloor})\|^2du\nonumber\\
 &\leq &m\int_0^s\|g(X_u)-g(\overline{Y}^M_u)+g(\overline{Y}^M_u)-g(\overline{Y}^M_{\lfloor u\rfloor})\|^2du\nonumber\\
 &\leq &2m\int_0^s\|g(X_u)-g(\overline{Y}^M_u)\|^2du+2m\int_0^s\|g(\overline{Y}^M_u)-g(\overline{Y}^M_{\lfloor u\rfloor})\|^2du\nonumber\\
 &\leq& 2mC^2\int_0^s\|X_u-\overline{Y}^M_u\|^2du+2mC^2\int_0^s\|\overline{Y}^M_u-\overline{Y}^M_{\lfloor u\rfloor}\|^2du.
 \label{ch4ThA3}
 \end{eqnarray}
 Using once again inequality $\Vert a+b \Vert^2\leq 2\|a\|^2+\|b\|^2$ for all $ a, b \in \mathbb{R}^d$ we obtain the following estimation of $A_6$ 
 \begin{eqnarray}
 A_6 & : =&\lambda\int_0^s\left[\|X_u-\overline{Y}^M_u+h(\overline{Y}^M_u)-h(\overline{Y}^M_{\lfloor u\rfloor})\|^2-\|X_u-\overline{Y}^M_u\|^2\right]du\nonumber\\
 &\leq &\lambda\int_0^s\|X_u-\overline{Y}^M_u\|^2du+2\lambda\int_0^s\|h(X_u)-h(\overline{Y}^M_{\lfloor u\rfloor})\|^2du\nonumber\\
 &\leq &\lambda\int_0^s\|X_u-\overline{Y}^M_u\|^2du+4\lambda\int_0^s\|h(X_u)-h(\overline{Y}^M_u)\|^2du\nonumber\\
 &+ &4\lambda\int_0^s\|h(\overline{Y}^M_u)-h(\overline{Y}^M_{\lfloor u\rfloor})\|^2du\nonumber\\
 &\leq &(\lambda+4\lambda C^2)\int_0^s\|X_u-\overline{Y}^M_u\|^2du+4\lambda C^2\int_0^s\|\overline{Y}^M_u-\overline{Y}^M_{\lfloor u\rfloor}\|^2du.
 \label{ch4ThA6}
 \end{eqnarray}
 Inserting \eqref{ch4ThA1}, \eqref{ch4ThA2}, \eqref{ch4ThA3} and \eqref{ch4ThA6} in \eqref{ch4Th1} we obtain
  \begin{eqnarray*}
  \left\|X_s-\overline{Y}^M_s\right\|^2&\leq &(2C+2+2mC^2+3\lambda+5\lambda C^2)\int_0^s\|X_u-\overline{Y}^M_u\|^2du\nonumber\\
  &+&(2mC^2+5\lambda C^2)\int_0^s\|\overline{Y}^M_u-\overline{Y}^M_{\lfloor u\rfloor}\|^2du\nonumber\\
  &+&\int_0^s\|f_{\lambda}(\overline{Y}^M_u)-f_{\lambda}(\overline{Y}^M_{\lfloor u\rfloor})\|^2du+\dfrac{T^{2\alpha}}{M^{2\alpha}}\int_0^s\|f_{\lambda}(\overline{Y}^M_{\lfloor u\rfloor})\|^4du\nonumber\\
  &+&2\sum_{i=1}^m\int_0^s\left<X_u-\overline{Y}^M_u, g_i(X_u)-g_i(\overline{Y}^M_{\lfloor u\rfloor})\right>dW^{i}_u\nonumber\\
  &+&\int_0^s\left[\|X_u-\overline{Y}^M_u+h(X_u)-h(\overline{Y}^M_{\lfloor u\rfloor})\|^2-\|X_u-\overline{Y}^M_u||^2\right]d\overline{N}_u.
  \end{eqnarray*}
Taking the supremum in both sides of the previous inequality leads to
\begin{eqnarray}
\sup_{s\in[0,t]}\left\|X_s-\overline{Y}^M_s\right\|^2&\leq &(2C+2+2mC^2+3\lambda+5\lambda C^2)\int_0^t\|X_u-\overline{Y}^M_u\|^2du\nonumber\\
  &+&(2mC^2+5\lambda C^2)\int_0^t\|\overline{Y}^M_u-\overline{Y}^M_{\lfloor u\rfloor}\|^2du\nonumber\\
  &+&\int_0^t\|f_{\lambda}(\overline{Y}^M_u)-f_{\lambda}(\overline{Y}^M_{\lfloor u\rfloor})\|^2du+\dfrac{T^{2\alpha}}{M^{2\alpha}}\int_0^t\|f_{\lambda}(\overline{Y}^M_{\lfloor u\rfloor})\|^4du\nonumber\\
  &+&2\sup_{s\in[0,t]}\left|\sum_{i=1}^m\int_0^s\left<X_u-\overline{Y}^M_u, g_i(X_u)-g_i(\overline{Y}^M_{\lfloor u\rfloor})\right>dW^{i}_u\right|\nonumber\\
  &+&\sup_{s\in[0,t]}\left|\int_0^s\left[\|X_u-\overline{Y}^M_u+h(X_u)-h(\overline{Y}^M_{\lfloor u\rfloor})\|^2\right]d\overline{N}_u\right|\nonumber\\
  &+&\sup_{s\in[0,t]}\left|\int_0^s\|X_u-\overline{Y}^M_u\|^2d\overline{N}_u\right|.
  \label{ch4Th2}
\end{eqnarray}
 Using Lemma \ref{ch4lemma15} we  have the following estimation for all $p\geq 2$
 \begin{eqnarray*}
 B_1& :=&\left\|2\sup_{s\in[0,t]}\left|\sum_{i=1}^m\int_0^s\left<X_u-\overline{Y}^M_u, g_i(X_u)-g_i(\overline{Y}^M_{\lfloor u\rfloor})\right>dW^i_u\right|\right\|_{L^{p/2}(\Omega, \mathbb{R})}\\
 &\leq &C_p\left(\int_0^t\sum_{i=1}^m\left\|\left<X_u-\overline{Y}^M_u, g_i(X_u)-g_i(\overline{Y}^M_{\lfloor u\rfloor})\right>\right\|^2_{L^{p/2}(\Omega, \mathbb{R})}ds\right)^{1/2}.
 \end{eqnarray*}
Moreover, using the inequalities $ a b  \leq \dfrac{a^2}{2}+\dfrac{b^2}{2}$ and
$(a+b)^2\leq 2 a^2+2 b^2$ for all  $a,b \in \mathbb{R}$,
we  have the following estimations  for all $p\geq 2$
 \begin{eqnarray}
 B_1 &\leq&  C_p\left(\int_0^t\sum_{i=1}^m\left\|\left<X_u-\overline{Y}^M_u, g_i(X_u)-g_i(\overline{Y}^M_{\lfloor u\rfloor})\right>\right\|^2_{L^{p/2}(\Omega, \mathbb{R})}du\right)^{1/2}\nonumber\\ 
 &\leq &C_p\left(\int_0^t\sum_{i=1}^m\|X_u-\overline{Y}_u^M\|^2_{L^p(\Omega, \mathbb{R})}\|g_i(X_u)-g_i(\overline{Y}^M_{\lfloor u\rfloor})\|^2_{L^p(\Omega, \mathbb{R}^d)}du\right)^{1/2}\nonumber\\
 &\leq &\dfrac{C_p}{\sqrt{2}}\left(\sup_{s\in[0,t]}\|X_s-\overline{Y}^M_s\|_{L^p(\Omega, \mathbb{R}^d)}\right)\left(2C^2m\int_0^t\|X_s-\overline{Y}^M_{\lfloor s\rfloor}\|^2_{L^p(\Omega, \mathbb{R}^d)}ds\right)^{1/2}\nonumber\\
 &\leq &\dfrac{1}{4}\sup_{s\in[0,t]}\|X_s-\overline{Y}^M_s\|^2_{L^p(\Omega, \mathbb{R}^d)}+C_p^2 m\int_0^t\|X_s-\overline{Y}^M_{\lfloor s\rfloor}\|^2_{L^p(\Omega,\mathbb{R}^d)}ds\nonumber\\
 &\leq &\dfrac{1}{4}\sup_{s\in[0,t]}\|X_s-\overline{Y}^M_s\|^2_{L^p(\Omega, \mathbb{R}^d)}+2C_p^2m\int_0^t\|X_s-\overline{Y}^M_s\|^2_{L^p(\Omega,\mathbb{R}^d)}ds\nonumber\\
 &+&2 C_p^2m\int_0^t\|\overline{Y}^M_s-\overline{Y}^M_{\lfloor s\rfloor}\|^2_{L^p(\Omega,\mathbb{R}^d)}ds. 
 \label{ch4ThB1}
 \end{eqnarray}
 Using Lemma \ref{ch4lemma18} and the inequality $(a+b)^4\leq 16a^4+16b^4$, it follows that 
 \begin{eqnarray*}
 B_2 &: =&\left\|\sup_{s\in[0,t]}\left|\int_0^s\|X_u-\overline{Y}^M_u+h(X_u)-h(\overline{Y}^M_{\lfloor u\rfloor})\|^2d\overline{N}_u\right|\right\|_{L^{p/2}(\Omega, \mathbb{R}^d)}\nonumber\\
 &\leq &C_p\left(\int_0^t \|X_u-\overline{Y}^M_u+h(X_u)-h(\overline{Y}^M_{\lfloor u\rfloor})\|^4_{L^{p/2}(\Omega, \mathbb{R}^d)}du\right)^{1/2}\nonumber\\
 &\leq &C_p\left(\int_0^t16\|X_u-\overline{Y}^M_u||^4_{L^{p/2}(\Omega, \mathbb{R}^d)}+16\|h(X_u)-h(\overline{Y}^M_{\lfloor u\rfloor})||^4_{L^{p/2}(\Omega, \mathbb{R}^d)}du\right)^{1/2},
 \end{eqnarray*}
 for all $p\geq 2$.
 
 Using the inequality $\sqrt{a+b}\leq\sqrt{a}+\sqrt{b}$  for $a,b \in \mathbb{R}^+$, it follows that 
 \begin{eqnarray}
 B_2&\leq &2C_p\left(\int_0^t\|X_u-\overline{Y}^M_u\|^4_{L^{p/2}(\Omega, \mathbb{R}^d)}du\right)^{1/2} +2C_p\left(\int_0^t\|h(X_u)-h(\overline{Y}^M_{\lfloor u\rfloor})\|^4_{L^{p/2}(\Omega, \mathbb{R}^d)}du\right)^{1/2}\nonumber\\
 & =& B_{21}+B_{22}.
 \label{ch4ThB}
 \end{eqnarray}
 Using Holder's inequality, it follows that 
 \begin{eqnarray*}
 B_{21} &: = &2 C_p\left(\int_0^t\|X_u-\overline{Y}^M_u\|^4_{L^{p/2}(\Omega, \mathbb{R}^d)}du\right)^{1/2}\nonumber\\
 &\leq &2C_p\left(\int_0^t\|X_u-\overline{Y}^M_u\|^2_{L^p(\Omega, \mathbb{R}^d)}\|X_u-\overline{Y}^M_u\|^2_{L^p(\Omega, \mathbb{R}^d)}du\right)^{1/2}\nonumber\\
 &\leq &\dfrac{1}{4}\sup_{u\in[0,t]}\|X_u-\overline{Y}^M_u\|_{L^p(\Omega,\mathbb{R}^d)} \,\,8C_p\left(\int_0^t\|X_u-\overline{Y}^M_u\|^2_{L^p(\Omega,\mathbb{R}^d)}du\right)^{1/2}.
 \end{eqnarray*}
 Using the inequality $2ab\leq a^2+b^2$  for $a,b \in \mathbb{R}$ leads to 
 \begin{eqnarray}
 B_{21}&\leq &\dfrac{1}{16}\sup_{u\in[0,t]}\|X_u-\overline{Y}^M_u\|^2_{L^p(\Omega, \mathbb{R}^d)}
 +16C^2_p\int_0^t\|X_u-\overline{Y}^M_u\|^2_{L^p(\Omega, \mathbb{R}^d)}du.
 \label{ch4ThB21}
 \end{eqnarray}
 Using the inequalities $(a+b)^4\leq 4a^4+4b^4$ and $\sqrt{a+b}\leq\sqrt{a}+\sqrt{b}$ for $a,b \in \mathbb{R}^+$, we obtain 
 \begin{eqnarray*}
 B_{22}&: = &2C_p\left(\int_0^t\|h(X_u)-h(\overline{Y}^M_{\lfloor u\rfloor})\|^4_{L^{p/2}(\Omega, \mathbb{R}^d)}du\right)^{1/2}\nonumber\\
 &\leq &2C_p\left(\int_0^t\left[4 \|h(X_u)-h(\overline{Y}^M_u)\|^4_{L^{p/2}(\Omega,\mathbb{R}^d)}+4\|h(\overline{Y}^M_u)-h(\overline{Y}^M_{\lfloor u\rfloor})\|^4_{L^{p/2}(\Omega,\mathbb{R}^d)}\right]du\right)^{1/2}\\
 &\leq &4C_p\left(\int_0^t\|h(X_u)-h(\overline{Y}^M_u)\|^4_{L^{p/2}(\Omega, \mathbb{R}^d)}du\right)^{1/2}
 +4C_p\left(\int_0^t\|h(\overline{Y}^M_u)-h(\overline{Y}^M_{\lfloor u\rfloor})\|^4_{L^{p/2}(\Omega, \mathbb{R}^d)}du\right)^{1/2}.
 \end{eqnarray*}
 Using the global Lipschitz condition, leads to 
 \begin{eqnarray*}
 B_{22}&\leq &4C_p\left(\int_0^tC\|X_u-\overline{Y}^M_u\|^4_{L^{p/2}(\Omega, \mathbb{R}^d)}du\right)^{1/2}+4C_p\left(\int_0^tC \|\overline{Y}^M_u-\overline{Y}^M_{\lfloor u\rfloor}\|^4_{L^{p/2}(\Omega, \mathbb{R}^d)}du\right)^{1/2}.
 \end{eqnarray*}
 Using the same estimations as for $B_{21}$, it follows that :
 \begin{eqnarray}
 B_{22} &\leq &\dfrac{1}{16}\sup_{s\in[0,t]}\|X_s-\overline{Y}^M_s\|^2_{L^p(\Omega, \mathbb{R}^d)} +64C_p\int_0^t\|X_u-\overline{Y}^M_u\|^2_{L^p(\Omega, \mathbb{R}^d)}du\nonumber\\
 &+&\dfrac{1}{4}\sup_{s\in[0,t]}\|\overline{Y}^M_s-\overline{Y}^M_{\lfloor s\rfloor}\|^2_{L^p(\Omega, \mathbb{R}^d)}+64C_p\int_0^t\|\overline{Y}^M_u-\overline{Y}^M_{\lfloor u\rfloor}\|^2_{L^p(\Omega, \mathbb{R}^d)}du.\nonumber
 \end{eqnarray}
 Taking the supremum under the integrand in the last term of the above inequality and using the fact that $C_p$ is an arbitrary constant  leads to 
 \begin{eqnarray}
 B_{22}&\leq &\dfrac{1}{16}\sup_{s\in[0,t]}\|X_s-\overline{Y}^M_s\|^2_{L^p(\Omega, \mathbb{R}^d)} +64C_p\int_0^t\|X_u-\overline{Y}^M_u\|^2_{L^p(\Omega, \mathbb{R}^d)}du\nonumber\\
 &+&C_p\sup_{s\in[0,t]}\|\overline{Y}^M_s-\overline{Y}^M_{\lfloor s\rfloor}\|^2_{L^p(\Omega, \mathbb{R}^d)}.
 \label{ch4ThB22}
 \end{eqnarray}
 Inserting \eqref{ch4ThB21} and \eqref{ch4ThB22} into \eqref{ch4ThB} gives
 \begin{eqnarray}
 B_2&\leq &\dfrac{1}{8}\sup_{s\in[0,t]}\|X_s-\overline{Y}^M_s\|^2_{L^p(\Omega, \mathbb{R}^d)}+C_p\int_0^t\|X_u-\overline{Y}^M_u\|^2_{L^p(\Omega, \mathbb{R}^d)}du\nonumber\\
 &+&C_p\sup_{s\in[0,t]}\|\overline{Y}^M_s-\overline{Y}^M_{\lfloor s\rfloor}\|^2_{L^p(\Omega, \mathbb{R}^d)}.
 \label{ch4ThB2}
 \end{eqnarray}
 Using again Lemma \ref{ch4lemma18} leads to 
 \begin{eqnarray*}
B_3 &: =& \left\|\sup_{u\in[0,t]}\left(\int_0^s\|X_u-\overline{Y} ^M_u\|^2d\overline{N}_u\right)^{1/2}\right\|_{L^{p/2}(\Omega, \mathbb{R}^d)}\\
&\leq &C_p\left(\int_0^t\|X_u-\overline{Y}^M_u\|^4_{L^{p/2}(\Omega, \mathbb{R}^d)}du\right)^{1/2}.
 \end{eqnarray*}
 Using the same argument as for $B_{21}$, we obtain 
 \begin{eqnarray}
 B_3 &\leq &\dfrac{1}{8}\sup_{u\in[0,t]}\|X_u-\overline{Y}^M_u\|^2_{L^p(\Omega, \mathbb{R}^d)}+C_p\int_0^t\|X_u-\overline{Y}^M_u\|^2_{L^p(\Omega, \mathbb{R}^d)}du.
 \label{ch4ThB3}
 \end{eqnarray}
 Taking the $L^{p}$ norm in both  side of \eqref{ch4Th2}, inserting inequalities \eqref{ch4ThB1}, \eqref{ch4ThB2}, \eqref{ch4ThB3} and using Minkowski's inequality in its integral form leads to 
\begin{eqnarray*}
\left\|\sup_{s\in[0,t]}\|X_s-\overline{Y}^M_s\|\right\|^2_{L^p(\Omega, \mathbb{R})}&=&\left\|\sup_{s\in[0,t]}\|X_s-\overline{Y}^M_s\|^2\right\|_{L^{p/2}(\Omega, \mathbb{R})},
\end{eqnarray*}
so
  \begin{eqnarray}
  \label{rev2}
  \lefteqn{\left\|\sup_{s\in[0,t]}\|X_s-\overline{Y}^M_s\|\right\|^2_{L^p(\Omega, \mathbb{R})}}\nonumber\\
  &\leq & C_p\int_0^t\|X_s-\overline{Y}^M_s\|^2_{L^p(\Omega, \mathbb{R}^d)}ds+C_p\int_0^t\|\overline{Y}^M_s-\overline{Y}^M_{\lfloor s\rfloor}\|^2_{L^p(\Omega,\mathbb{R}^d)}ds\nonumber\\
 &+&\int_0^t\|f_{\lambda}(X_s)-f_{\lambda}(\overline{Y}^M_{\lfloor s\rfloor})\|^2_{L^{p}(\Omega, \mathbb{R}^d)}ds+C_p\sup_{u\in[0,t]}\|\overline{Y}^M_u-\overline{Y}^M_{\lfloor u\rfloor}\|^2_{L^p(\Omega, \mathbb{R}^d)}\nonumber\\
 &+&\dfrac{T^{2\alpha}}{M^{2\alpha}}\int_0^t\|f_{\lambda}(\overline{Y}^M_{\lfloor s\rfloor})\|^4_{L^{2p}(\Omega, \mathbb{R}^d)}ds +2C_p\int_0^t\|\overline{Y}^M_s-\overline{Y}^M_{\lfloor s\rfloor}\|^2_{L^p(\Omega, \mathbb{R}^d)}ds\nonumber\\
 &+&\dfrac{1}{2}\left\|\sup_{s\in[0,t]}\|X_s-\overline{Y}^M_s\|\right\|^2_{L^{p}(\Omega, \mathbb{R})},
 \end{eqnarray} 
 for all $t\in[0,T]$ and all $p\in[2,+\infty)$.
 
Inequality \eqref{rev2} can be rewritten in the following appropriate form 

\begin{eqnarray}
\label{rev3}
 \lefteqn{\left\|\sup_{s\in[0,t]}\|X_s-\overline{Y}^M_s\|\right\|^2_{L^p(\Omega, \mathbb{R})}}\nonumber\\
&\leq & C_p\int_0^t\|X_s-\overline{Y}^M_s\|^2_{L^p(\Omega, \mathbb{R}^d)}ds+C_p\int_0^t\|\overline{Y}^M_s-\overline{Y}^M_{\lfloor s\rfloor}\|^2_{L^p(\Omega,\mathbb{R}^d)}ds\nonumber\\
 &+&\int_0^t\|f_{\lambda}(X_s)-f_{\lambda}(\overline{Y}^M_{\lfloor s\rfloor})\|^2_{L^{p}(\Omega, \mathbb{R}^d)}ds+C_p\sup_{u\in[0,t]}\|\overline{Y}^M_u-\overline{Y}^M_{\lfloor u\rfloor}\|^2_{L^p(\Omega, \mathbb{R}^d)}\nonumber\\
 &+&\dfrac{T^{2\alpha}}{M^{2\alpha}}\int_0^t\|f_{\lambda}(\overline{Y}^M_{\lfloor s\rfloor})\|^4_{L^{2p}(\Omega, \mathbb{R}^d)}ds +2C^2m\int_0^t\|\overline{Y}^M_s-\overline{Y}^M_{\lfloor s\rfloor}\|^2_{L^p(\Omega, \mathbb{R}^d)}ds.
 \end{eqnarray}
 
  Applying Gronwall's lemma to \eqref{rev3} leads to
  {\small{
 \begin{eqnarray}
 \label{rev4}
 \lefteqn{\dfrac{1}{2}\left\|\sup\limits_{s\in[0,t]}\|X_s-\overline{Y}^M_s\|\right\|^2_{L^p(\Omega, \mathbb{R})}}\nonumber\\ 
 &\leq &C_pe^{C_p}\left(\int_0^T\|f_{\lambda}(\overline{Y}^M_s)-f_{\lambda}(\overline{Y}^M_{\lfloor s\rfloor})\|^2_{L^p(\Omega, \mathbb{R}^d)}ds+C_p\sup_{u\in[0,t]}\|\overline{Y}^M_u-\overline{Y}^M_{\lfloor u\rfloor}\|^2_{L^p(\Omega, \mathbb{R}^d)}\right.\nonumber\\
 & &\left. +\dfrac{T^{2\alpha}}{M^{2\alpha}}\int_0^T\|f_{\lambda}(\overline{Y}^M_{\lfloor s\rfloor})\|^4_{L^{2p}(\Omega, \mathbb{R}^d)}ds +C_p\int_0^T\|\overline{Y}^M_s-\overline{Y}^M_{\lfloor s\rfloor}\|^2_{L^p(\Omega, \mathbb{R}^d)}ds\right).
 \end{eqnarray}
 }}
 From \eqref{rev4} and  the inequality $\sqrt{a+b+c}\leq \sqrt{a}+\sqrt{b}+\sqrt{c}$ for all $ a, b, c \in \mathbb{R}^+$, it follows that
 \begin{eqnarray}
 \lefteqn{\dfrac{1}{2}\left\|\sup\limits_{s\in[0,t]}\|X_s-\overline{Y}^M_s\|\right\|_{L^p(\Omega, \mathbb{R})}}\nonumber\\ 
 &\leq &C_pe^{C_p}\left(\sup_{t\in[0,T]}\|f_{\lambda}(\overline{Y}^M_t)-f_{\lambda}(\overline{Y}^M_{\lfloor t\rfloor})\|_{L^p(\Omega, \mathbb{R}^d)}+C_p\sup_{t\in[0,T]}\|\overline{Y}^M_t-\overline{Y}^M_{\lfloor t\rfloor}\|_{L^p(\Omega, \mathbb{R}^d)} \right.\nonumber\\
 &+& \left. \dfrac{T^{\alpha}}{M^{\alpha}}\left[\sup_{n\in\{0,\cdots, M\}}\|f_{\lambda}(Y^M_n)\|^2_{L^p(\Omega, \mathbb{R}^d)}\right]+C_p\sup_{t\in[0,T]}\|\overline{Y}^M_t-\overline{Y}^M_{\lfloor t\rfloor}\|_{L^p(\Omega, \mathbb{R}^d)}\right),
 \label{ch4Gronwall}
 \end{eqnarray}
 for all $p\in[2,\infty)$.
 
 Using Lemma \ref{ch4lemma21},  Lemma \ref{ch4lemma22} and the inequality $\dfrac{T^{\alpha}}{M^{\alpha}}\leq C_{\alpha}\,\Delta t^{1/2}$, it follows from \eqref{ch4Gronwall} that 
 \begin{eqnarray}
 \label{compenfinal}
 \left\|\sup\limits_{t\in[0,T]}\|X_t-\overline{Y}^M_t\|\right\|_{L^p(\Omega, \mathbb{R})}=\left(\mathbb{E}\left[\sup_{t\in[0,T]}\left\|X_t-\overline{Y}^M_t\right\|^p\right]\right)^{1/p}\leq C_p \Delta t^{1/2},
 \end{eqnarray}
 for all $p\in[2,\infty)$ and all $M\in\mathbb{N}$. Using Holder's inequality, one can prove  that \eqref{compenfinal} holds for $p\in[1,2]$. The proof of the theorem is complete.
 \subsection{Proof of \thmref{ch4theorem1} for STS scheme ($\chi^M_t= \overline{Z}^M_t$)}
 \label{stss}
 After replacing the increment of the poisson process $\Delta N_n^M$ by its compensated form $\Delta\overline{N}_n^M$ in STS  \eqref{sts}, we obtain an equivalent scheme similar to the compensated tamed scheme (CTS). 
 Therefore, the proof of the strong convergence of the STS  follows exactly the one of compensated tamed scheme (CTS) \eqref{tamedjump} in \secref{tamedjs}.  Here we should make the following changes for our semi-tamed scheme
 \begin{eqnarray*}
 \alpha^M_k := \mathbbm{1}_{\{\|Z^M_k\|\geq 1\}}\left\langle\dfrac{Z^M_k+u_{\lambda}(Z^M_n)\Delta t}{\|Z^M_k\|}, \dfrac{g(Z^M_k)}{\|Z^M_k\|}\Delta W^M_k\right\rangle,\\\\
 \beta^M_k := \mathbbm{1}_{\{\|Z^M_k\|\geq 1\}}\left\langle\dfrac{Z^M_k+u_{\lambda}(Z^M_n)\Delta t}{\|Z^M_k\|}, \dfrac{h(Z^M_k)}{\|Z^M_k\|}\Delta\overline{N}^M_k\right\rangle,\\
 \end{eqnarray*}
 where $u_{\lambda}=u+\lambda h$.  The function $v$ which is one-side Lipschitz (see \rmref{Lips}) should replace  the  function
 $f_\lambda$ in  the proof of  the compensated tamed scheme (CTS).
 It follows from the proof  in \secref{tamedjs} that there exists a constant $C_p>0$ such that  
  \begin{eqnarray}
 \left(\mathbb{E}\left[\sup_{t\in[0,T]}\left\|X_t-\overline{Z}^M_t\right\|^p\right]\right)^{1/p}\leq C_p\Delta t^{1/2},
 \end{eqnarray}
 for all $p\in[1,\infty)$.
 Details  can also be found in \cite{Mukam}.
 
 \subsection{Proof of \thmref{ch4theorem1} for NCTS scheme ($\chi^M_t= \overline{X}^M_t$)}
 \label{nctss}
    Using the relation $\Delta\overline{N}^M_n=\Delta N^M_n-\lambda\Delta t$,  the continuous interpolation of \eqref{dncts} can  be expressed in the following form
    \begin{eqnarray*}
    \overline{X}^M_t =X^M_n+\lambda(t-n\Delta t)h(X^M_n)+\dfrac{(t-n\Delta t)f(X^M_n)}{1+\Delta t^{\alpha} \|f(X^M_n)\|}+g(X^M_n)(W_t-W_{n\Delta t})+h(X^M_n)(\overline{N}_t-\overline{N}_{n\Delta t}),
    \end{eqnarray*}
    for all $t\in[n\Delta t, (n+1)\Delta t)$.
    
   The numerical solution of non compensated tamed scheme (NCTS) \eqref{ncts}  is  also  equivalent to 
   \begin{eqnarray}
   X^M_{n+1}&=&X^M_n+\dfrac{\Delta tf(X^M_n)}{1+\Delta t^{\alpha}\|f(X^M_n)\|}+g(X^M_n)\Delta W^M_n+h(X^M_n)\Delta N^M_n\nonumber\\
   &=&X^M_n+\lambda h(X^M_n) \Delta t+\dfrac{\Delta tf(X^M_n)}{1+\Delta t^{\alpha}\|f(X^M_n)\|}+g(X^M_n)\Delta W^M_n\nonumber\\
   &+&h(X^M_n)\Delta \overline{N}^M_n.
   \label{ch5contam}
   \end{eqnarray}
  The functions $\lambda h$ and $f$ in the numerical solution of the scheme NCTS given by \eqref{ncts} (or \eqref{ch5contam}) satisfy  respectively the same conditions as
  the  $u_{\lambda}$ and $v$ in the numerical solution of  the STS given by \eqref{sts}.
  Hence, it follows from the proof  in \secref{stss} that there exists a constant $C_p>0$ such that  
  \begin{eqnarray}
 \left(\mathbb{E}\left[\sup_{t\in[0,T]}\left\|X_t-\overline{X}^M_t\right\|^p\right]\right)^{1/p}\leq C_p\Delta t^{1/2},
 \end{eqnarray}
 for all $p\in[1,\infty)$.

\section{Linear mean-square stability} 
\label{linearsta}
The goal of  this section is to find the time  
 step-size  limit for which the tamed Euler scheme and the semi-tamed Euler scheme are stable
 in  the linear mean-square sense. 
 For the scalar linear test problem, the concept of A-stability of a numerical method may be interpreted as ``problem
stable $\Rightarrow$ method stable for all $\Delta t$''.
We consider the following linear test equation with real and scalar coefficients
\begin{eqnarray}
 dX(t)=  aX(t^{-})dt +bX(t^{-})dW(t)+cX(t^{-})dN(t),  \hspace{0.5cm}
 X(0)=X_0,
 \label{ch5linearequation}
\end{eqnarray}
where $X_0$ satisfied $\mathbb{E}\|X_0\|^2<\infty$. In the sequel  of  this paper we take $\alpha \in [0,1]$.
It has been proved in \cite{Desmond2} that the exact solution of \eqref{ch5linearequation} is mean-square stable if and only if 
  \begin{eqnarray}
  \label{lstable}
  \underset{t \rightarrow \infty}{\lim} \mathbb{E}(X(t)^2)=0  \Leftrightarrow l :=2a+b^2+\lambda c(2+c)<0.
  \end{eqnarray}
  Using the discrete form  of \eqref{ch5linearequation}, the numerical schemes \eqref{sts} and \eqref{ncts} will be  mean-square stable if  $l<0$ and 
   \begin{eqnarray}
   \underset{n \rightarrow \infty}{\lim} \mathbb{E} (Y_n^2)= \underset{n \rightarrow \infty}{\lim} \mathbb{E}(X_n^2)= 0.
   \end{eqnarray}
The following result  provides the  time step-size limit for which the semi-tamed scheme (STS) \eqref{sts}   is mean-square stable.
 \begin{theorem}
 \label{th:sts}
 Assume  that $l<0$, then the semi-tamed  scheme \eqref{sts} is mean-square stable if and only if
 \begin{eqnarray*}
 \Delta t<\dfrac{-l}{(a+\lambda c)^2}.
 \end{eqnarray*}
 \end{theorem}
 \begin{proof}
 Note  that if $l<0$, then $a+\lambda c <0$. Applying the semi-tamed Euler scheme to \eqref{ch5linearequation} and  using  the compensated 
 Poisson process $\overline{N}(t):=N(t)-\lambda t$ leads to
 \begin{eqnarray}
 Y_{n+1}=Y_n+aY_n\Delta t+\lambda cY_n\Delta t+bY_n\Delta W_n+cY_n\Delta\overline{N}_n.
 \label{ch5compen1}
 \end{eqnarray}
 Squaring both sides of \eqref{ch5compen1} leads to
 \begin{eqnarray}
 Y_{n+1}^2&=&Y_n^2+(a+\lambda c)^2\Delta t^2Y_n^2+b^2Y_n^2\Delta W_n^2+c^2Y_n^2\Delta\overline{N}_n^2 \nonumber\\
  && +2(a+\lambda c) \Delta t Y_n^2+2bY_n^2\Delta W_n\nonumber\\
 &&+2cY_n^2\Delta\overline{N}_n+2b(a+\lambda c)\Delta t\Delta W_n Y_n^2 \nonumber\\
 &&+2c(a+\lambda c)Y_n^2\Delta t\Delta\overline{N}_n+2bcY_n^2\Delta W_n\Delta\overline{N}_n.
 \label{ch5compen2}
 \end{eqnarray}
 Taking the expectation in both sides of \eqref{ch5compen2} and using the equalities 
 $\mathbb{E}(\Delta W_n^2)=\Delta t$, $\mathbb{E}(\Delta\overline{N}_n^2)=\lambda \Delta t$ and  $\mathbb{E}(\Delta W_n)=\mathbb{E}(\Delta\overline{N}_n)=0$ 
 with the fact  that $\Delta W_n$ and $ \Delta\overline{N}_n$ are independent leads to
 \begin{eqnarray*}
 \mathbb{E}|Y_{n+1}|^2=(1+(a+\lambda c)^2\Delta t^2+(b^2+\lambda c^2+2a+2\lambda c)\Delta t)\mathbb{E}|Y_n|^2.
 \end{eqnarray*}
 So, the semi-tamed scheme is stable if and only if
 \begin{eqnarray*}
 1+(a+\lambda c)^2\Delta t^2+(b^2+\lambda c^2+2a+2\lambda c)\Delta t<1.
 \end{eqnarray*}
 That is 
 $\Delta t<\dfrac{-l}{(a+\lambda c)^2}$.
 \end{proof}
 
 The following result  provides the  time step-size limit for which the non compensated  tamed scheme (NCTS) \eqref{ncts} is stable.
 \begin{theorem}
 \label{thmncts}
 Assume  that $l<0$, then the tamed Euler scheme \eqref{ncts} is mean-square stable  if one   of  the  following conditions is satisfied 
 
 \begin{itemize}
\item[(i)]   $a(1+\lambda c\Delta t)\leq 0$, $2a-l>0$ and $\Delta t<\dfrac{2a-l}{a^2+\lambda^2c^2}$.\\
\item [(ii)]  $a(1+\lambda c\Delta t)> 0$ and $\Delta t<\dfrac{-l}{(a+\lambda c)^2}$.
\end{itemize}
\end{theorem}
\begin{proof}
Applying the tamed Euler scheme \eqref{ncts} to equation \eqref{ch5linearequation} leads to 
\begin{eqnarray}
X_{n+1}=X_n+\dfrac{a X_n\Delta t}{1+\Delta t^{\alpha}|aX_n|}+bX_n\Delta W_n+cX_n\Delta N_n.
\label{ch5eq1}
\end{eqnarray}
By squaring both sides of \eqref{ch5eq1} leads to 
\begin{eqnarray}
\label{sq}
X_{n+1}^2&=&X_n^2+\dfrac{a^2X^2_n\Delta t^2}{(1+\Delta t^{\alpha}|a X_n|)^2}+b^2 X_n^2\Delta W_n^2+c^2 X_n^2\Delta N_n^2+\dfrac{2a X_n^2\Delta t}{1+\Delta t^{\alpha}|a X_n|}\nonumber\\
&&+2b X_n^2\Delta W_n
+2cX_n^2\Delta N_n+\dfrac{2abX_n^2\Delta t}{1+\Delta t^{\alpha}|aX_n|}\Delta W_n\nonumber\\
&&+\dfrac{2acX_n^2\Delta t}{1+\Delta t^{\alpha}|a X_n|}\Delta N_n+2bcY_n^2\Delta W_n\Delta N_n.
\end{eqnarray}
Using the inequality $ \dfrac{a^2X_n^2\Delta t^2}{(1+\Delta t^{\alpha}|a X_n|)^2}<a^2X_n^2\Delta t^2$, \eqref{sq} becomes 
\begin{eqnarray*}
\label{sq1}
X_{n+1}^2&\leq &X_n^2+a^2X^2_n\Delta t^2+b^2 X_n^2\Delta W_n^2+c^2 X_n^2\Delta N_n^2+\dfrac{2a X_n^2\Delta t}{1+\Delta t^{\alpha}|a X_n|}\nonumber\\
&&+2bX_n^2\Delta W_n
+2c X_n^2\Delta N_n+\dfrac{2ab X_n^2\Delta t}{1+\Delta t^{\alpha}|aX_n|}\Delta W_n\nonumber\\
&& +\dfrac{2acX_n^2\Delta t}{1+\Delta t^{\alpha}|aX_n|}\Delta N_n+2bcX_n^2\Delta W_n\Delta N_n.
\end{eqnarray*}
Taking the  expectation in both sides of \eqref{sq1},  using the independence of the Brownian and Poisson processes,
and  the fact that $\mathbb{E}(\Delta W_n)=0$, $\mathbb{E}(\Delta W_n^2)=\Delta t$, $\mathbb{E}(\Delta N_n)=\lambda\Delta t$, $\mathbb{E}(\Delta N_n^2)=\lambda \Delta t+\lambda^2\Delta t^2$ 
leads to
\begin{eqnarray}
\mathbb{E}|X_{n+1}|^2&\leq& \left[1+a^2\Delta t^2+b^2\Delta t+\lambda^2c^2\Delta t^2+(2+ c)\lambda c\Delta t\right]\mathbb{E}|X_n|^2\nonumber\\
&+&\mathbb{E}\left(\dfrac{2aX^2_n\Delta t(1+\lambda c\Delta t)}{1+\Delta t^{\alpha}|aX_n|}\right).
\label{ch5eq2}
\end{eqnarray}
 If $a(1+\lambda c\Delta t)\leq 0$, it follows from \eqref{ch5eq2} that
\begin{eqnarray*}
\mathbb{E}|X_{n+1}|^2\leq \{1+(a^2+\lambda^2c^2)\Delta t^2+[b^2+\lambda c(2+c)]\Delta t\}\mathbb{E}|X_n|^2.
\end{eqnarray*}
Therefore, the numerical solution is stable if 
\begin{eqnarray*}
1+(a^2+\lambda^2c^2)\Delta t^2+[b^2+\lambda c(2+c)]\Delta t<1.
\end{eqnarray*}
That is  $\Delta t<\dfrac{2a-l}{a^2+\lambda^2c^2}$.

If $a(1+\lambda c\Delta t)> 0$, using the fact that $$\dfrac{2a X_n^2\Delta t(1+\lambda c\Delta t)}{1+\Delta t^{\alpha}|aX_n|}< 2aX_n^2\Delta t(1+\lambda c\Delta t),$$
inequality \eqref{ch5eq2} becomes
\begin{eqnarray}
\label{ch5eq3}
\lefteqn{\mathbb{E}|X_{n+1}|^2}  && \nonumber\\
&\leq&\left[1+a^2\Delta t^2+b^2\Delta t+\lambda^2c^2\Delta t^2+2\lambda ac\Delta t^2 \right.\nonumber\\ 
&& \left.+(2+ c)\lambda c\Delta t+2a\Delta t\right]\mathbb{E}|X_n|^2.
\end{eqnarray}
Therefore, it follows from \eqref{ch5eq3} that the numerical solution is stable if
 $$1+a^2\Delta t^2+b^2\Delta t+\lambda^2c^2\Delta t^2+2\lambda ac\Delta t^2+(2+ c)\lambda c\Delta t+2a\Delta t<1.$$ 
 That is $\Delta t<\dfrac{-l}{(a+\lambda c)^2}$.
\end{proof}
 \begin{remark}
In \thmref{thmncts}, we can  easily check that if $l<0$, we have
\begin{eqnarray*}
\left\lbrace \begin{array}{l}
 a(1+\lambda c\Delta t)\leq0,\\
 2a-l>0 \\
 \Delta t<\dfrac{2a-l}{a^2+\lambda^2c^2}\\
 \end{array} \right. 
 \Leftrightarrow 
 \left\lbrace \begin{array}{l}
   a \in (l/2,0], c\geq 0,\\
   \Delta t <\dfrac{2a-l}{a^2+\lambda^2 c^2}\\
 \end{array} \right.
 \bigcup
 \left\lbrace \begin{array}{l}
   a \in (l/2,0), c<0,\\
   \Delta t <\dfrac{2a-l}{a^2+\lambda^2 c^2} \\
   \Delta t\leq \dfrac{-1}{ \lambda c} \\
 \end{array} \right.\\
 \bigcup
 \left\lbrace \begin{array}{l}
   a>0, c<0 \\
   \Delta t <\dfrac{2a-l}{a^2+\lambda^2 c^2}  \\
  \Delta t \geq  \dfrac{-1}{ \lambda c}
 \end{array} \right. 
\end{eqnarray*}
\begin{eqnarray*}
\left\lbrace \begin{array}{l}
 a(1+\lambda c\Delta t)>0,\\
 \Delta t<\dfrac{-l}{(a+\lambda c)^2}\\
 \end{array} \right. 
 \Leftrightarrow 
 \left\lbrace \begin{array}{l}
   a >0, c>0,\\
   \Delta t < \dfrac{-l}{(a+\lambda c)^2}\\
 \end{array} \right.
 \bigcup
 \left\lbrace \begin{array}{l}
   a >0, c<0,\\
   \Delta t <\dfrac{-l}{(a+\lambda c)^2} \wedge \dfrac{-1}{ \lambda c} \\
 \end{array} \right.\\
 \bigcup
 \left\lbrace \begin{array}{l}
   a<0, c<0 \\
   \Delta t <\dfrac{-l}{(a+\lambda c)^2}  \\
  \Delta t >  \dfrac{-1}{ \lambda c}.
 \end{array} \right. 
\end{eqnarray*}
 \end{remark}
\begin{remark}
  Note that from the above studies,  we can  deduce the   linear stabilities of schemes   \eqref{ncts} and \eqref{sts} for SDEs  without jump by  just take  $c=0$ in \eqref{ch5linearequation}, 
  \thmref{th:sts} and \thmref{thmncts}.
  \end{remark}

\section{Nonlinear mean-square stability}
\label{nlinearsta}
In this section, we focus on the  exponential mean-square stability of
the approximation \eqref{sts}. 
We follow closely \cite{semitamed,Desmond2} and assume that $f(0)=u(0)=v(0)=g(0)=h(0)=0$  and $\mathbb{E}\Vert X_0 \Vert^2<\infty$.
It has been proved in \cite{Desmond2} that under the following conditions
\begin{eqnarray}
\label{l1}
\langle x-y, f(x)-f(y)\rangle&\leq& \mu\|x-y\|^2,\\
\label{l2}
\|g(x)-g(y)\|^2&\leq& \sigma \|x-y\|^2,\\
\label{l3}
\|h(x)-h(y)\|^2 &\leq & \gamma \|x-y\|^2,
\end{eqnarray}
 for all $x,y\in\mathbb{R}^d$, where $\mu$, $\sigma$ and $ \gamma$ are constants, the exact solution of SDE \eqref{model} is  nonlinear mean-square stable  
 if $\alpha :=2\mu+\sigma+\lambda  \sqrt{\gamma}(\sqrt{\gamma}+2)<0$.
   Indeed under the above assumptions,  we have \cite[Theorem 4]{Desmond2}
 \begin{eqnarray*}
\mathbb{E}\Vert X(t) \Vert^2 \leq \mathbb{E}\Vert X_0 \Vert^2 e^{\alpha t}.
 \end{eqnarray*}
 So, if $\alpha <0$ we have $ \underset { t \rightarrow \infty}{\lim} \mathbb{E}\Vert X(t) \Vert^2=0$ and  the exact solution $X$ is exponentially mean-square stable.
 
 In the sequel of  this section, we will  use some weaker assumptions, which of course imply that the conditions \eqref{l1}-\eqref{l3} hold.
   More precisely, for  nonlinear stability of  the semi-tamed scheme (STS), we also make  the following assumptions.
  \begin{Assumption}
\label{ch5assumption2}
 There exist some  positive constants $\rho$, $\beta$,$\overline{\beta}$, $K$, $C$, $\theta$ and $a>1$  such that
 \begin{eqnarray*}
 \langle x-y, u(x)-u(y)\rangle\leq -\rho\|x-y\|^2,\hspace{1cm} \|u(x)-u(y)\|\leq K\|x-y\|,\\
 \langle x-y, v(x)-v(y)\rangle\leq-\beta\|x-y\|^{a+1}, \hspace{1.5cm} \|v(x)\|\leq \overline{\beta}\|x\|^a,\\
 \|g(x)-g(y)\|\leq\theta\|x-y\|,\hspace{2cm} \|h(x)-h(y)\|\leq C\|x-y\|.
 \end{eqnarray*}
 \end{Assumption}
We denote by $\alpha_1: =-2\rho+\theta^2+\lambda  C(C+2)$ and  we will always assume that $\alpha_1<0$  to ensure  the stability of  the exact solution.
 The nonlinear stability  result for   the scheme (STS) is given in the following theorem.
  \begin{theorem}
 \label{nt1}
 Under Assumptions \ref{ch5assumption2} and the further hypothesis $2\beta-\overline{\beta}>0$, 
 for any stepsize $\Delta t$ such that $\Delta t<\dfrac{-\alpha_1}{(K+\lambda C)^2}\wedge\dfrac{2\beta}{[2(K+\lambda C)+\overline{\beta}]\overline{\beta}}\wedge\dfrac{2\beta-\overline{\beta}}{2(K+\lambda c)\overline{\beta}}$,
  there exists  a constant  $\gamma=\gamma (\Delta t) >0$ such that
 \begin{eqnarray*}
 \mathbb{E}\|Y_n\|^2\leq \mathbb{E}\|X_0\|^2 e^{-\gamma\, t_n}, \,\, t_n=n\,\Delta t,\qquad \underset{ \Delta t \rightarrow 0}{\lim}\gamma(\Delta t)= -\alpha_1,
 \end{eqnarray*}  
 and  the  numerical solution \eqref{sts} is exponentiallly mean-square stable.
 \end{theorem}
 \begin{proof}
 Recall that the numerical solution \eqref{sts} can also be written as 
 \begin{eqnarray*}
 \label{rsts}
 Y_{n+1}=Y_n+\Delta tu_{\lambda}(Y_n)+\dfrac{\Delta tv(Y_n)}{1+\Delta t^{\alpha} \|v(Y_n)\|}+g(Y_n)\Delta W_n+h(Y_n)\Delta\overline{N}_n,
 \end{eqnarray*}
  where $u_{\lambda}=u+\lambda h$. Taking the inner product in both sides  of \eqref{rsts} leads to 
 \begin{eqnarray}
 \|Y_{n+1}\|^2&=&\|Y_n\|^2+\Delta t^2\|u_{\lambda}(Y_n)\|^2+\dfrac{\Delta t^2 \|v(Y_n)\|^2}{\left(1+\Delta t^{\alpha}\|v(Y_n)\|\right)^2}+\|g(Y_n)\|^2 \|\Delta W_n\|^2\nonumber\\
 &+& \|h(Y_n)\|^2|\Delta \overline{N}_n |^2
 +2\Delta t\langle Y_n, u_{\lambda}(Y_n)\rangle+2\Delta t\left\langle Y_n, \dfrac{v(Y_n)}{1+\Delta t^{\alpha}\|v(Y_n)\|}\right\rangle\nonumber\\
 &+& 2\langle Y_n, g(Y_n)\Delta W_n\rangle + 2\langle Y_n, h(Y_n)\Delta\overline{N}_n\rangle \nonumber\\
 &+& 2\Delta t^2\left\langle u_{\lambda}(Y_n), \dfrac{v(Y_n)}{1+\Delta t^{\alpha}||v(Y_n)||}\right\rangle
 +2\Delta t\langle u_{\lambda}(Y_n), g(Y_n)\Delta W_n\rangle \nonumber\\
 &+&2\Delta t\langle u_{\lambda}(Y_n), h(Y_n)\Delta\overline{N}_n\rangle
 +2\Delta t\left\langle\dfrac{v(Y_n)}{1+\Delta t^{\alpha}||v(Y_n)\|}, g(Y_n)\Delta W_n\right\rangle \nonumber\\
 &+& 2\Delta t\left\langle\dfrac{v(Y_n)}{1+\Delta t^{\alpha}\|v(Y_n)\|}, h(Y_n)\Delta\overline{N}_n\right\rangle +2\langle g(Y_n)\Delta W_n, h(Y_n)\Delta\overline{N}_n\rangle.
 \label{ch5meansemi1}
 \end{eqnarray}
 Using Assumptions \ref{ch5assumption2} together with the fact that $v(0)=0$, it follows that 
 \begin{eqnarray}
 2\Delta t\left\langle Y_n, \dfrac{v(Y_n)}{1+\Delta t^{\alpha} \|v(Y_n)\|}\right\rangle\leq\dfrac{-2\beta\Delta t\|Y_n\|^{a+1}}{1+\Delta t^{\alpha}\|v(Y_n)\|}
 \label{ch5meansemi2}
 \end{eqnarray}
 \begin{eqnarray}
 2\Delta t^2\left<u_{\lambda}(Y_n), \dfrac{v(Y_n)}{1+\Delta t^{\alpha}\|v(Y_n)\|}\right>&\leq &\dfrac{2\Delta t^2\|u_{\lambda}(Y_n)\|\|v(Y_n)\|}{1+\Delta t^{\alpha}\|v(Y_n\|}\nonumber\\
 &\leq&\dfrac{2\Delta t^2(K+\lambda C)\overline{\beta}\|Y_n\|^{a+1}}{1+\Delta t^{\alpha}\|v(Y_n)\|}.
 \label{ch5meansemi3}
 \end{eqnarray}
 By setting  $\Omega_n:=\{\omega\in\Omega : \|Y_n\|>1\}$, we have  
 on $\Omega_n$ 
 \begin{eqnarray}
 \dfrac{\Delta t^2 \|v(Y_n)\|^2}{\left(1+\Delta t^{\alpha}\|v(Y_n)\|\right)^2}\leq\dfrac{\Delta t\|v(Y_n)\|}{1+\Delta t^{\alpha}\|v(Y_n)\|}\leq\dfrac{\overline{\beta}\Delta t\|Y_n\|^{a+1}}{1+\Delta t^{\alpha}\|v(Y_n)\|}.
 \label{ch5meansemi4}
 \end{eqnarray}
 Therefore using \eqref{ch5meansemi2}, \eqref{ch5meansemi3} and \eqref{ch5meansemi4} in \eqref{ch5meansemi1} yields
\begin{eqnarray}
 \|Y_{n+1}\|^2& \leq &\|Y_n\|^2+\Delta t^2 \|u_{\lambda}(Y_n)\|^2+ \|g(Y_n)\|^2\|\Delta W_n\|^2+\|h(Y_n)\|^2|\Delta\overline{N}_n|^2\nonumber\\
 &+&2\Delta t\langle Y_n, u_{\lambda}(Y_n)\rangle+2\langle Y_n, g(Y_n)\Delta W_n\rangle
 +2\langle Y_n, h(Y_n)\Delta\overline{N}_n\rangle \nonumber\\
 &+&2\Delta t\langle u_{\lambda}(Y_n), g(Y_n)\Delta W_n\rangle+2\Delta t\langle u_{\lambda}(Y_n), h(Y_n)\Delta\overline{N}_n\rangle\nonumber\\
 &+&2\Delta t\left\langle\dfrac{v(Y_n)}{1+\Delta t^{\alpha}\|v(Y_n)\|}, g(Y_n)\Delta W_n\right\rangle\nonumber\\
 &+& 2\Delta t\left\langle\dfrac{v(Y_n)}{1+\Delta t^{\alpha}\|v(Y_n)\|}, h(Y_n)\Delta\overline{N}_n\right\rangle
 +2\langle g(Y_n)\Delta W_n, h(Y_n)\Delta\overline{N}_n\rangle \nonumber\\
 &+&\dfrac{\left[-2\beta\Delta t+2(K+\lambda C)\overline{\beta}\Delta t^2+\overline{\beta}\Delta t\right]\|Y_n\|^{a+1}}{1+\Delta t^{\alpha}\|v(Y_n)\|}.
 \label{ch5meansemi4a}
 \end{eqnarray} 
 Since $\Delta t<\dfrac{2\beta-\overline{\beta}}{2(K+\lambda C)\overline{\beta}}$, we have  $-2\beta\Delta t+2(K+\lambda C)\overline{\beta}\Delta t^2+\overline{\beta}\Delta t<0$ and  \eqref{ch5meansemi4a} becomes
 \begin{eqnarray}
 \|Y_{n+1}\|^2&\leq & \|Y_n\|^2+\Delta t^2\|u_{\lambda}(Y_n)\|^2+2\Delta t\langle Y_n, u_{\lambda}(Y_n)\rangle+\|g(Y_n)\|^2\|\Delta W_n\|^2\nonumber\\
 &+&\|h(Y_n)\|^2 |\Delta\overline{N}_n|^2+2\langle Y_n,g(Y_n)\Delta W_n\rangle+2\langle Y_n, h(Y_n)\Delta\overline{N}_n\rangle\nonumber\\
 &+&2\Delta t\langle u_{\lambda}(Y_n), g(Y_n)\Delta W_n\rangle+2\Delta t\langle u_{\lambda}(Y_n), h(Y_n)\Delta\overline{N}_n\rangle\nonumber\\
 &+&2\Delta t\left\langle \dfrac{v(Y_n)}{1+\Delta t^{\alpha} \|v(Y_n)\|}, g(Y_n)\Delta W_n\right\rangle \nonumber\\
 &+&2\Delta t\left\langle\dfrac{v(Y_n)}{1+\Delta t^{\alpha}\|v(Y_n)\|}, h(Y_n)\Delta\overline{N}_n\right\rangle
 +2\langle g(Y_n)\Delta W_n, h(Y_n)\Delta\overline{N}_n\rangle.
 \label{ch5meansemi4b}
 \end{eqnarray}
 However on  $\Omega_n^c$,  we have 
 \begin{eqnarray}
 \dfrac{\Delta t^2 \|v(Y_n)\|^2}{\left(1+\Delta t^{\alpha} \|v(Y_n)\|\right)^2}&\leq&
 \dfrac{\Delta t^2\|v(Y_n)\|^2}{1+\Delta t^{\alpha}\|v(Y_n)\|} \nonumber\\
 &\leq&\dfrac{\overline{\beta}^2\Delta t^2\|Y_n\|^{2a}}{1+\Delta t^{\alpha}
 \|v(Y_n)\|} \nonumber\\
 &\leq& \dfrac{\overline{\beta}^2\Delta t^2 \|Y_n\|^{a+1}}{1+\Delta t^{\alpha}\|v(Y_n)\|}.
 \label{ch5meansemi5}
 \end{eqnarray}
 Therefore, using \eqref{ch5meansemi2}, \eqref{ch5meansemi3} and \eqref{ch5meansemi5} in \eqref{ch5meansemi1} yields
  \begin{eqnarray}
 \|Y_{n+1}\|^2&\leq &\|Y_n\|^2+\Delta t^2 \|u_{\lambda}(Y_n)\|^2+2\Delta t\langle Y_n, u_{\lambda}(Y_n)\rangle+\|g(Y_n)\|^2 \|\Delta W_n\|^2\nonumber\\
 &+&\|h(Y_n)\|^2 |\Delta\overline{N}_n|^2+2\langle Y_n,g(Y_n)\Delta W_n\rangle+2\langle Y_n, h(Y_n)\Delta\overline{N}_n\rangle\nonumber\\
 &+&2\Delta t\langle u_{\lambda}(Y_n), g(Y_n)\Delta W_n\rangle+2\Delta t\langle u_{\lambda}(Y_n), h(Y_n)\Delta\overline{N}_n\rangle\nonumber\\
 &+&2\Delta t\left\langle \dfrac{v(Y_n)}{1+\Delta t^{\alpha}\|v(Y_n)\|}, g(Y_n)\Delta W_n\right\rangle\nonumber \\
 &+&2\Delta t\left\langle\dfrac{v(Y_n)}{1+\Delta t^{\alpha}\|v(Y_n)\|}, h(Y_n)\Delta\overline{N}_n\right\rangle 
 +2\langle g(Y_n)\Delta W_n, h(Y_n)\Delta\overline{N}_n\rangle \nonumber\\
 &+&\dfrac{\left[-2\beta\Delta t+2(K+\lambda C)\overline{\beta}\Delta t^2+\overline{\beta}^2\Delta t^2\right]\|Y_n\|^{a+1}}{1+\Delta t ^{\alpha} \|v(Y_n)\|}.
 \label{ch5meansemi5a}
 \end{eqnarray}
 Since  $\Delta t<\dfrac{2\beta}{[2(K+\lambda C)+\overline{\beta}]\overline{\beta}}$, we have $-2\beta\Delta t+2(K+\lambda C)\overline{\beta}\Delta t^2+\overline{\beta}^2\Delta t^2<0$ and 
  \eqref{ch5meansemi5a} becomes 
 \begin{eqnarray}
 \|Y_{n+1}\|^2&\leq &\|Y_n\|^2+\Delta t^2\|u_{\lambda}(Y_n)\|^2+2\Delta t\langle Y_n, u_{\lambda}(Y_n)\rangle+\|g(Y_n)\|^2\|\Delta W_n\|^2\nonumber\\
 &+&\|h(Y_n)\|^2 |\Delta\overline{N}_n|^2+2\langle Y_n,g(Y_n)\Delta W_n\rangle+2\langle Y_n, h(Y_n)\Delta\overline{N}_n\rangle\nonumber\\
 &+&2\Delta t\langle u_{\lambda}(Y_n), g(Y_n)\Delta W_n\rangle+2\Delta t\langle u_{\lambda}(Y_n), h(Y_n)\Delta\overline{N}_n\rangle\nonumber\\
 &+&2\Delta t\left\langle \dfrac{v(Y_n)}{1+\Delta t^{\alpha} \|v(Y_n)\|}, g(Y_n)\Delta W_n\right\rangle \nonumber\\
 &+&2\Delta t\left\langle\dfrac{v(Y_n)}{1+\Delta t^{\alpha}\|v(Y_n)\|}, h(Y_n)\Delta\overline{N}_n\right\rangle
 +2\langle g(Y_n)\Delta W_n, h(Y_n)\Delta\overline{N}_n\rangle.
 \label{ch5meansemi5b}
 \end{eqnarray}
  Finally, from the discussion above on $\Omega_n$ and $\Omega_n^c$, 
  it follows from  the  inequality $$\Delta t \leq \dfrac{2\beta}{[2(K+\lambda C)+\overline{\beta}]\overline{\beta}}\wedge\dfrac{2\beta-\overline{\beta}}{2(K+\lambda c)\overline{\beta}}$$
  and  \eqref{ch5meansemi5a} that on $\Omega$, we have
 \begin{eqnarray}
 \|Y_{n+1}\|^2&\leq &\|Y_n\|^2+\Delta t^2 \|u_{\lambda}(Y_n)\|^2+2\Delta t\langle Y_n, u_{\lambda}(Y_n)\rangle+\|g(Y_n)\|^2 \|\Delta W_n\|^2\nonumber\\
 &+&\|h(Y_n)\|^2|\Delta\overline{N}_n|^2+2\langle Y_n,g(Y_n)\Delta W_n\rangle+2\langle Y_n, h(Y_n)\Delta\overline{N}_n\rangle\nonumber\\
 &+&2\Delta t\langle u_{\lambda}(Y_n), g(Y_n)\Delta W_n\rangle+2\Delta t\langle u_{\lambda}(Y_n), h(Y_n)\Delta\overline{N}_n\rangle\nonumber\\
 &+&2\Delta t\left\langle \dfrac{v(Y_n)}{1+\Delta t^{\alpha}\|v(Y_n)\|}, g(Y_n)\Delta W_n\right\rangle \nonumber\\
 &+&2\Delta t\left\langle\dfrac{v(Y_n)}{1+\Delta t^{\alpha} \|v(Y_n)\|}, h(Y_n)\Delta\overline{N}_n\right\rangle
 +2\langle g(Y_n)\Delta W_n, h(Y_n)\Delta\overline{N}_n\rangle.
 \label{ch5meansemi6}
 \end{eqnarray}
 Taking the expectation in both sides of \eqref{ch5meansemi6} 
 leads to 
 \begin{eqnarray}
 \mathbb{E}\|Y_{n+1}\|^2&\leq&\mathbb{E}\|Y_n\|^2+\Delta t^2\mathbb{E}\|u_{\lambda}(Y_n)\|^2+2\Delta t\mathbb{E}\langle Y_n, u_{\lambda}(Y_n)\rangle+\Delta t\mathbb{E}\|g(Y_n)\|^2\nonumber\\
 &+&\lambda\Delta t\mathbb{E}\|h(Y_n)\|^2.
 \label{ch5meansemi7}
 \end{eqnarray}
 From Assumptions \ref{ch5assumption2}, we have 
 \begin{eqnarray*}
 \|u_{\lambda}(Y_n)\|^2\leq (K+\lambda C)^2\|Y_n\|^2 \hspace{0.5cm} \text{and} \hspace{0.5cm}\langle Y_n, u_{\lambda}(Y_n)\rangle\leq (-\rho+\lambda C)\|Y_n\|^2.
 \end{eqnarray*}
 So inequality \eqref{ch5meansemi7} yields
 \begin{eqnarray}
 \label{fin}
 \mathbb{E}\|Y_{n+1}\|^2&\leq&\mathbb{E}\|Y_n\|^2+(K+\lambda C)^2\Delta t^2\mathbb{E}\|Y_n\|^2+2(-\rho+\lambda C)\Delta t\mathbb{E}\|Y_n\|^2+\theta^2\Delta t\mathbb{E}\|Y_n\|^2\nonumber\\
 &+&\lambda C^2\Delta t\mathbb{E}\|Y_n\|^2\nonumber\\
 &=&\left[1-2\rho\Delta t+(K+\lambda C)^2\Delta t^2+2\lambda C\Delta t+\theta^2\Delta t+\lambda C^2\Delta t\right]\mathbb{E}\|Y_n\|^2.
 \end{eqnarray}
  If $$\Delta t \leq \dfrac{2\beta}{[2(K+\lambda C)+\overline{\beta}]\overline{\beta}}\wedge\dfrac{2\beta-\overline{\beta}}{2(K+\lambda c)\overline{\beta}},$$
  iterating the \eqref{fin} leads to
 \begin{eqnarray*}
 \mathbb{E}\|Y_n\|^2\leq\left[1-2\rho\Delta t+(K+\lambda C)^2\Delta t^2+2\lambda C\Delta t+\theta^2\Delta t+\lambda C^2\Delta t\right]^n\mathbb{E}\| X_0\|^2.
 \end{eqnarray*}
 The stability occurs if and only if  $\underset{n \rightarrow \infty}{\lim} \mathbb{E}\Vert Y_n \Vert^2= 0$, so we should  also have 
 \begin{eqnarray*}
 1-2\rho\Delta t+(K+\lambda C)^2\Delta t^2+2\lambda C\Delta t+\theta^2\Delta t+\lambda C^2\Delta t<1.
 \end{eqnarray*}
 That is 
 \begin{eqnarray}
 \Delta t<\dfrac{-[-2\rho+\theta^2+\lambda C(2+C)]}{(K+\lambda C)^2} =\dfrac{-\alpha_1}{(K+\lambda C)^2},
 \end{eqnarray}
 and there exists  a constant $\gamma=\gamma (\Delta t) >0$ such that
 \begin{eqnarray*}
 \mathbb{E}\|X_n\|^2\leq\left[1-2\rho\Delta t+(K+\lambda C)^2\Delta t^2+2\lambda C\Delta t+\theta^2\Delta t+\lambda C^2\Delta t\right]^n\mathbb{E}\| X_0\|^2\leq \mathbb{E}\|X_0\|^2 e^{-\gamma t_n}, 
 \end{eqnarray*}
 By the Taylor expansion, as 
 \begin{eqnarray*}
 \lefteqn {\ln(1-2\rho \Delta t+(K+\lambda C)^2\Delta t^2+2\lambda C\Delta t+\theta^2\Delta t+\lambda C^2\Delta t)}&& \\
 &=& -2\rho\Delta t+(K+\lambda C)^2\Delta t^2+2\lambda C\Delta t+\theta^2\Delta t+\lambda C^2\Delta t\ +...... ,
 \end{eqnarray*}
 we obviously have  $\underset{ \Delta t \rightarrow 0}{\lim}\gamma(\Delta t)=- (-2 \rho +\theta^2+\lambda C(2+C))=-\alpha_1$.
 \end{proof}
 
  To  analyse the nonlinear mean-square stability of the tamed Euler scheme (NCTS),  we make the following assumption. 
 
 \begin{Assumption}
 \label{ch5assumption3}
 There   positive constants $\beta$, $\overline{\beta}$,  $\theta$, $\mu$, $K$,  $\rho$, $C$  and $a>1$ such that :
 \begin{align}
 \langle x-y, f(x)-f(y)\rangle\leq &-\rho \|x-y\|^2-\beta \|x-y\|^{a+1},\nonumber\\
 \|f(x)\| \leq \overline{\beta}\| x \|^a+K\|x\|,\nonumber\\
 \|g(x)-g(y)\| \leq &\theta\|x-y\|,\qquad \,\,\|h(x)-h(y)\| \leq C\|x-y\|,\nonumber\\
  \langle x-y, h(x)-h(y)\rangle \leq &-\mu \|x-y\|^2.
  \label{assumparticular}
 \end{align}
 \end{Assumption}
 \begin{remark}
 \assref{ch5assumption3} is a consequence of \assref{ch5assumption2}, except \eqref{assumparticular}.
 \end{remark}
 
  Using \assref{ch5assumption3},  we can easily check  that the exact solution of \eqref{model} is  exponentiallly mean-square stable if
  $ \alpha_{2}:=  -2\rho +\theta^2+\lambda C(2 +C)<0$.
 \begin{theorem}
  Under \assref{ch5assumption3}, if $\alpha_3 :=K+\theta^2+\lambda C^2-2\lambda\mu C<0$ and $\overline{\beta}(1+2C)-2\beta<0$ for any stepsize
 \begin{eqnarray*}
\Delta t<\dfrac{-\alpha_3}{2K^2+\lambda^2C^2+2\lambda CK}\wedge\dfrac{\beta-C\overline{\beta}}{\overline{\beta}^2},
 \end{eqnarray*}
 there exists  a constant $\gamma=\gamma (\Delta t) >0$ such that
 \begin{eqnarray*}
 \mathbb{E}\|X_n\|^2\leq \mathbb{E}\|X_0\|^2 e^{-\gamma  t_n}, \,\, t_n=n\, \Delta t,  \underset{ \Delta t \rightarrow 0}{\lim}\gamma(\Delta t)=-\alpha_3.
 \end{eqnarray*}  
 and  the  numerical solution \eqref{ncts} is exponentiallly mean-square stable.
 \end{theorem}

\begin{proof}
 From equation \eqref{ncts}, we have 
 \begin{eqnarray}
 \|X_{n+1}\|^2&=&\|X_n\|^2+\dfrac{\Delta t^2 \|f(X_n)\|^2}{\left(1+\Delta t^{\alpha}\|f(X_n)\|\right)^2}+\|g(X_n)\Delta W_n\|^2+\|h(X_n)\Delta N_n\|^2\nonumber\\
 &+&2\left\langle X_n,\dfrac{\Delta tf(X_n)}{1+\Delta t^{\alpha} \|f(X_n)\|}\right\rangle+2\left\langle X_n+\dfrac{\Delta tf(X_n)}{1+\Delta t^{\alpha}\|f(X_n)\|}, g(X_n)\Delta W_n\right\rangle\nonumber\\
 &+&2\left\langle X_n+\dfrac{\Delta tf(X_n)}{1+\Delta t^{\alpha}\|f(X_n)\|}, h(X_n)\Delta N_n\right\rangle+2\langle g(X_n)\Delta W_n, h(X_n)\Delta N_n\rangle.
 \label{ch5meantamed1}
 \end{eqnarray}
 Using  \assref{ch5assumption3}, it follows that 
 \begin{eqnarray*}
 2\left\langle X_n, \dfrac{\Delta t f(X_n)}{1+\Delta t^{\alpha} \|f(X_n)\|}\right\rangle &\leq &\dfrac{-2\Delta t\rho\|X_n\|^2}{1+\Delta t^{\alpha}\|f(X_n)\|}-\dfrac{2\beta \Delta t \|X_n\|^{a+1}}{1+\Delta t ^{\alpha}\|f(X_n)\|}\\
 &\leq& -\dfrac{2\beta \Delta t\|X_n\|^{a+1}}{1+\Delta t^{\alpha}\|f(X_n)\|}.
 \label{ch5meantamed2} \\
 \|g(X_n)\Delta W_n\|^2 &\leq& \theta^2 \|X_n\|^2\|\Delta W_n\|^2\\
 \|h(X_n)\Delta N_n \|^2 &\leq& C^2\|X_n\|^2|\Delta N_n|^2.\label{ch5meantamed2a}\\
 2\langle X_n, h(X_n)\Delta N_n\rangle &=& 2\langle X_n, h(X_n)\rangle\Delta N_n \leq -2\mu\|X_n\|^2|\Delta N_n|,\\
 \label{ch5meantamed3}\\
 2\left\langle\dfrac{\Delta tf(X_n)}{1+\Delta t^{\alpha} \|f(X_n)\|}, h(X_n)\Delta N_n\right\rangle &\leq &\dfrac{2\Delta t\|f(X_n)\|\|h(X_n)\||\Delta N_n|}{1+\Delta t^{\alpha}\|f(X_n)\|}\nonumber\\
 &\leq &\dfrac{2\Delta tC\overline{\beta}\|X_n\|^{a+1}}{1+\Delta t^{\alpha}\|f(X_n)\|}|\Delta N_n|+2CK\Delta t\|X_n\|^2|\Delta N_n|
 \end{eqnarray*}

 So from \assref{ch5assumption3}, we have
 \begin{eqnarray}
\label{ch5assmeantamed} 
 2\left\langle X_n, \dfrac{\Delta tf(X_n)}{1+\Delta t^{\alpha} \|f(X_n)\|}\right\rangle&\leq& -\dfrac{2\beta \Delta t\|X_n\|^{a+1}}{1+\Delta t^{\alpha}\|f(X_n)\|} \nonumber\\
  \|g(X_n)\Delta W_n\|^2 &\leq& \theta^2 \|X_n\|^2\|\Delta W_n\|^2 \nonumber\\
 \|h(X_n)\Delta N_n\|^2 &\leq& C^2\|X_n\|^2|\Delta N_n|^2 \\
 2\langle X_n, h(X_n)\Delta N_n\rangle &\leq& -2\mu\|X_n\|^2|\Delta N_n|\nonumber\\
  2\left\langle\dfrac{\Delta tf(X_n)}{1+\Delta t^{\alpha} \|f(X_n)\|}, h(X_n)\Delta N_n\right\rangle 
  &\leq&\dfrac{2\Delta tC\overline{\beta}\|X_n\|^{a+1}}{1+\Delta t^{\alpha}\|f(X_n)\|}|\Delta N_n|+2CK \|X_n\|^2|\Delta N_n| \nonumber
 \end{eqnarray}
 Let $\Omega_n :=\{w\in\Omega : \|X_n(\omega)\|>1\}$,
 on $\Omega_n$,  using Assumption \ref{ch5assumption3}, we have
 \begin{eqnarray}
 \dfrac{\Delta t^2\|f(X_n)\|^2}{\left(1+\Delta t^{\alpha} \|f(X_n)\|\right)^2}&\leq& \dfrac{\Delta t\|f(X_n)\|}{1+\Delta t^{\alpha} \|f(X_n)\|}
 \leq \dfrac{\Delta t\overline{\beta}\|X_n\|^a}{1+\Delta t^{\alpha}\|f(X_n)\|}+K\Delta t\|X_n\|\nonumber\\
 &\leq& \dfrac{\Delta t\overline{\beta}\|X_n\|^{a+1}}{1+\Delta t^{\alpha}\|f(X_n)\|}+K\Delta t \|X_n\|^2.
 \label{ch5meantamed5}
 \end{eqnarray} 
 Therefore substituting \eqref{ch5assmeantamed} and \eqref{ch5meantamed5} in \eqref{ch5meantamed1} yields 
 \begin{eqnarray}
 \|X_{n+1}\|^2&\leq&\|X_n\|^2+K\Delta t\|X_n\|^2+\theta^2\|X_n\|^2\|\Delta W_n\|^2+C^2 \|X_n\|^2|\Delta N_n|^2\nonumber\\
 &+&2\left\langle X_n+\dfrac{\Delta t f(X_n)}{1+\Delta t^{\alpha}\|f(X_n)\|}, g(X_n)\Delta W_n\right\rangle \nonumber\\
 &-&2\mu\|X_n\|^2|\Delta N_n|+2CK\Delta t|\Delta N_n|\nonumber\\
 &+&2\langle g(X_n)\Delta W_n, h(X_n)\Delta N_n\rangle+\dfrac{\left[-2\beta\Delta t+\overline{\beta}\Delta t+2\overline{\beta}C\Delta t\right]\|X_n\|^{a+1}}{1+\Delta t^{\alpha}\|f(X_n)\|}.
 \label{ch5meantamed6}
 \end{eqnarray}
 Since  $\overline{\beta}(1+2C)-2\beta<0$, \eqref{ch5meantamed6} becomes 
 \begin{eqnarray}
 \|X_{n+1}\|^2&\leq&\|X_n\|^2+K\Delta t\|X_n\|^2+\theta^2\|X_n\|^2 \|\Delta W_n\|^2+C^2\|X_n\|^2|\Delta N_n|^2\nonumber\\
 &+&2\left\langle X_n+\dfrac{\Delta tf(X_n)}{1+\Delta t^{\alpha}\|f(X_n)\|}, g(X_n)\Delta W_n\right\rangle-2\mu\|X_n\|^2|\Delta N_n|\nonumber\\
 &+&2CK\Delta t|\Delta N_n|+2\langle g(X_n)\Delta W_n, h(X_n)\Delta N_n\rangle.
 \label{ch5meantamed7}
 \end{eqnarray}
 On $\Omega_n^c$, using Assumption \ref{ch5assumption3} and the inequality $(a+b)^2\leq 2a^2+2b^2$, we have 
 \begin{eqnarray}
 \dfrac{\Delta t^2 \|f(X_n)\|^2}{\left(1+\Delta t^{\alpha}\|f(X_n)\|\right)^2}&\leq& \dfrac{\Delta t^2\|f(X_n)\|^2}{1+\Delta t^{\alpha}\|f(X_n)\|}\nonumber\\ 
 &\leq& \dfrac{2\Delta t^2\overline{\beta}^2 \|X_n\|^{2a}}{1+\Delta t^{\alpha}\|f(X_n)\|}+2K^2\Delta t^2\|X_n\|^2\nonumber\\
 &\leq& \dfrac{2\Delta t^2\overline{\beta}^2\|X_n\|^{a+1}}{1+\Delta t^{\alpha}\|f(X_n)\|}+2K^2\Delta t^2\|X_n\|^2.
 \label{ch5meantamed8}
 \end{eqnarray}
 Therefore, using \eqref{ch5assmeantamed} and \eqref{ch5meantamed8}, \eqref{ch5meantamed1} becomes
 \begin{eqnarray}
 \|X_{n+1}\|^2&\leq&\|X_n \|^2+2K^2\Delta t^2 \| X_n\|^2+\theta^2 \|X_n\|^2\|\Delta W_n\|^2+C^2\|X_n\|^2|\Delta N_n|^2\nonumber\\
 &+&2\left\langle X_n+\dfrac{\Delta tf(X_n)}{1+\Delta t^{\alpha}\|f(X_n)\|}, g(X_n)\Delta W_n\right\rangle-2\mu\|X_n\|^2|\Delta N_n|\nonumber\\
 &+&2CK\Delta t|\Delta N_n|+2\langle g(X_n)\Delta W_n, h(X_n)\Delta N_n\rangle\nonumber\\
 &+&\dfrac{\left[2C\overline{\beta}\Delta t-2\beta\Delta t
 +2\overline{\beta}^2\Delta t^2\right] \|X_n\|^{a+1}}{1+\Delta t^{\alpha} \|f(X_n)\|} \nonumber.
 \label{ch5meantamed9}
 \end{eqnarray}
 Note that
 $$\Delta t<\dfrac{\beta-C\overline{\beta}}{\overline{\beta}^2}\Leftrightarrow 2C\overline{\beta}\Delta t-2\beta\Delta t+2\overline{\beta}^2\Delta t^2<0. $$
 Therefore 
 \eqref{ch5meantamed9} becomes 
 \begin{eqnarray}
 \|X_{n+1}\|^2&\leq&\|X_n\|^2+2K^2\Delta t^2\|X_n\|^2+\theta^2\|X_n\|^2\|\Delta W_n\|^2+C^2\|X_n\|^2|\Delta N_n|^2\nonumber\\
 &+&2\left\langle X_n+\dfrac{\Delta tf(X_n)}{1+\Delta t^{\alpha}\|f(X_n)\|}, g(X_n)\Delta W_n\right\rangle \nonumber\\
 &-&2\mu \|Y_n\|^2|\Delta N_n|+2CK\Delta t|\Delta N_n| +2\langle g(X_n)\Delta W_n, h(X_n)\Delta N_n\rangle
 \label{ch5meantamed10}
 \end{eqnarray}
 From the above discussion on $\Omega_n$ and $\Omega_n^c$, using the fact that  $\overline{\beta}(1+2C)-2\beta<0$  
 and  $\Delta t<\dfrac{\beta-C\overline{\beta}}{\overline{\beta}^2}$ that on $\Omega$, we have
 \begin{eqnarray}
 \|X_{n+1}\|^2&\leq&\|X_n\|^2+K\Delta t\|X_n\|^2+2K^2\Delta t^2\|X_n\|^2+\theta^2\|X_n\|^2\|\Delta W_n\|^2 \nonumber\\
 &+&C^2\|X_n\|^2|\Delta N_n|^2+2\left\langle X_n+\dfrac{\Delta tf(X_n)}{1+\Delta t^{\alpha}\|f(X_n)\|}, g(X_n)\Delta W_n\right\rangle\nonumber\\
 &-&2\mu\|X_n\|^2|\Delta N_n|+2CK\Delta t|\Delta N_n|+2\langle g(X_n)\Delta W_n, h(X_n)\Delta N_n\rangle.
 \label{ch5meantamed11}
 \end{eqnarray}
 
 Taking the expectation in both sides of \eqref{ch5meantamed11},  
 using independence of $\Delta W_n$ and $\Delta N_n$ together with the relation $\mathbb{E} \Delta W_n =0$, $\mathbb{E}\|\Delta W_n\|^2=\Delta t$, 
 $\mathbb{E}|\Delta N_n|=\lambda\Delta t$ and $\mathbb{E}|\Delta N_n|^2=\lambda^2\Delta t^2+\lambda\Delta t$ 
 leads to 
 \begin{eqnarray*}
 \mathbb{E}\|X_{n+1}\|^2&\leq& \mathbb{E}\|X_n\|^2+K\Delta t \mathbb{E}\|X_n\|^2+2K^2\Delta t^2\mathbb{E}\|X_n\|^2+\theta^2\Delta t\mathbb{E}\|X_n\|^2 \nonumber\\
 &+&\lambda^2C^2\Delta t^2\mathbb{E}\|X_n\|^2
 +\lambda C^2\Delta t\mathbb{E}\|X_n\|^2 \nonumber\\
 &-&2\mu\lambda\Delta t\mathbb{E}\|X_n\|^2
 +2\lambda CK\Delta t^2\mathbb{E}\|X_n\|^2 \nonumber\\
 &=&\left[1+(2K^2+\lambda^2C^2+2\lambda CK)\Delta t^2+(K+\theta^2+\lambda C^2-2\mu\lambda)\Delta t\right]\mathbb{E}\|X_n\|^2.
 \end{eqnarray*}
 Iterating the last inequality leads to 
 \begin{eqnarray*}
 \mathbb{E}\|X_n\|^2\leq \left[1+(2K^2+\lambda^2C^2+2\lambda CK)\Delta t^2+(K+\theta^2+\lambda C^2-2\mu\lambda)\Delta t\right]^n\mathbb{E}\|X_0\|^2.
 \end{eqnarray*}
 To have the stability of the NCTS scheme, we should also have
 \begin{eqnarray*}
 1+(2K^2+\lambda^2C^2+2\lambda CK)\Delta t^2+(K+\theta^2+\lambda C^2-2\mu\lambda)\Delta t<1.
 \end{eqnarray*}
 That is 
 \begin{eqnarray*}
 \Delta t<\dfrac{-[K+\theta^2+\lambda C^2-2\mu\lambda]}{2K^2+\lambda^2C^2+2\lambda CK},
 \end{eqnarray*}
 and there exists  a constant  $\gamma=\gamma (\Delta t) >0$ such that
 \begin{eqnarray*}
 \mathbb{E}\|X_n\|^2\leq \mathbb{E}\|X_0\|^2 e^{-\gamma \, t_n}, \,\, t_n= n\,\Delta t.
 \end{eqnarray*}
 As  in  the proof  of  \thmref{nt1}, we obviously have  $\underset{ \Delta t \rightarrow 0}{\lim}\gamma(\Delta t)=- (K +\theta^2+\lambda C^2-2\mu \lambda)=-\alpha_3$.
 \end{proof}
\section{Numerical simulations}
\label{simulations}
\subsection{Convergence}
 In this section, we present some numerical experiments to illustrate our theoretical strong convergence result.
  We consider the following stochastic differential equations
  {\small{
 \begin{eqnarray}
 \label{ex1}
 dX(t)=(-4X(t)-X^3(t))dt+X(t)dW(t)+X(t)dN_t,\\
 \label{ex2}
 dX(t)=(-4X(t)-X^3(t))dt+X(t)dW_1(t)+2X(t)dW_2(t)+X(t)dN_t,\\
 \label{ex3}
 dX(t)=(-4X(t)-X^3(t))dt+X(t)dW(t)+X(t)dN_1(t)-2X(t)dN_2(t),\\ 
 \label{ex4}
 dX(t)=(-4X(t)-X^3(t))dt+X(t)dW_1(t)+2X(t)dW_2(t),
+X(t)dN_1(t)-2X(t)dN_2(t).  
 \end{eqnarray}
 }}
 Note that $W$, $W_1$ and $W_2$ are independent Brownian motions and  $N$, $N_1$ and $N_2$ are independent Poisson  processes.
 Here $u(x)=-4x$. It is obvious to check that $u, v, g$ and $h$ satisfy \assref{ass1}. 
 Indeed  $ \langle x-y,f(x)-f(x) \rangle \leq  c (x-y)^2 $ for all  $c\geqslant 0$. 
 
 \figref{FIG01n} shows the strong errors for equation \eqref{ex1} with  different  values of  $\alpha$.
 As you can observe in Figure 1, all schemes have  strong convergence order 0.5, which confirm the theoretical result in \thmref{ch4theorem1}.
 
Remember that we have assumed scalar Poisson jump just for simplicity. 
In \figref{FIG02n},  the convergence of  our schemes with vector-valued jumps and vector-valued Brownian motions  is investigated.
 The errors graph corresponding to the equation \eqref{ex2} is given at 
 Figure \ref{FIG02an},  Figure \ref{FIG02bn} corresponds to the  equation \eqref{ex3},  while the errors graphs corresponding to the equation 
\eqref{ex4} are  given at Figure \ref{FIG02cn} and Figure \ref{FIG02dn}. In Figures \ref{FIG02an}, \ref{FIG02b} and \ref{FIG02an} the errors of the semi-tamed and the tamed Euler are identical.
All schemes have  strong convergence order 0.5, which also confirm the theoretical result in \thmref{ch4theorem1}.
\begin{figure}[h!!]
  \subfigure[]{
   \includegraphics[width=0.7\textwidth]{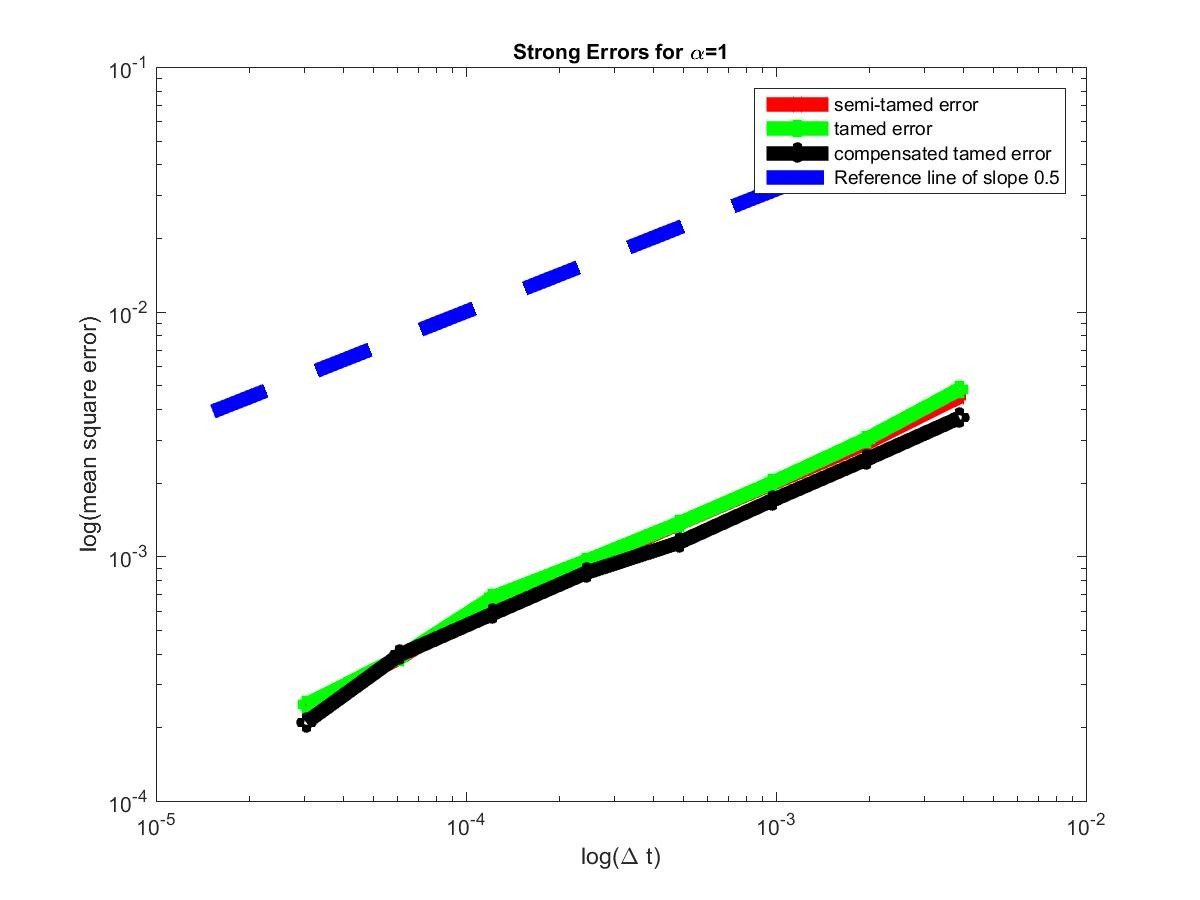}}
   \hskip 0.01\textwidth
   \subfigure[]{
   \includegraphics[width=0.7\textwidth]{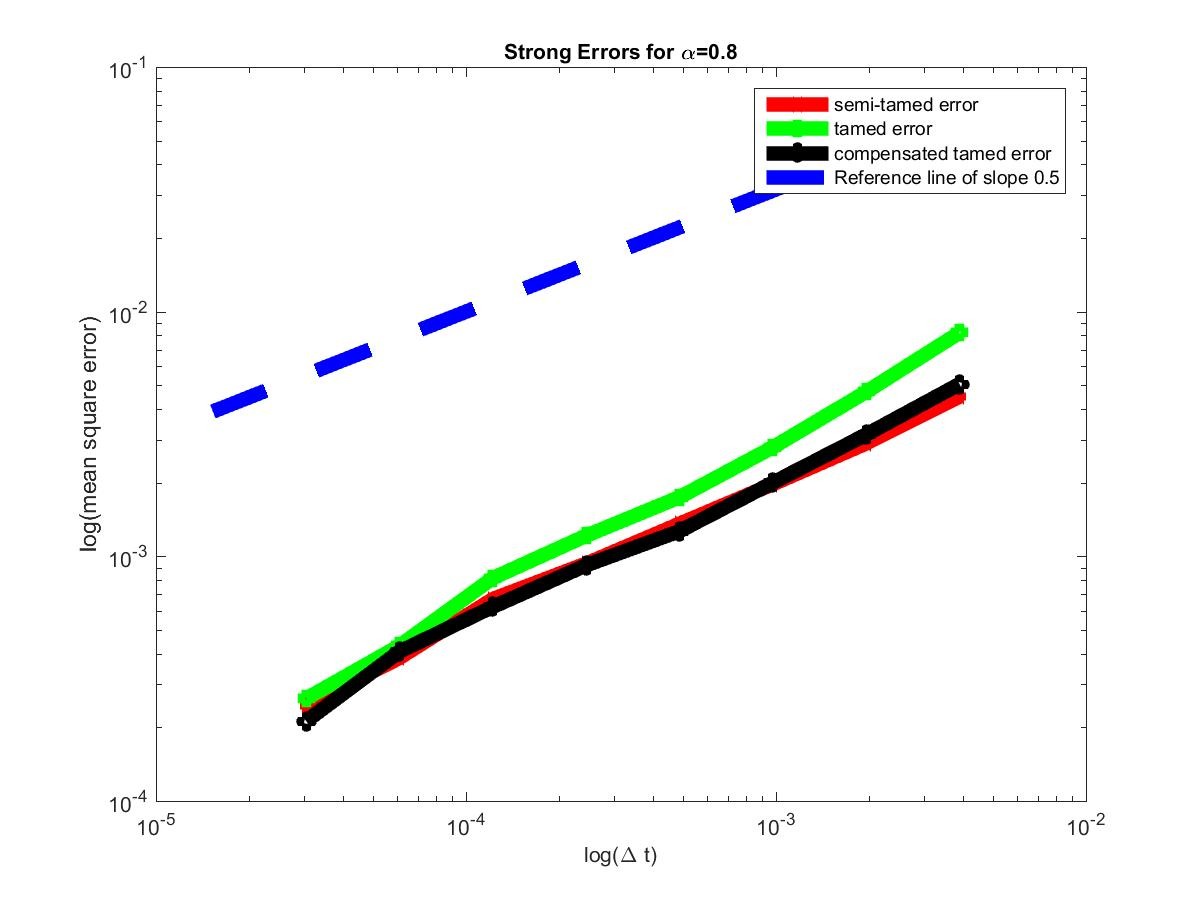}}
   \hskip 0.01\textwidth
\caption{ Strong convergence  of the compensated tamed scheme (CTS), the non compensated tamed scheme (NCTS) and the semi-tamed scheme (STS) 
for different values of $\alpha$ for SDEs \eqref{ex1}. For each value of $\alpha$ we use  $5000$ sample paths and the reference solutions are the numerical solutions with step size $\Delta t=2^{-16}$.
The initial solution is  $X_0=1$ and  the parameter of the scalar Poisson $\lambda=1$ and $T=1$.
Graph (a) corresponds to $\alpha=1$, graph (b) corresponds to $\alpha=0.8$.}
 \label{FIG01n}
\end{figure}

\begin{figure}[h!!]
\subfigure[]{
  \includegraphics[width=0.7\textwidth]{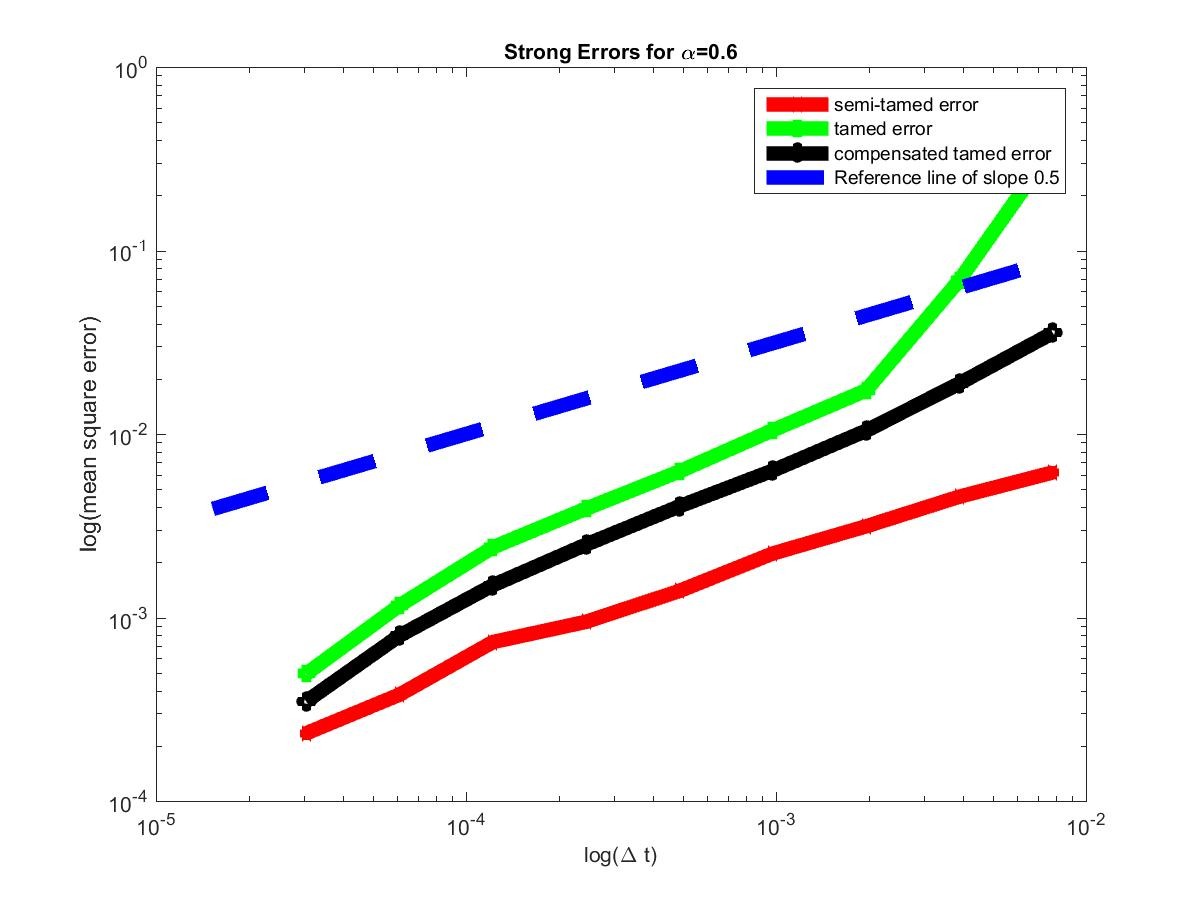}}
   \hskip 0.01\textwidth
\subfigure[]{
  \includegraphics[width=0.7\textwidth]{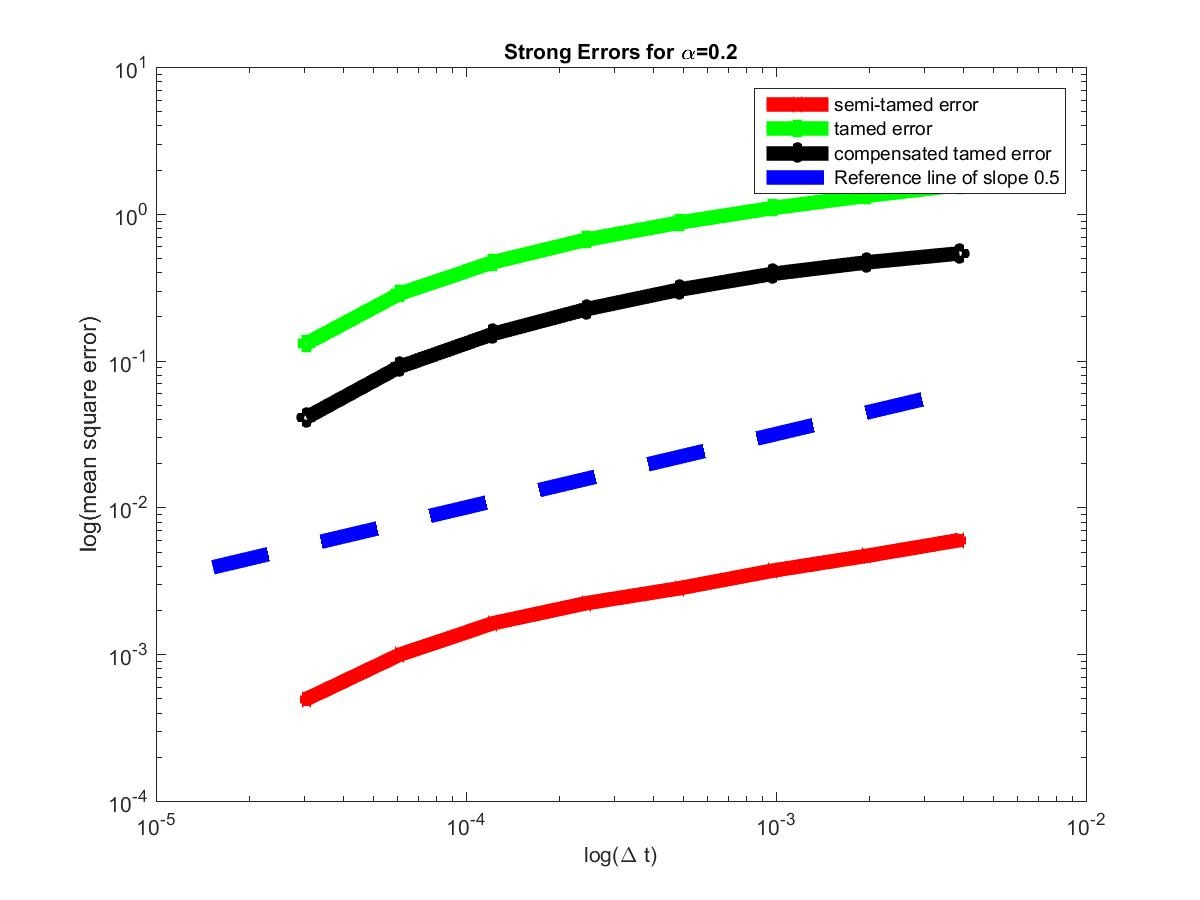}}
\caption{ Strong convergence  of the compensated tamed scheme (CTS), the non compensated tamed scheme (NCTS) and the semi-tamed scheme (STS) 
for different values of $\alpha$ for SDEs \eqref{ex1}. For each value of $\alpha$ we use  $5000$ sample paths and the reference solutions are the numerical solutions with step size $\Delta t=2^{-16}$.
The initial solution is  $X_0=1$ and  the parameter of the scalar Poisson $\lambda=1$ and $T=1$.
Graph (a) corresponds to $\alpha=1$, graph (b) corresponds to $\alpha=0.8$, graph (c) corresponds to $\alpha=0.6$ and  graph (d) corresponds to $\alpha=0.2$. }
 \label{FIG01n}
\end{figure}

 \begin{figure}[h!!]
  \subfigure[]{
\label{FIG02an}
   \includegraphics[width=0.7\textwidth]{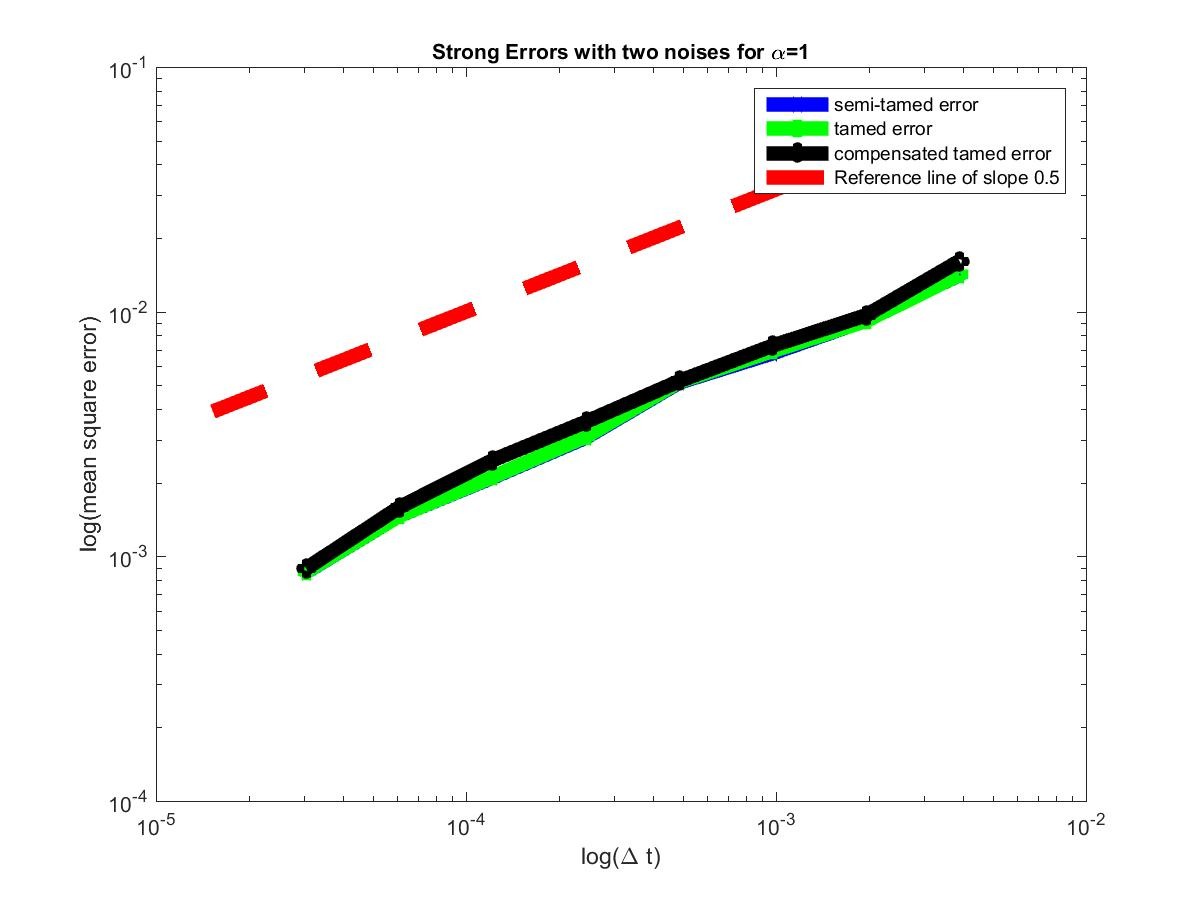}}
   \hskip 0.01\textwidth
   \subfigure[]{
   \label{FIG02bn}
   \includegraphics[width=0.7\textwidth]{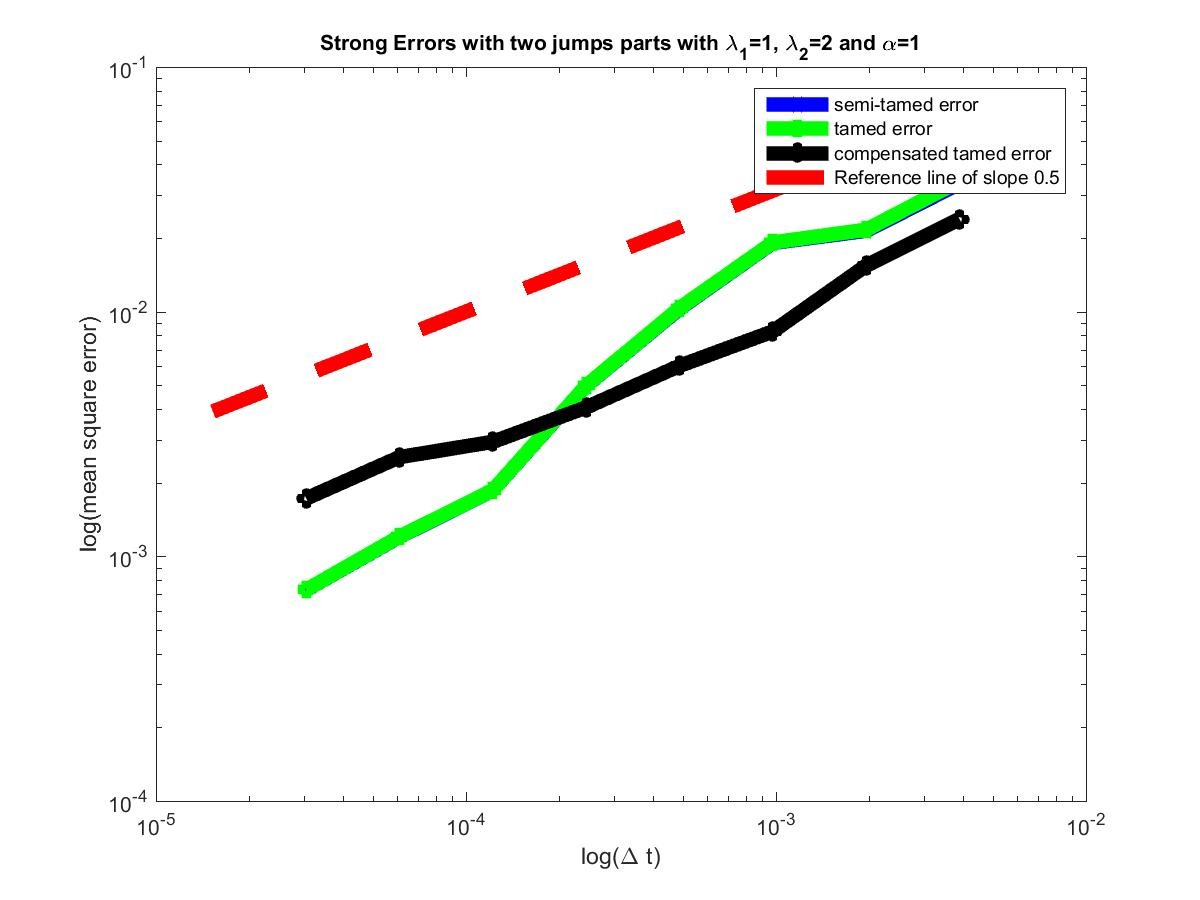}}
   \hskip 0.01\textwidth
\caption{ Strong convergence  of the compensated tamed scheme (CTS), the non compensated tamed scheme (NCTS) and the semi-tamed scheme (STS) 
for multiple noise terms and jumps terms. The initial solution is  $X_0=1$, $T=1$. We use $5000$ sample paths and the reference solutions are the numerical solutions with step size $\Delta t=2^{-16}$. 
Figure \ref{FIG02an} is for SDEs \eqref{ex2}, 
Figure \ref{FIG02bn} is for SDEs \eqref{ex3}.
Graphs (a)  corresponds to $\alpha=1$, $\lambda_1=1$ and $\lambda_2=2$, while  graph (b) corresponds to $\alpha=0.2$, $\lambda_1=1$ and $\lambda_2=2$.
}
 \label{FIG02n}
\end{figure}

\begin{figure}[h!!]
\subfigure[]{
   \label{FIG02cn}
  \includegraphics[width=0.7\textwidth]{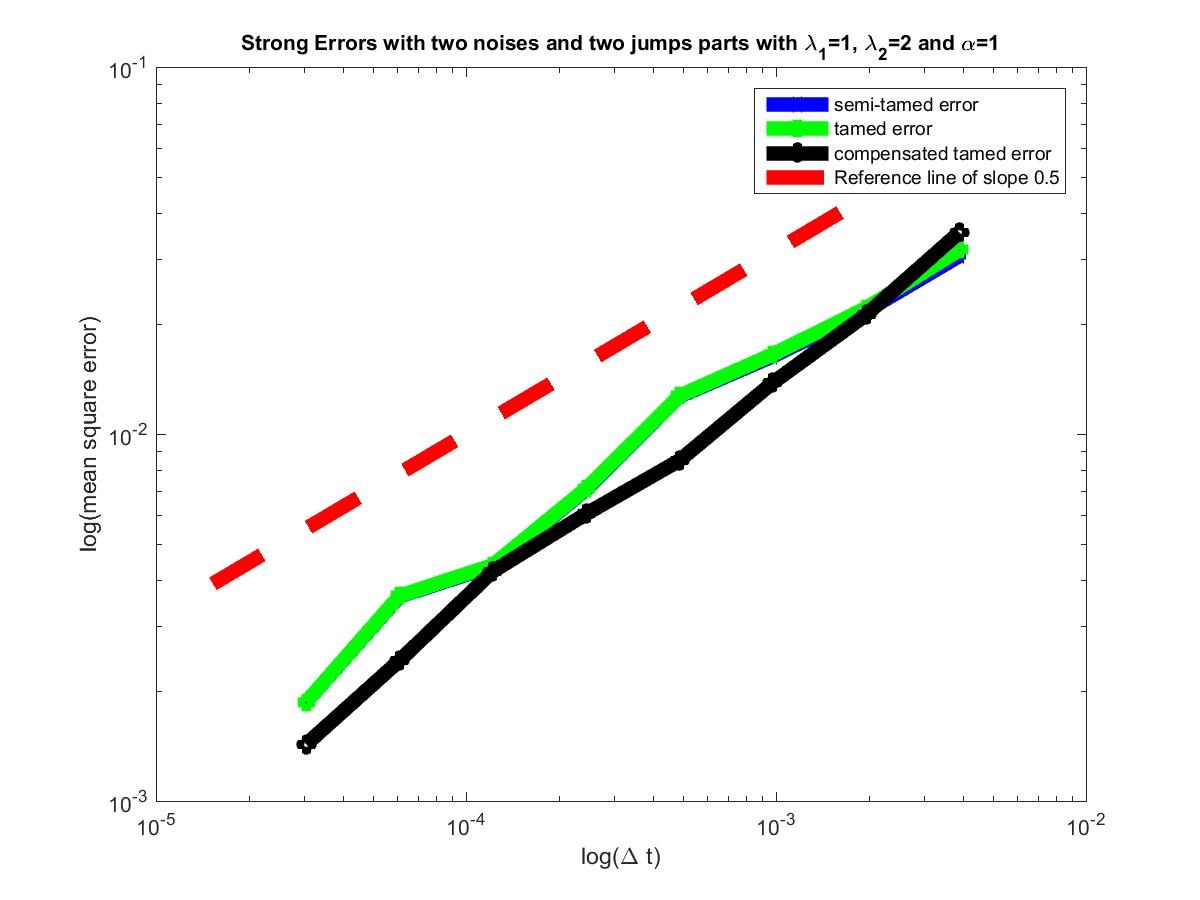}}
   \hskip 0.01\textwidth
\subfigure[]{
   \label{FIG02dn}
  \includegraphics[width=0.7\textwidth]{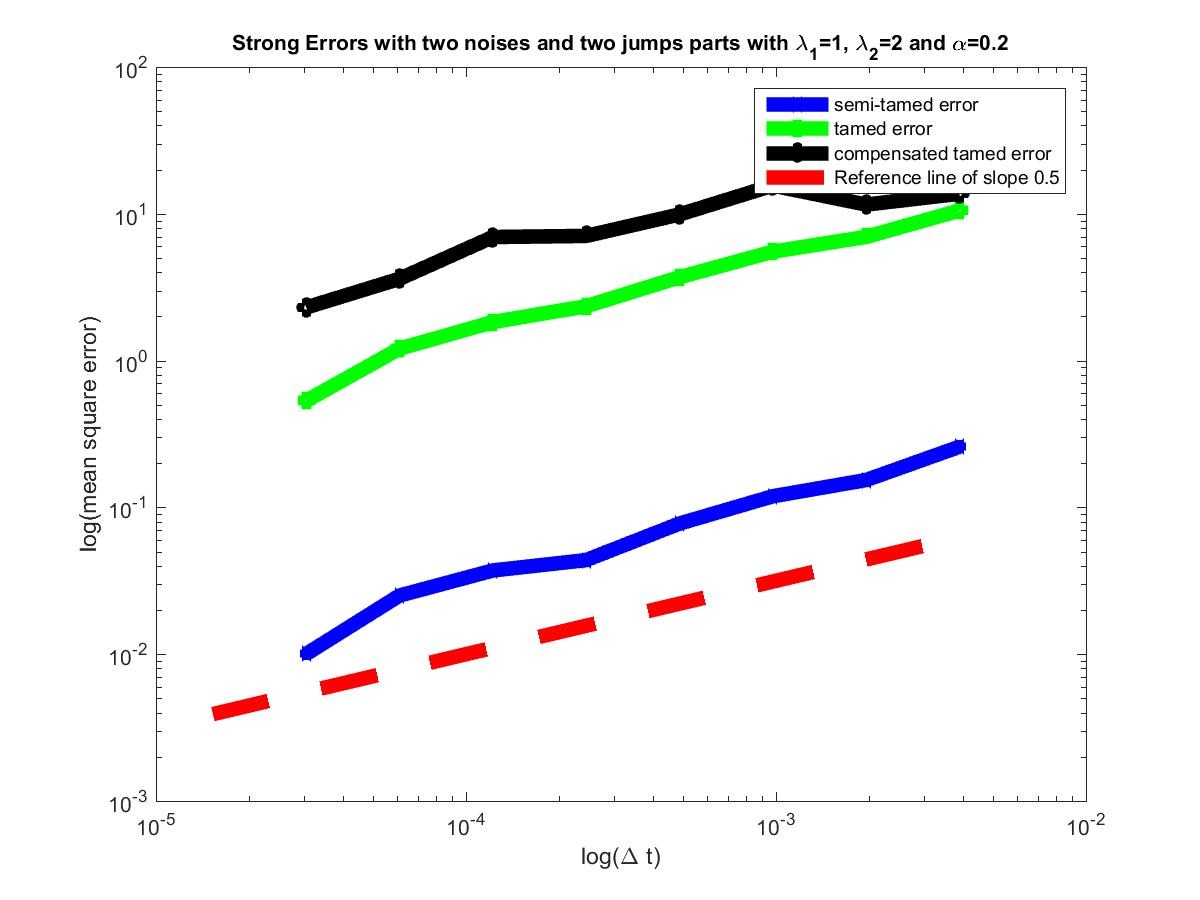}}
\caption{ Strong convergence  of the compensated tamed scheme (CTS), the non compensated tamed scheme (NCTS) and the semi-tamed scheme (STS) 
for multiple noise terms and jumps terms. The initial solution is  $X_0=1$, $T=1$. We use $5000$ sample paths and the reference solutions are the numerical solutions with step size $\Delta t=2^{-16}$. 
 Figures \ref{FIG02cn} and \ref{FIG02dn} are for SDEs \eqref{ex4}.
Graphs (a)  corresponds to $\alpha=1$, $\lambda_1=1$ and $\lambda_2=2$.
}
 \label{FIG02n}
\end{figure}
 
 \subsection{Linear stability}
 The goal of this section is to provide some practical examples to sustain our theoretical results in  the previous section.
 We compare  the stability behaviors of the tamed Euler and the compensated tamed Euler schemes with the one  of  semi-tamed Euler scheme. 
 We denote by $Y_n$ all the approximated solutions from  those schemes.
 Here we consider the following linear stochastic differential equation 
 \begin{eqnarray}
 \left\{\begin{array}{ll}
 dX(t)=aX(t)dt+bX(t)dW(t)+cX(t)dN(t),\\
 X(0)=1,
 \end{array}
 \right.
 \end{eqnarray}
 with the following two parameters 
 \begin{itemize}
 \item Case I. $a=-1$, $b=2$, $c=-0.9$ and $\lambda =9$.
 \item Case II. $a=2$, $b=2$, $c=-0.9$ and $\lambda=9$.
 \end{itemize}
 We can easily check that in both cases $l<0$, which ensure the linear mean-square stable of the exact solution in the two situations. We can also easily check from the theoretical result  that  the semi-tamed and the tamed Euler scheme reproduce the linear  mean-square property of the exact solution in the first case for all $\Delta t<0.048$ and in the second case for all $\Delta t<0.0245$. In \figref{FIG01}, we illustrate the mean-square stability of the tamed Euler, the compensated tamed Euler and the semi-tamed Euler schemes for different stepsizes. We observe from   \figref{FIG01} that the semi-tamed  scheme works better than the tamed and compensated tamed schemes. We also observe that when $\alpha$ approaches $1$, the tamed and the compensated schemes are more stable ( see Figure \ref{FIG04} ).
 
\subsection{Nonlinear stability}
For nonlinear stability,  we consider the following nonlinear stochastic differential equation
\begin{eqnarray}
\label{linearexple}
\left\{\begin{array}{ll}
dX(t)=\left(-2X(t)-X(t)^3 \right)dt+\sqrt{2}X(t)dW(t)-\dfrac{1}{4}X(t)dN(t),\\
 X(0)=1.
 \end{array}
 \right.
\end{eqnarray}
The Poisson process  intensity is $\lambda=1$, $f(x)=-2x-x^3, \; g(x)=\sqrt{2}x$, $h(x)=-\dfrac{1}{4}x$ and $T=2$.
We take  $u(x)=-2x$ and $v(x)=-x^3$. Indeed, we obviously have 
\begin{eqnarray*}
\langle x-y, f(x)-f(y)\rangle \leq -2(x-y)^2\\
|g(x)-g(y)|^2\leq 2(x-y)^2, \,\;\;
|h(x)-h(y)|^2\leq \dfrac{1}{16}(x-y)^2.
\end{eqnarray*}
Then  $\mu=-2$, $\sigma=2$, $\gamma=\dfrac{1}{16}$ and  $\alpha=2\mu+\sigma+\lambda\sqrt{\gamma}(\sqrt{\gamma}+2)=-\dfrac{23}{16}<0$. 
It follows  that the exact solution is exponentially mean-square stable. One can easily check from  theoretical 
results that for $\Delta t<0.22$, the semi-tamed  Euler schemes reproduces the  exponential mean-square stability property of the exact solution. 
 \figref{FIG03} illustrates the stability of the tamed  scheme, compensated tamed  scheme and the semi-tamed scheme for different step-sizes.
 We take $\Delta t =1/6$, $\Delta t= 1/12$ and $ \Delta t=1/24$ and  generate $7\times 10^3$ samples for each numerical method. We  observe that the semi-tamed scheme works
 better than the tamed and the compensated tamed Euler schemes. We observe also that when $\alpha$ approaches $1$ the tamed and compensated tamed Euler scheme are more stable.
 
\begin{figure}
\begin{center}
  \subfigure[]{
\label{FIG01a}
   \includegraphics[width=0.48\textwidth]{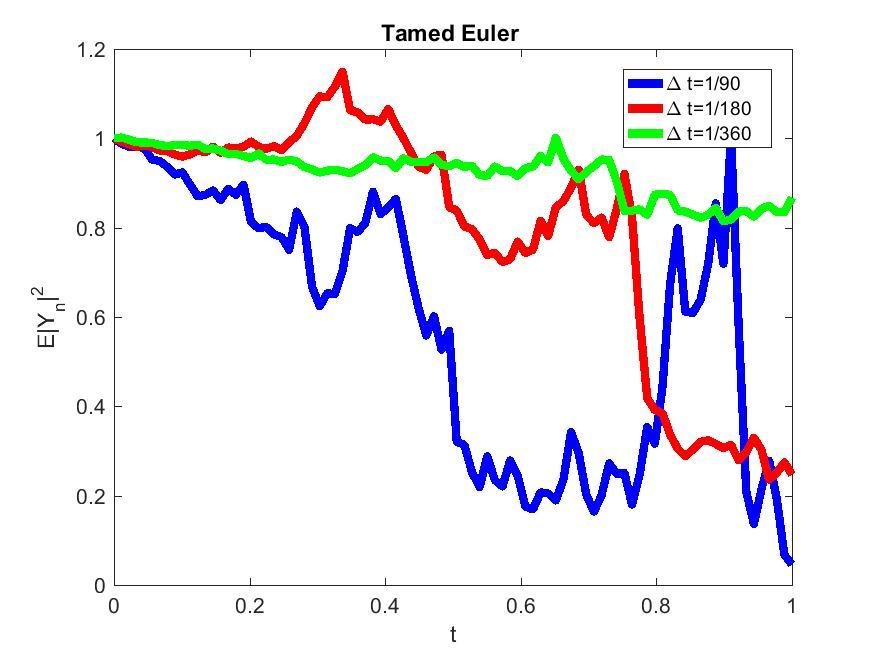}}
   \hskip 0.01\textwidth
   \subfigure[]{
   \label{FIG01b}
   \includegraphics[width=0.48\textwidth]{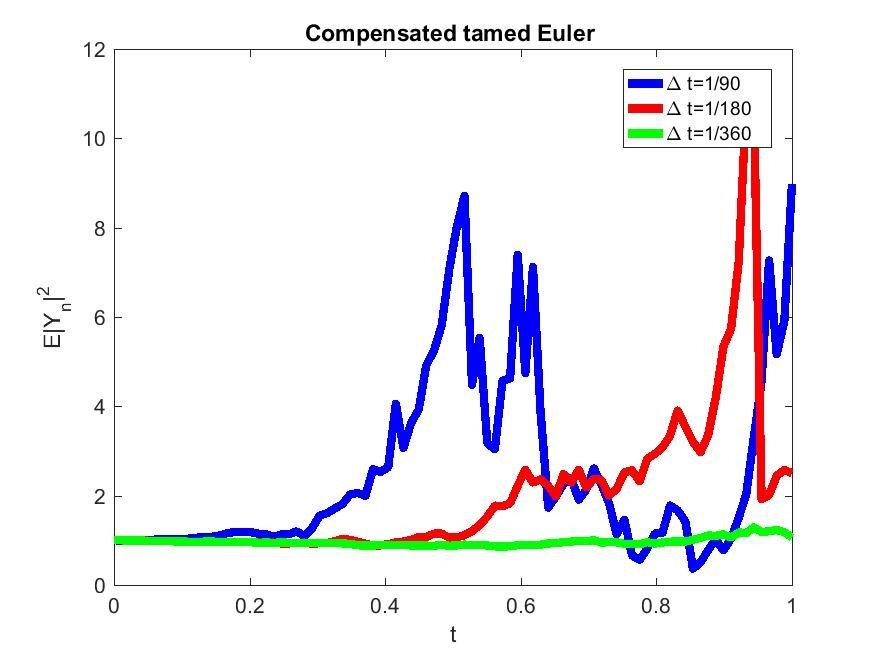}}
   \hskip 0.01\textwidth
\subfigure[]{
   \label{FIG01c}
   \includegraphics[width=0.66\textwidth]{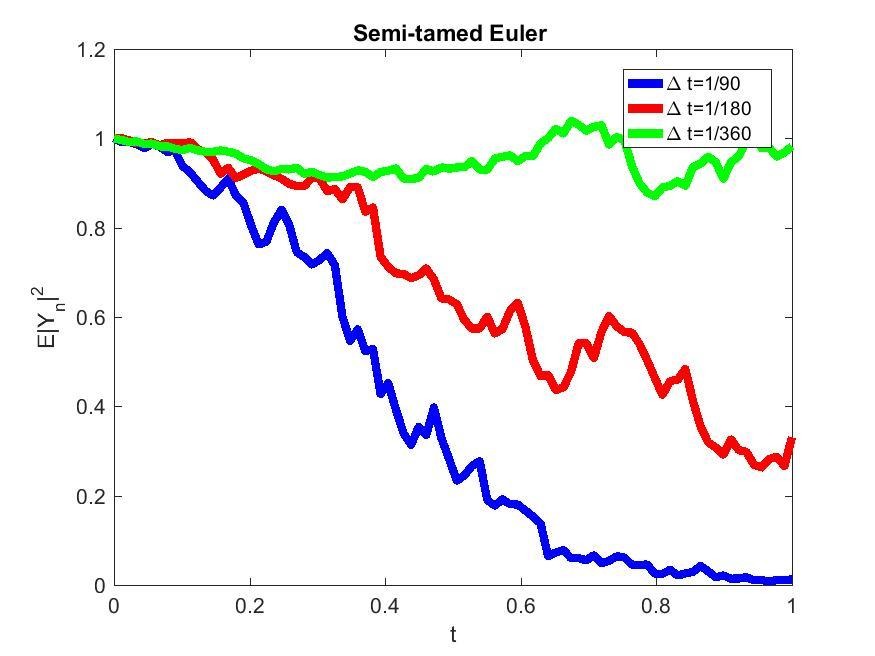}}
\caption{ Linear stability of  with $\alpha=1$ with different stepsizes  for SDE \eqref{linearexple} with  Case II (a) Tamed Euler scheme,  (b) Compensated tamed Euler scheme (c) Semi-tamed Euler scheme. This reveals that the semi-tamed Euler scheme works better than the tamed Euler and the compensated tamed Euler schemes }
 \label{FIG01}
 \end{center}
 \end{figure}
 \clearpage
\begin{figure}
 \begin{center}
  \subfigure[]{
\label{FIG02a}
   \includegraphics[width=0.48\textwidth]{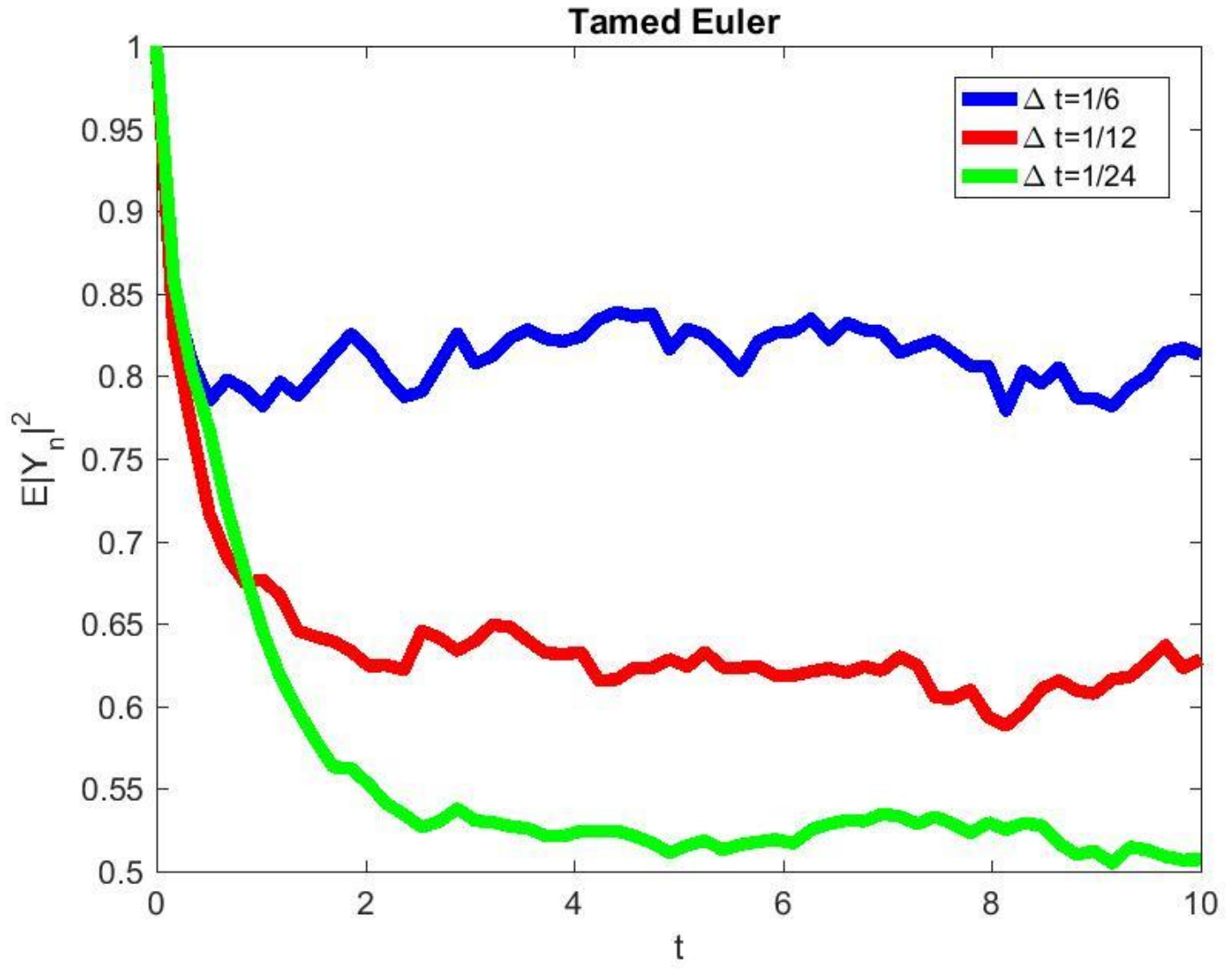}}
   \hskip 0.01\textwidth
   \subfigure[]{
   \label{FIG02b}
   \includegraphics[width=0.48\textwidth]{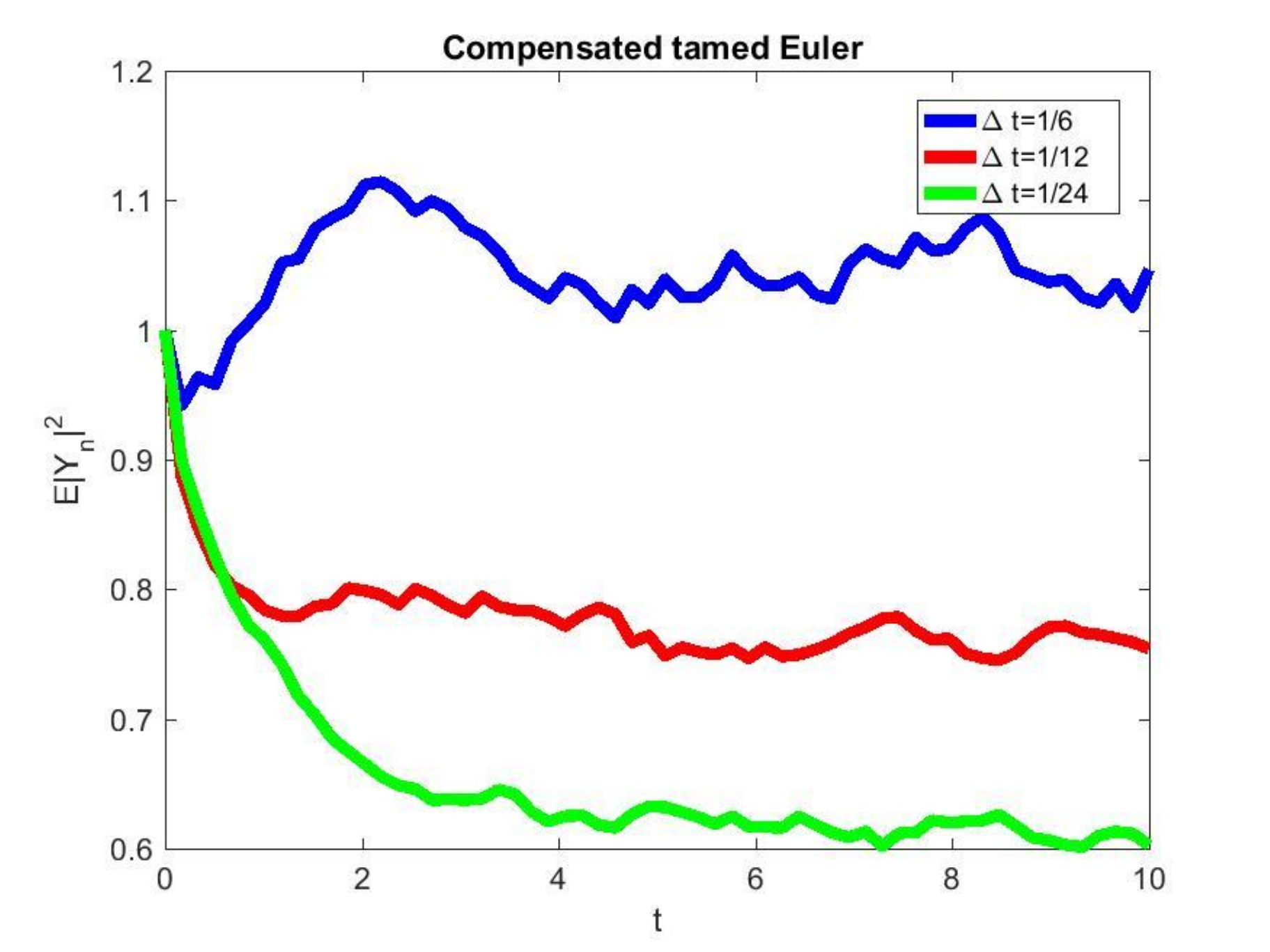}}
   \hskip 0.01\textwidth
\subfigure[]{
   \label{FIG02c}
   \includegraphics[width=0.66\textwidth]{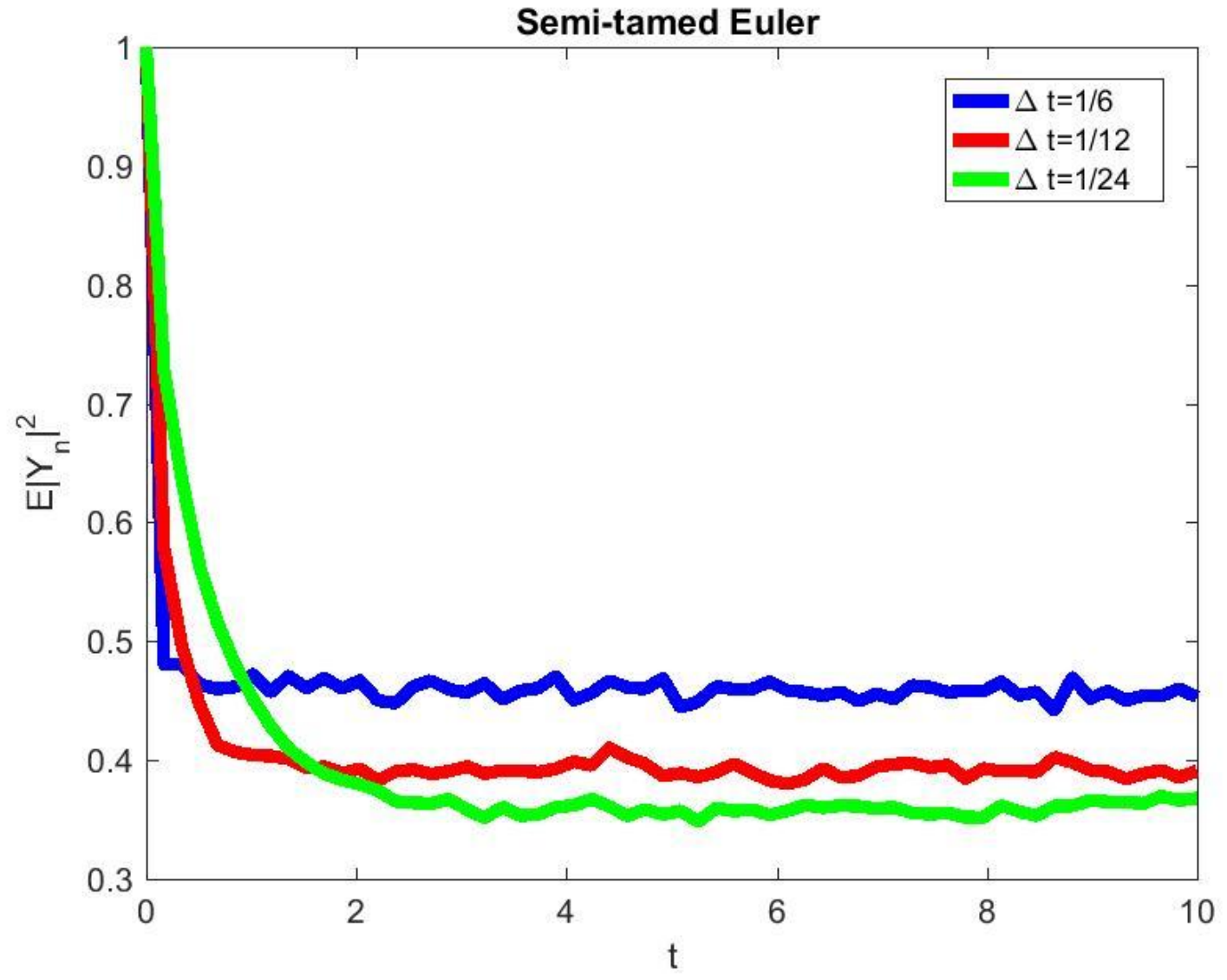}}
\caption{ Nonlinear stability with $\alpha=0.5$ with different stepsizes, (a) Tamed Euler scheme, (b) Semi-tamed Euler scheme, (c) Compensated tamed Euler scheme (d) for  $7\times 10^3$ samples of each numerical scheme. This illustrate that semi-taned Euler scheme works better than
the tamed and the compensated tamed Euler schemes.}
 \label{FIG02}
 \end{center}
\end{figure}
%
\begin{figure}
 \begin{center}
  \subfigure[]{
\label{FIG03a}
   \includegraphics[width=0.48\textwidth]{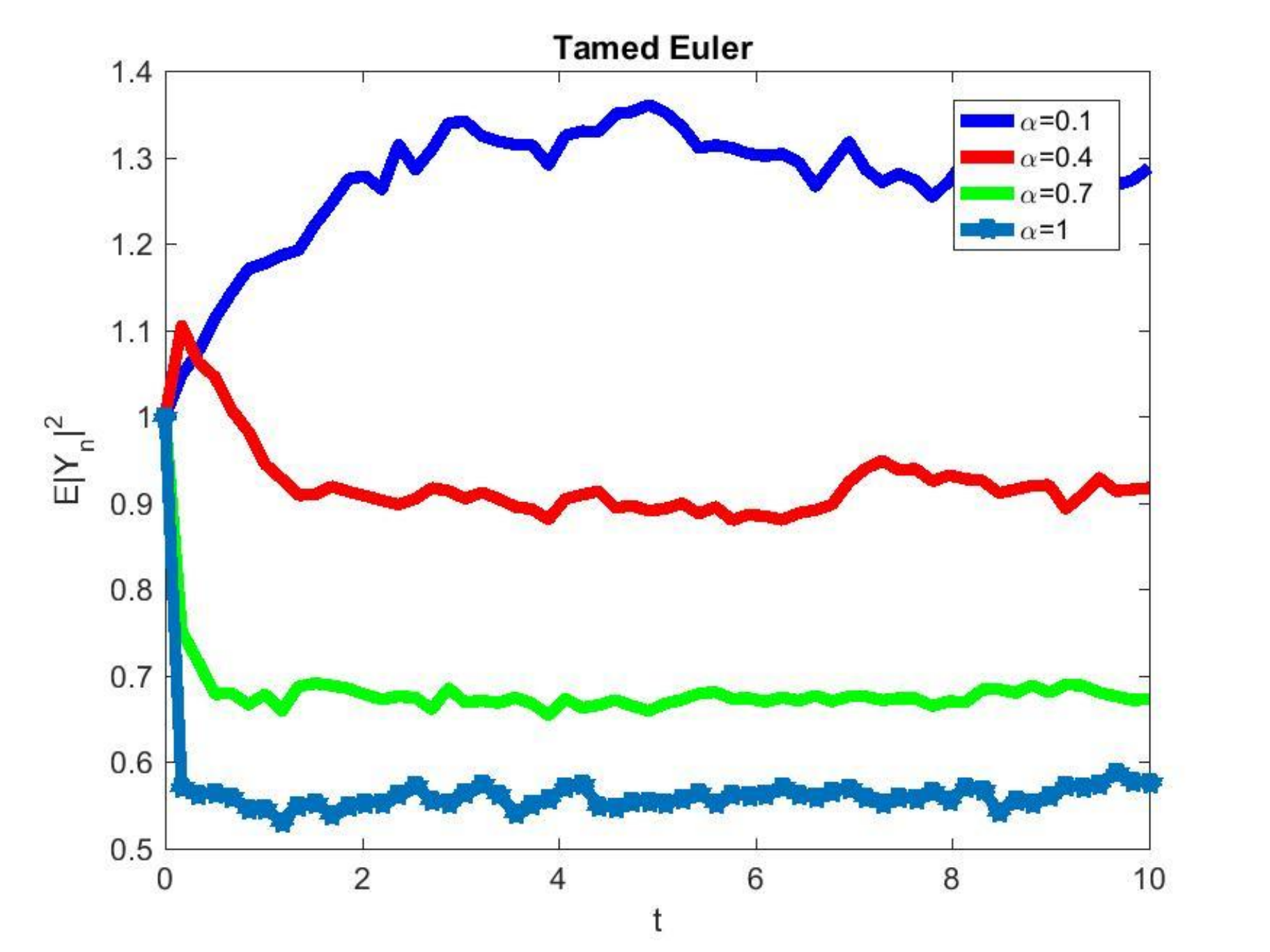}}
   \hskip 0.01\textwidth
   \subfigure[]{
   \label{FIG03b}
   \includegraphics[width=0.48\textwidth]{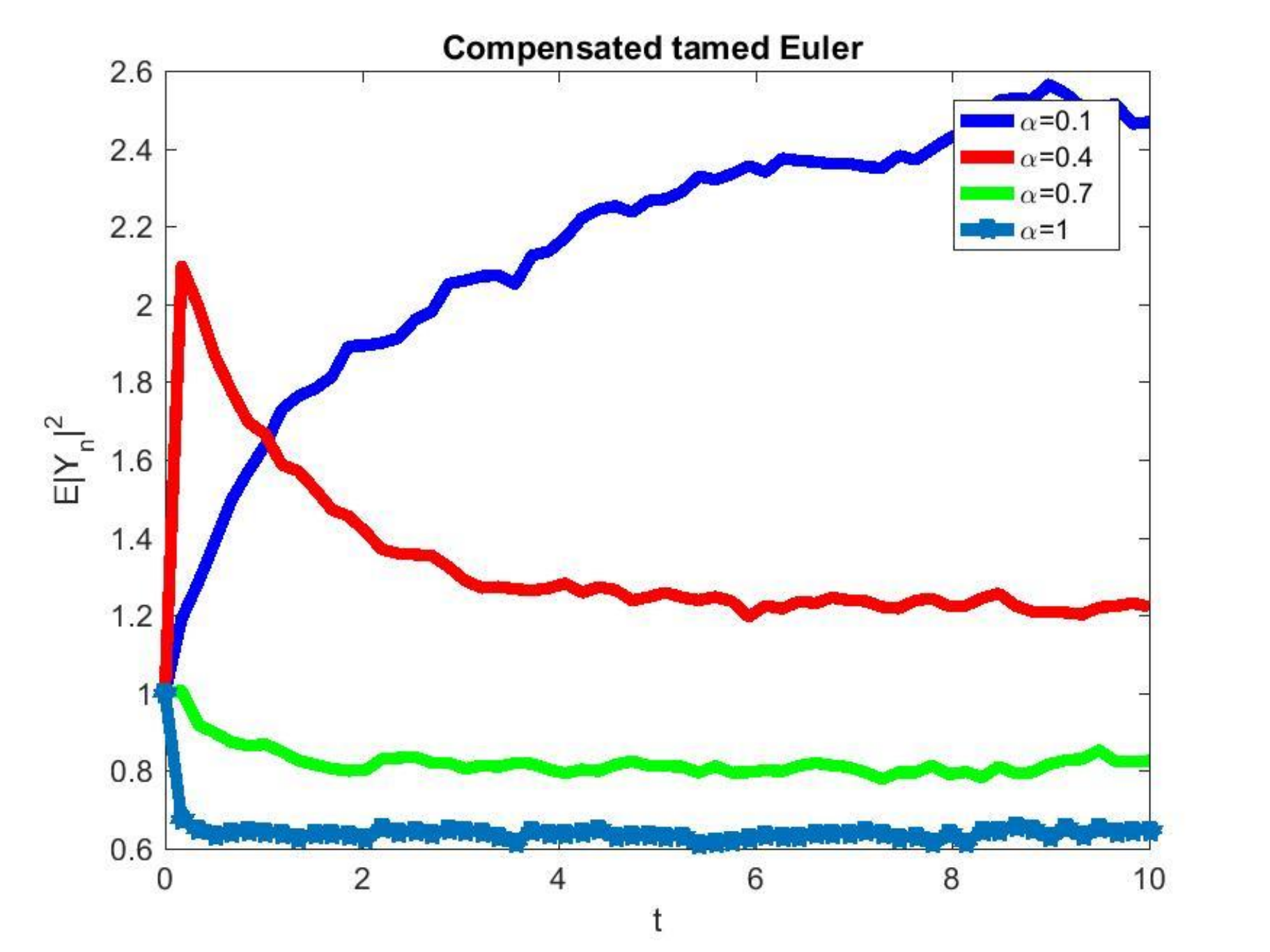}}
   \hskip 0.01\textwidth
\subfigure[]{
   \label{FIG03c}
   \includegraphics[width=0.66\textwidth]{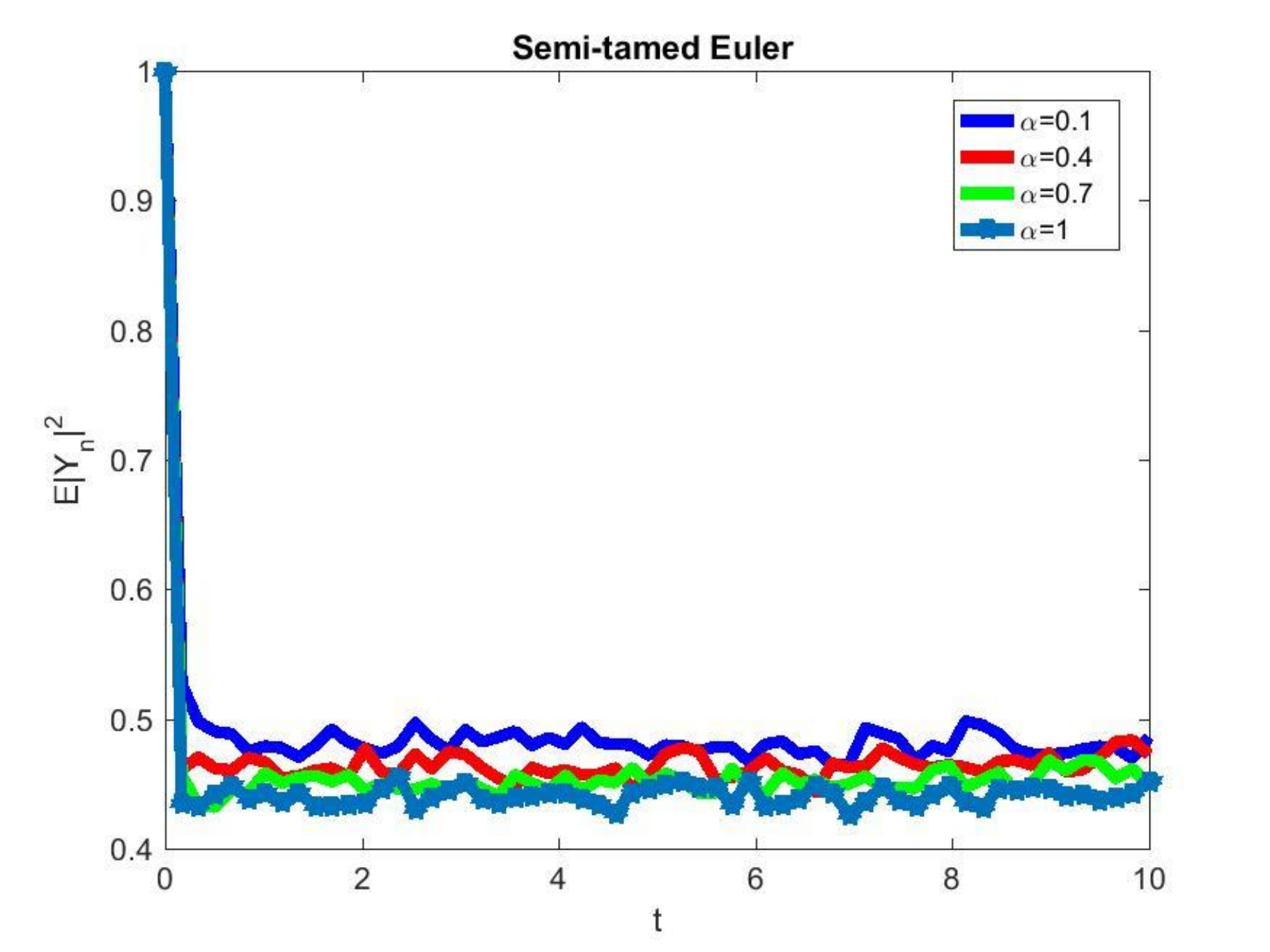}}
\caption{ Nonlinear stability with different values of $\alpha$ for with $\Delta t=1/6$ with $7\times 10^3$ samples paths. (a) Tamed Euler scheme, (b) Compensated tamed Euler scheme, (c) Semi-tamed Euler scheme (d). This reveals that when $\alpha$ approaches $1$ the tamed Euler and the compensated tamed Euler schemes are more stable and behave like the semi-tamed Euler scheme.}
 \label{FIG03}
 \end{center}
\end{figure}
\clearpage

\begin{figure}
 \begin{center}
  \subfigure[]{
\label{FIG04a}
   \includegraphics[width=0.7\textwidth]{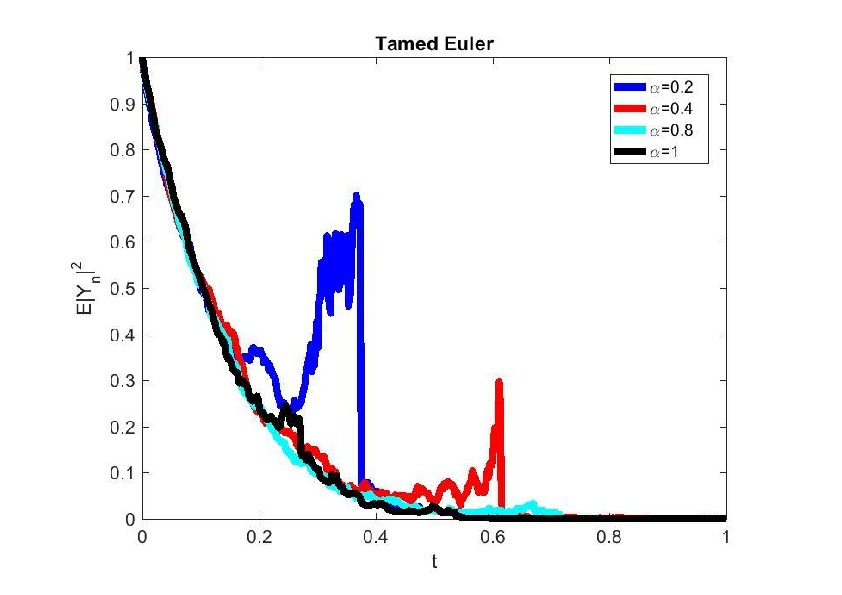}}
   \hskip 0.01\textwidth
   \subfigure[]{
   \label{FIG04b}
   \includegraphics[width=0.7\textwidth]{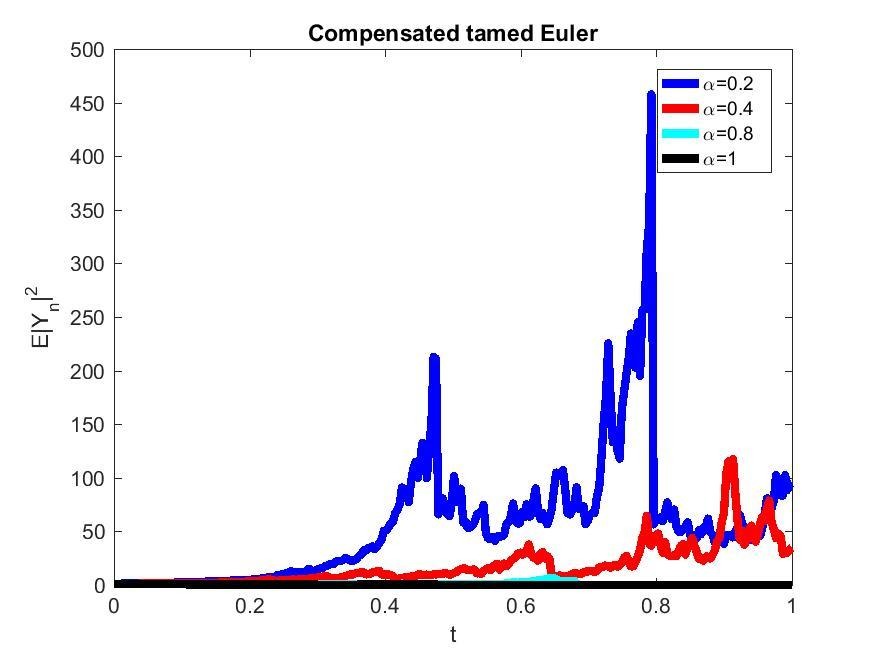}}
\caption{ Linear stability with different values of $\alpha$ for SDE \eqref{linearexple} with  Case I with $7\times 10^3$ samples. (a) Tamed Euler scheme with $\Delta t=0.0033$ (b) Compensated tamed Euler scheme with $\Delta t=0.002$. This illustrate that the Tamed and the compensated tamed Euler schemes have good stability behaviors when $\alpha$ approaches $1$.}
 \label{FIG04}
 \end{center}
\end{figure}
\clearpage

\section*{Acknowledgements}
This project was supported by the Robert Bosch
Stiftung through the AIMS ARETE chair programme.

\end{document}